\pdfoutput=1
\RequirePackage{ifpdf}
\ifpdf 
\documentclass[pdftex]{sigma}
\else
\documentclass{sigma}
\fi

\numberwithin{equation}{section}

\newtheorem{Theorem}{Theorem}[section]
\newtheorem{Corollary}[Theorem]{Corollary}
\newtheorem{Lemma}[Theorem]{Lemma}
\newtheorem{Proposition}[Theorem]{Proposition}
 { \theoremstyle{definition}
\newtheorem{Definition}[Theorem]{Definition}

\newtheorem{Remark}[Theorem]{Remark} }

\usepackage{bm,upgreek}
\usepackage{tensor}
\usepackage[all]{xy}

\newcommand{\parderv}[2] {\frac{\partial#1}{\partial#2}}
\newcommand{\mbf} [1]{\mathbf{#1}}
\newcommand{\mbb} [1]{\mathbb{#1}}
\newcommand{\mc} [1]{\mathcal{#1}}
\newcommand{\sk}{\scriptscriptstyle}
\newcommand{\Drm} {{\mathrm{D}}}
\newcommand{\Cl} {{\mathcal{C}\ell}}
\newcommand{\dd} {\mathrm{d}}
\newcommand{\ii} {\mathrm{i}}
\newcommand{\ee} {\mathrm{e}}
\newcommand{\mr}[1] {\mathring{#1}}
\newcommand{\ind} {\indices}
\newcommand{\Rho} {\mathrm{P}}
\newcommand{\End} {\mathrm{End}}
\newcommand{\Spin} {\mathrm{Spin}}
\newcommand{\SO} {\mathrm{SO}}
\newcommand{\Tgt} {\mathrm{T}}
\newcommand{\Nrm} {\mathrm{N}}
\newcommand{\so} {\mathfrak{so}}
\newcommand{\glie} {\mathfrak{gl}}
\newcommand{\g} {\mathfrak{g}}
\newcommand{\mfp} {\mathfrak{p}}
\newcommand{\mfq} {\mathfrak{q}}
\newcommand{\mfr} {\mathfrak{r}}
\newcommand{\mfz} {\mathfrak{z}}
\newcommand{\PT} {\mathbb{PT}}
\newcommand{\T} {\mathbb{T}}
\newcommand{\Ss} {\mathbb{S}}
\newcommand{\Pp} {\mathbb{P}}
\newcommand{\E} {\mathbb{E}}
\newcommand{\M} {\mathbb{M}}
\newcommand{\CP} {\mathbb{CP}}
\newcommand{\C} {\mathbb{C}}
\newcommand{\F} {\mathbb{F}}
\newcommand{\V} {\mathbb{V}}

\begin{document}


\newcommand{\arXivNumber}{1505.06938}

\renewcommand{\PaperNumber}{005}

\FirstPageHeading

\ShortArticleName{Twistor Geometry of Null Foliations in Complex Euclidean Space}

\ArticleName{Twistor Geometry of Null Foliations\\ in Complex Euclidean Space}

\Author{Arman TAGHAVI-CHABERT}
\AuthorNameForHeading{A.~Taghavi-Chabert}
\Address{Universit\`{a} di Torino, Dipartimento di Matematica ``G.~Peano'',\\ Via Carlo Alberto, 10 - 10123, Torino, Italy}
\Email{\href{mailto:ataghavi@unito.it}{ataghavi@unito.it}}

\ArticleDates{Received April 01, 2016, in f\/inal form January 14, 2017; Published online January 23, 2017}

\Abstract{We give a detailed account of the geometric correspondence between a smooth complex projective quadric hypersurface $\mathcal{Q}^n$ of dimension $n \geq 3$, and its twistor space $\mathbb{PT}$, def\/ined to be the space of all linear subspaces of maximal dimension of $\mathcal{Q}^n$. Viewing complex Euclidean space $\mathbb{CE}^n$ as a dense open subset of $\mathcal{Q}^n$, we show how local foliations tangent to certain integrable holomorphic totally null distributions of maximal rank on $\mathbb{CE}^n$ can be constructed in terms of complex submanifolds of $\mathbb{PT}$. The construction is illustrated by means of two examples, one involving conformal Killing spinors, the other, conformal Killing--Yano $2$-forms. We focus on the odd-dimensional case, and we treat the even-dimensional case only tangentially for comparison.}


\Classification{32L25; 53C28; 53C12}

\section{Introduction}
The \emph{twistor space $\PT$} of a smooth complex projective quadric hypersurface $\mc{Q}^n$ of dimension $n=2m+1 \geq 3$, is def\/ined to be the space of all \emph{$\gamma$-planes}, i.e., $m$-dimensional linear subspaces of~$\mc{Q}^n$. This is a complex projective variety of dimension $\frac{1}{2}(m+1)(m+2)$ equipped with a canonical holomorphic distribution~$\mathrm{D}$ of rank~$m+1$, and maximally non-integrable, i.e., $\Tgt \PT = [ \mathrm{D} , \mathrm{D} ] + \mathrm{D}$. Here, $\Tgt \PT$ denotes the holomorphic tangent bundle of~$\PT$. Noting that a smooth quadric can be identif\/ied with a complexif\/ied $n$-sphere and is naturally equipped with a holomorphic conformal structure, we shall view complex Euclidean space $\C \E^n$ as a dense open subset of $\mc{Q}^n$. In this context, we shall prove the following new results holding locally:
\begin{itemize}\itemsep=0pt
\item totally geodetic integrable holomorphic $\gamma$-plane distributions on $\C \E^n$ arise from $(m+1)$-dimensional complex submanifolds of~$\PT$~-- Theorem~\ref{thm-odd-Kerr-theorem-I};
\item totally geodetic integrable holomorphic $\gamma$-plane distributions on $\C \E^n$ with integrable orthogonal complements arise from $(m+1)$-dimensional complex submanifolds of $\PT$ foliated by holomorphic curves tangent to~$\Drm$~-- Theorem~\ref{thm-odd-Kerr-theorem-II};
\item totally geodetic integrable holomorphic $\gamma$-plane distributions on $\C \E^n$ with totally geodetic integrable orthogonal complements arise from $m$-dimensional complex submanifolds of a~$1$-dimensional reduction of a subset of $\PT$ known as \emph{mini-twistor space $\M\T$}~-- Theorem~\ref{thm-odd-Kerr-theorem-NC}.
\end{itemize}
Conversely, any such distributions arise in the ways thus described. These f\/indings may be viewed as odd-dimensional counterparts of the work of \cite{Hughston1988}, where it is shown that local foliations of a $2m$-dimensional smooth quadric $\mc{Q}^{2m}$ by \emph{$\alpha$-planes}, i.e., totally null self-dual $m$-planes, are in one-to-one correspondence with certain $m$-dimensional complex submanifolds of twistor space, here def\/ined as the space of all $\alpha$-planes in $\mc{Q}^{2m}$.

The f\/irst two of the above results are \emph{conformally invariant}, and to arrive at them, we shall f\/irst describe the geometrical correspondence between $\mc{Q}^n$ and $\PT$ in a manifestly conformally invariant manner, by exploiting the vector and spinor representations of the complex conformal group $\SO(n+2,\C)$ and of its double-covering $\Spin(n+2,\C)$. Such a \emph{tractor} or \emph{twistor calculus}, as it is known, builds on Penrose's twistor calculus in four dimensions \cite{Penrose1967}. The more `standard', local and Poincar\'{e}-invariant approach to twistor geometry will also be introduced to describe non-conformally invariant mini-twistor space $\M\T$. In fact, a fairly detailed description of twistor geometry in odd dimensions will make up the bulk of this article, and should, we hope, have a wider range of applications than the one presented here. Once our calculus is all set up, our main results will follow almost immediately. The ef\/fectiveness of the tractor calculus will be exemplif\/ied by the construction of algebraic subvarieties of $\PT$, which describe the null foliations of $\mc{Q}^n$ arising from certain solutions of conformally invariant dif\/ferential operators.

Another aim of the present article is to distil the \emph{complex} geometry contained in a number of geometrical results on \emph{real} Euclidean space and Minkowski space in dimensions three and four. In fact, our work is motivated by the f\/indings of~\cite{Nurowski2010} and~\cite{Baird2013}. In the former reference, the author recasts the problem of f\/inding pairs of analytic conjugate functions on~$\E^n$ as a problem of f\/inding closed null complex-valued $1$-forms, and arrives at a description of the solutions in terms of real hypersurfaces of~$\C^{n-1}$. The case $n=3$ is of particular interest, and is the focus of the article~\cite{Baird2013}: the kernel of a null complex $1$-form on $\E^3$ consists of a complex line distribution $\Tgt^{(1,0)} \E^3$ and the span of a real unit vector~$\bm{u}$. This complex $2$-plane distribution is in fact the orthogonal complement $\big(\Tgt^{(1,0)} \E^3\big)^\perp$ of $\Tgt^{(1,0)} \E^3$, and we can think of $\Tgt^{(1,0)} \E^3 $ as a CR-structure compatible with the conformal structure on $\E^3$ viewed as an open dense subset of $S^3$. The condition that $\big(\Tgt^{(1,0)} \E^3\big)^\perp$ be integrable is equivalent to $\bm{u}$ being tangent to a \emph{conformal foliation}, otherwise known as a \emph{shearfree congruence} of curves. To f\/ind such congruences, the authors construct the $S^2$-bundle of unit vectors over $S^3$, which turns out to be a CR hypersurface in $\CP^3$. A section of this $S^2$-bundle def\/ines a congruence of curves, and this congruence is shearfree if and only if the section is a $3$-dimensional CR submanifold.

There are three antecedents for this result:
\begin{enumerate}\itemsep=0pt
\item[1)] 
there is a one-to-one correspondence between local self-dual Hermitian structures on \mbox{$\E^4 {\subset} S^4$} and holomorphic sections of the $S^2$-bundle $\CP^3 \rightarrow S^4$ known as the \emph{twistor bundle} -- this is a well-known result, see, e.g., \cite{Baird2013,Baird2003, Eells1985,Hughston1988,Salamon2009};
\item[2)] 
there is a one-to-one correspondence between local analytic \emph{shearfree congruences of null geodesics} in Minkowski space $\M$ and certain complex hypersurfaces of its twistor space, an auxilliary space isomorphic to $\CP^3$ -- this is known as the \emph{Kerr theorem} \cite{Cox1976,Penrose1967,Penrose1986};
\item[3)] 
there is a one-to-one correspondence between local \emph{shearfree congruences of geodesics} in~$\E^3$ and certain holomorphic curves in its mini-twistor space, the holomorphic tangent bundle of $\CP^1 \cong S^2$~-- such congruences can also be equivalently described by \emph{harmonic morphisms} \cite{Baird1988,Tod1995,Tod1995a}.
\end{enumerate}
Statements (1) and (2) are essentially the same result once they are cast in the complexif\/ication of~$\E^4$ and~$\M$.

The analogy between statement (1) and the result of \cite{Baird2013} can be understood in the following terms: in the former case, the integrable complex null $2$-plane distribution $\Tgt^{(1,0)} \E^4$ def\/ining the Hermitian structure is \emph{totally geodetic}, i.e., $\nabla_{\bm{X}} \bm{Y} \in \Gamma \big(\Tgt^{(1,0)} \E^4\big)$ for all $\bm{X},\bm{Y} \in \Gamma \big(\Tgt^{(1,0)} \E^4\big)$. In the latter case, the condition that $\bm{u}$ be tangent to a shearfree congruence is also equivalent to the complex null line distribution $\Tgt^{(1,0)} \E^3$ being (totally) geodetic. One could also think of the integrability of both $\Tgt^{(1,0)} \E^3$ (trivially) and $\big(\Tgt^{(1,0)} \E^3\big)^\perp$ as an analogue of the integrability of~$\Tgt^{(1,0)} \E^4$.

\looseness=-1 Finally, statement (3), unlike (1) and (2), breaks conformal invariance, and the additional data f\/ixing a metric on $\E^3$ induces a reduction of the $S^2$-bundle constructed in~\cite{Baird2013} to mini-twistor space~$\Tgt S^2$ of~(3). Correspondingly, for $\bm{u}$ to be tangent to a shearfree congruence of null geodesics, both~$\Tgt^{(1,0)} \E^3$ and $\big(\Tgt^{(1,0)} \E^3\big)^\perp$ must be totally geodetic, which is not a conformally invariant condition.

The structure of the paper is as follows. Section~\ref{sec-geo-bkgd} deals with the twistor geometry of a~smooth quadric $\mc{Q}^n$ focussing mostly on the case $n=2m+1$. In particular, we give an algebraic description of the canonical distribution on its twistor space. The geometric correspondence bet\-ween~$\mc{Q}^n$ and~$\PT$ is made explicit. Propositions~\ref{prop-XiZ} and~\ref{prop-geometric-T}, and Corollary~\ref{cor-D-geom} give a~twistorial articulation of incidence relations between $\gamma$-planes in~$\mc{Q}^n$. The mini-twistor space~$\M\T$ of complex Euclidean space~$\C \E^n$ is introduced in Section~\ref{sec-co-gam}. Points in~$\C \E^n$ correspond to embedded complex submanifolds of~$\PT$ and~$\M\T$, and their normal bundles are described in Section~\ref{sec-normal}. The main results, Theorems~\ref{thm-odd-Kerr-theorem-I}, \ref{thm-odd-Kerr-theorem-II} and~\ref{thm-odd-Kerr-theorem-NC}, as outlined above, are given in Section~\ref{sec-null-foliations}. In each case, a~purely geometrical explanation precedes a computational proof. In Section~\ref{sec-exa}, we give two examples on how to relate null foliations in~$\mc{Q}^n$ to complex varieties in~$\PT$, based on certain solutions to the twistor equation, in Propositions \ref{prop-Robinson-congruence-odd} and \ref{prop-Robinson-congruence-even}, and the conformal Killing--Yano equation, in Proposition~\ref{prop-Kerr-variety-O}. We wrap up the article with Appendix~\ref{app-cover}, which contains a~description of standard open covers of twistor space and correspondence space.

\section{Twistor geometry}\label{sec-geo-bkgd}
We describe each of the three main protagonists involved in this article in turn: a smooth quadric hypersurface in projective space, its twistor space and a~correspondence space f\/ibered over them. The projective variety approach is very much along the line of \cite{Hughston1983, Penrose1986}, while the reader should consult~\cite{Baston1989,vCap2009} for the corresponding homogeneous space description.

\looseness=-1 Throughout $\V$ will denote an $(n+2)$-dimensional complex vector space. We shall make use of the following abstract index notation: elements of $\V$ and its dual $\V^*$ will carry upstairs and downstairs calligraphic upper case Roman indices respectively, i.e., $V^\mc{A} \in \V$ and \mbox{$\alpha_\mc{A} \in \V^*$}. Symmetrisation and skew-symmetrisation will be denoted by round and square brackets respectively, i.e., $\alpha_{(\mc{A} \mc{B})} = \frac{1}{2} (\alpha_{\mc{A} \mc{B}} + \alpha_{\mc{B} \mc{A}})$ and $\alpha_{[\mc{A} \mc{B}]} = \frac{1}{2} (\alpha_{\mc{A} \mc{B}} - \alpha_{\mc{B} \mc{A}})$. These conventions will apply to other types of indices used throughout this article. We shall also use Einstein's summation convention, e.g., $V^\mc{A} \alpha_\mc{A}$ will denote the natural pairing of elements of $\V$ and $\V^*$. We equip~$\V$ with a~non-degenerate symmetric bilinear form $h_{\mc{A} \mc{B}}$, by means of which $\V \cong \V^*$: indices will be lowered and raised by~$h_{\mc{A} \mc{B}}$ and its inverse~$h^{\mc{A} \mc{B}}$ respectively. We also choose a complex orientation on~$\V$, i.e., a complex volume element $\varepsilon_{\mc{A}_1 \ldots \mc{A}_{n+2}}$ in $\wedge^{n+2} \V$. We shall denote by $G$ the complex spin group $\Spin(n+2,\C)$, the two-fold cover of the complex Lie group $\SO(n+2,\C)$ preserving $h_{\mc{A} \mc{B}}$ and $\varepsilon_{\mc{A}_1 \ldots \mc{A}_{n+2}}$.

Turning now to the spinor representations of $G$, we distinguish the odd- and even-dimensional cases:
\begin{itemize}\itemsep=0pt
\item $n=2m+1$: denote by $\Ss$ the $2^{m+1}$-dimensional irreducible spinor representation of $G$. Elements of $\Ss$ will carry upstairs bold lower case Greek indices, e.g., $S^{\bm{\upalpha}} \in \Ss$, and dual elements, downstairs indices. The Clif\/ford algebra $\Cl(\V,h_{\mc{A} \mc{B}})$ is linearly isomorphic to the exterior algebra $\wedge^\bullet \V$, and, identifying $\wedge^k \V$ with $\wedge^{2m+3-k} \V$ by Hodge duality for $k=0, \ldots, m+1$, it is also isomorphic, as a matrix algebra, to the space $\End(\Ss)$ of endomorphisms of $\Ss$. It is generated by matrices, denoted $\Gamma \ind{_{\mc{A}}_{\bm{\upalpha}}^{\bm{\upgamma}}}$, which satisfy the \emph{Clifford identity}
\begin{gather}\label{eq-Clifford_twistor_odd}
 \Gamma \ind{_{(\mc{A}}_{\bm{\upalpha}}^{\bm{\upgamma}}} \Gamma \ind{_{\mc{B})}_{\bm{\upgamma}}^{\bm{\upbeta}}} = - h \ind{_{\mc{A} \mc{B}}} \delta \ind*{_{\bm{\upalpha}}^{\bm{\upbeta}}} .
\end{gather}
Here $\delta \ind*{_{\bm{\upalpha}}^{\bm{\upbeta}}}$ is the identity element on $\Ss$. There is a spin-invariant inner product on $\Ss$ denoted $\Gamma^{(0)}_{\bm{\updelta \upbeta}}\colon \Ss \times \Ss \rightarrow \C$, yielding the isomorphism $\End(\Ss) \cong \Ss \otimes \Ss$. The resulting isomorphisms $\Cl(\V,h_{\mc{A}\mc{B}}) \cong \wedge^\bullet \V \cong \Ss \otimes \Ss$ will be realised by means of the bilinear forms on $\Ss$ with values in $\wedge^k \V^*$, for $k=1, \ldots , n+2$:
\begin{gather}\label{eq-spin-bilinear-form-odd}
\Gamma \ind*{^{(k)}_{\mc{A}_1 \ldots \mc{A}_k}_{\bm{\upalpha \upbeta}}} := \Gamma \ind{_{[\mc{A}_1}_{\bm{\upalpha}}^{\bm{\upgamma_1}}} \cdots \Gamma \ind{_{\mc{A}_k]}_{\bm{\upgamma_{k-1}}}^{\bm{\updelta_k}}} \Gamma^{(0)}_{\bm{\updelta_k \upbeta}} .
\end{gather}
These are symmetric in their spinor indices when $k \equiv m+1, m+2 \pmod{4}$ and skew-symmetric otherwise.
\item $n=2m$: $G$ has two $2^m$-dimensional irreducible chiral spinor representations, which we shall denote $\Ss$ and $\Ss'$. Elements of $\Ss$ and $\Ss'$ will carry upstairs unprimed and primed lower case bold Greek indices respectively, i.e., $A^{\bm{\upalpha}} \in \Ss$ and $B^{\bm{\upalpha'}} \in \Ss'$. Dual elements will carry downstairs indices. The Clif\/ford algebra $\Cl(\V,h_{\mc{A} \mc{B}})$ is isomorphic to $\End(\Ss \oplus \Ss')$ as a matrix algebra, and, linearly, to $\wedge^\bullet \V$. We can write its generators in terms of matrices~$\Gamma \ind{_{\mc{A}}_{\bm{\upalpha}}^{\bm{\upgamma'}}}$ and~$\Gamma \ind{_{\mc{A}}_{\bm{\upalpha'}}^{\bm{\upgamma}}}$ satisfying
\begin{gather*}
 \Gamma \ind{_{(\mc{A}}_{\bm{\upalpha}}^{\bm{\upgamma'}}} \Gamma \ind{_{\mc{B})}_{\bm{\upgamma'}}^{\bm{\upbeta}}} = - h \ind{_{\mc{A} \mc{B}}} \delta \ind*{_{\bm{\upalpha}}^{\bm{\upbeta}}} , \qquad \Gamma \ind{_{(\mc{A}}_{\bm{\upalpha'}}^{\bm{\upgamma}}} \Gamma \ind{_{\mc{B})}_{\bm{\upgamma}}^{\bm{\upbeta'}}} = - h \ind{_{\mc{A} \mc{B}}} \delta \ind*{_{\bm{\upalpha'}}^{\bm{\upbeta'}}} ,
\end{gather*}
where $\delta \ind*{_{\bm{\upalpha}}^{\bm{\upbeta}}}$ and $\delta \ind*{_{\bm{\upalpha'}}^{\bm{\upbeta'}}}$ are the identity elements on $\Ss$ and $\Ss'$ respectively. There are spin-invariant bilinear forms on $\Ss\oplus\Ss'$ inducing isomorphisms $\Ss^* \cong \Ss'$, $(\Ss')^* \cong \Ss$ when $m$ is even, and $\Ss^* \cong \Ss$ and $(\Ss')^* \cong \Ss'$ when $m$ is odd, and denoted $\Gamma \ind*{^{(0)}_{\bm{\upalpha \upbeta'}}}$, $\Gamma \ind*{^{(0)}_{\bm{\upalpha' \upbeta}}}$, and $\Gamma \ind*{^{(0)}_{\bm{\upalpha \upbeta}}}$, $\Gamma \ind*{^{(0)}_{\bm{\upalpha' \upbeta'}}}$ respectively.
The resulting isomorphisms $\Cl(\V,h_{\mc{A} \mc{B}}) \cong \wedge^\bullet \V \cong (\Ss \oplus \Ss' ) \otimes (\Ss \oplus \Ss')$ are realised by $\wedge^k \V$-valued bilinear forms $\Gamma \ind*{^{(k)}_{\bm{\upalpha}\bm{\upbeta}}}$, for $k \equiv m+1 \pmod 2$, and $\Gamma \ind*{^{(k)}_{\bm{\upalpha}\bm{\upbeta'}}}$, for $k \equiv m \pmod 2$ and so on.
\end{itemize}

We work in the holomorphic category throughout.

\subsection{Smooth quadric hypersurface}\label{sec-smooth-Q}
Let us denote by $X^\mc{A}$ the position vector in $\V$, which can be viewed as standard Cartesian coordinates on $\C^{n+2}$. The equivalence class of non-zero vectors in $\V$ that projects down to the same point in the projective space $\Pp \V \cong \CP^{n+1}$ will be denoted $[ \cdot ]$, and thus $[X^\mc{A}]$ will represent homogeneous coordinates on $\Pp \V$.

The zero set of the quadratic form associated to $h_{\mc{A} \mc{B}}$ on $\V$ def\/ines a null cone $\mc{C}$ in $\V$, and the projectivisation of $\mc{C}$ def\/ines a smooth quadric hypersurface $\mc{Q}^n$ in $\Pp \V$, i.e.,
\begin{gather*}
\mc{Q}^n = \big\{ \big[X^\mc{A}\big] \in \Pp \V \colon h \ind{_{\mc{A} \mc{B}}} X^\mc{A} X^\mc{B} = 0 \big\} .
\end{gather*}
By taking a suitable cross-section of $\mc{C}$, one can identify $\mc{Q}^n$ with the complexif\/ication $\C S^n$ of the standard $n$-sphere $S^n$ in Euclidean space $\E^{n+1}$. Using the af\/f\/ine structure on $\V$, $h_{\mc{A} \mc{B}}$ can be viewed as a f\/ield of bilinear forms on $\V$ and thus on $\mc{C}$. We can then pull back $h_{\mc{A} \mc{B}}$ to $\mc{Q}^n$ along any section of $\mc{C} \rightarrow \mc{Q}^n$ to a (holomorphic) metric on $\mc{Q}^n$. Dif\/ferent sections yield conformally related metrics on $\mc{Q}^n$, i.e., a (holomorphic) conformal structure on $\mc{Q}^n$. The \emph{projective tangent space} at a point $p$ of $\mc{Q}^n$ with homogeneous coordinate $\big[P^\mc{A}\big]$ is the linear subspace
\begin{gather*}
\mbf{T}_p \mc{Q}^n := \big\{ \big[X^\mc{A}\big] \in \mc{Q}^n \colon h_{\mc{A} \mc{B}} X^\mc{A} P^\mc{B} = 0\big \} ,
\end{gather*}
which can be seen to be the closure of the (holomorphic) tangent space $\Tgt_p \mc{Q}^n$ at $p \in \mc{Q}^n$ in the usual sense. The intersection of $\mbf{T}_p \mc{Q}^n$ and $\mc{Q}^n$ is a cone through $p$, and any point lying in this cone is connected to its vertex by a line that is null with respect to the conformal structure.

To obtain the Kleinian model of $\mc{Q}^n$, we f\/ix a null vector $\mr{X}^\mc{A}$ in $\V$, and denote by $P$ the stabiliser of the line spanned by $\mr{X}^\mc{A}$ in $G$. The transitive action of $G$ on $\V$ descends to a~transitive action on $\mc{Q}^n$, and since $P$ stabilises a point in $\mc{Q}^n$, we obtain the identif\/ication $G/P \cong \mc{Q}^n$. The subgroup $P$ is a parabolic subgroup of~$G$, and its Lie algebra $\mfp$ admits a Levi decomposition, that is, a splitting $\mfp = \mfp_0 \oplus \mfp_1$, where $\mfp_0$ is the reductive Lie algebra $\so(n,\C) \oplus \C$, and $\mfp_1$ is a nilpotent part, here isomorphic to~$(\C^n)^*$. We choose a complement $\mfp_{-1}$ of $\mfp$ in~$\g$, dual to~$\mfp_1$ via the Killing form on~$\g$, so that $\g = \mfp_{-1} \oplus \mfp$. There is a unique element spanning the centre $\mfz(\mfp_0) \cong \C$ of $\mfp_0$, which acts diagonally on $\mfp_0$, $\mfp_1$ and~$\mfp_{-1}$ with eigenvalues $0$, $1$ and~$-1$ respectively. For this reason, we refer to this element as the \emph{grading element} of the splitting $\g = \mfp_{-1} \oplus \mfp_0 \oplus \mfp_1$. This splitting is compatible with the Lie bracket $[ \cdot , \cdot ] \colon \g \times \g \rightarrow \g$ on $\g$ in the sense that $[ \mfp_i , \mfp_j ] \subset \mfp_{i+j}$, with the convention that $\mfp_i = \{ 0 \}$ for $|i|>1$. In particular, it is invariant under $\mfp_0$, but not under $\mfp$. However, the f\/iltration $\mfp^1 \subset \mfp^0 \subset \mfp^{-1} := \g$, where $\mfp^1 := \mfp_1$ and $\mfp^0 := \mfp_0 \oplus \mfp_1$, \emph{is} a f\/iltration of $\mfp$-modules on $\g$, and each of the $\mfp$-modules $\mfp^{-1}/\mfp^0$, $\mfp^0/\mfp^1$ and $\mfp^1$ is linearly isomorphic to the $\mfp_0$-modules $\mfp_{-1}$, $\mfp_0$ and $\mfp_1$ respectively. These properties are most easily verif\/ied by realising $\g$ in matrix form, i.e.,
\begin{gather*}
\includegraphics{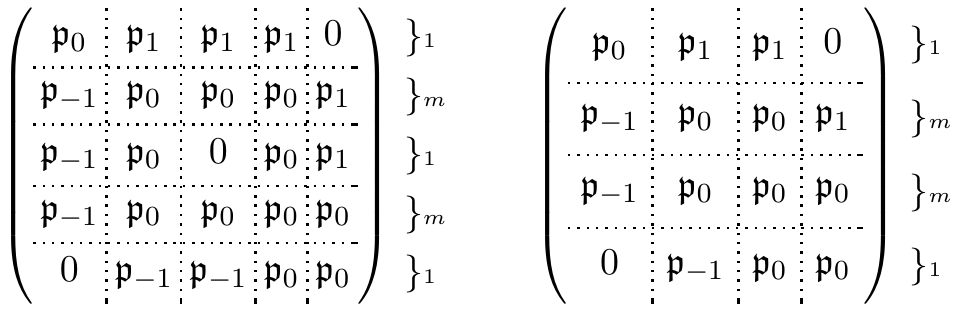}
\end{gather*}
when $n=2m+1$ and $n=2m$ respectively.

Given a vector representation $V$ of $P$, one can construct the holomorphic homogeneous vector bundle $G \times_P V$ over $G/P$: this is the orbit space of a point in $G \times V$ under the right action of $G$. In particular, the tangent bundle of $\mc{Q}^n$ can be described as $\Tgt (G/P) \cong G \times_P \left( \g/\mfp \right)$, and the tangent space at any point of $\mc{Q}^n$ is isomorphic to $\mfp_{-1} \cong \g/\mfp$~-- for a proof, see, e.g., \cite{vCap2009}. Similarly, denoting by $P_0$ the reductive subgroup of $P$ with Lie algebra $\mfp_0$, we can construct holomorphic homogeneous vector bundles from representations of $P_0$.

\subsubsection{The tractor bundle}\label{sec-tractor}
An important homogeneous vector bundle over $\mc{Q}^n$ is the one constructed from the standard representation~$\V$ of~$G$. It leads to a conformal invariant calculus, known as \emph{tractor calculus}. The reader should consult, e.g.,~\cite{Bailey1994,Curry2014} for further details.

\begin{Definition}
The \emph{$($complex$)$ standard tractor bundle} over $\mc{Q}^n \cong G/P$ is the rank-$(n+2)$ vector bundle $\mc{T} := G \times_P \V \cong G/P \times \V$.
\end{Definition}
The symmetric bilinear form $h_{\mc{A} \mc{B}}$ on $\V$ induces a non-degenerate holomorphic section of $\odot^2 \mc{T}^* \rightarrow \mc{Q}^n$ on $\mc{T}$, called the \emph{tractor metric}, also denoted by $h_{\mc{A} \mc{B}}$. Further, the af\/f\/ine structure on $\V$ induces a unique \emph{tractor connection} on $\mc{T}$ preserving $h_{\mc{A} \mc{B}}$.

The vector space $\V$ admits a f\/iltration of $P$-modules $\V =: \V^{-1} \supset \V^0 \supset \V^1$, where $\V^1 = \big\langle \mr{X}^\mc{A} \big\rangle$ and~$\V^0$ is the orthogonal complement of $\V^1$. These $P$-modules and their quotients $\V^i/\V^{i+1}$ give rise to $P$-invariant vector bundles as explained above. For convenience, we choose a splitting
\begin{align}\label{eq-split-V}
\V & = \V_{-1} \oplus \V_0 \oplus \V_1 ,
\end{align}
where $\V_1 :=\V^{-1}$, $\V_{-1}$ is a null line in $\V$ complementary to $\V^0 \subset \V^{-1}$, and $\V_0$ is the $n$-dimensional vector subspace orthogonal to both $\V_{-1}$ and $\V_1$. We note the linear isomorphisms $\V_{-1} \cong \V^{-1}/\V^0$ and $\V_0 \cong \V^0/\V^1$.

\looseness=-1 Let us introduce some abstract index notation. Elements of $\V_0$ and its dual $(\V_0)^*$ will be adorned with upstairs and downstairs lower-case Roman indices respectively, e.g., $V^a \in \V_0$ and $\alpha_a \in (\V_0)^*$. We f\/ix a null vector $\mr{Y}^\mc{A}$ spanning $\V_{-1}$ such that $\mr{X}^\mc{A} \mr{Y}_\mc{A} =1$. We also introduce the injector $\mr{Z}^\mc{A}_a \colon \V_0 \rightarrow \V$. Then, $h_{\mc{A} \mc{B}}$ restricts to a non-degenerate symmetric bilinear form $g_{ab} := \mr{Z}^\mc{A}_a \mr{Z}^\mc{B}_b h_{\mc{A} \mc{B}}$ on $\V_0$. Indices can be raised or lowered by means of $h_{\mc{A} \mc{B}}$, $g_{ab}$ and their inverses.

A geometric interpretation of $\mc{T} \rightarrow G/P$ can be found in~\cite{Curry2014} in a real setting. Here, we note that the line subbundle $G \times_P \V^1$ of $\mc{T}$ can be identif\/ied with the pull-back $\mc{O}[-1]$ of the tautological line bundle $\mc{O}(-1)$ on~$\Pp\V$ to $\mc{Q}^n$. The bundle $G \times_P \left( \V^{-1}/\V^0 \right)$ is isomorphic to the dual of $\mc{O}[-1]$, i.e., to the pullback $\mc{O}[1]$ of the hyperplane bundle $\mc{O}(1)$ on $\Pp \V$. Finally, since $\mfp_{-1} \otimes \V_1 \cong \V_0$, we have the identif\/ication $G \times_P \left( \V^0/ \V^1 \right) \cong \Tgt \mc{Q}^n \otimes \mc{O}[-1]$.

The structure sheaf of $\mc{Q}^n$ will be denoted $\mc{O}$, and the sheaf of germs of holomorphic functions on $\mc{Q}^n$ homogeneous of degree $w$ by $\mc{O}[w]$. We shall write $\mc{O}^a$ for the sheaf of germs of holomorphic sections of $\Tgt \mc{Q}^n$, and extend this notation in the obvious way to tensor products, e.g., $\mc{O}^{\mc{A}}_{ab}[w] := \mc{O}^{\mc{A}} \otimes \mc{O}_{ab} \otimes \mc{O}[w]$, and so on. In particular, the sheaf of germs of holomorphic sections of the tractor bundle $\mc{T}$ reads
\begin{gather}\label{eq-composition-series-V}
\mc{O}^\mc{A} = \mc{O}[1] + \mc{O}^a[-1] + \mc{O}[-1] .
\end{gather}
The line bundle $\mc{O}[1]$ has the geometric interpretation of the bundle of \emph{conformal scales}, and the conformal structure on $\mc{Q}^n$ can be equivalently encoded in terms of a distinguished global section $\mathbf{g}_{ab}$ of $\mc{O}_{(ab)}[2]$ called the \emph{conformal metric}. For any non-vanishing local section~$\sigma$ of~$\mc{O}[1]$, $g_{ab}=\sigma^{-2} \mathbf{g}_{ab}$ is a metric in the conformal class. A choice of metric in the conformal class is essentially equivalently to a splitting of~\eqref{eq-composition-series-V}, i.e., a choice of section $Y^\mc{A}$ of $\mc{O}^\mc{A} [-1]$ such that $Y^\mc{A} Y_\mc{A} =0$ and $X^\mc{A} Y_\mc{A} = 1$, where we view $X^\mc{A} \in \mc{O}^\mc{A} [1]$ as the Euler vector f\/ield on $\mc{C} \subset \V$. We can then choose a section $Z_a^\mc{A}$ of $\mc{O}^\mc{A}_a[1]$ satisfying $Z_a^\mc{A} Z_{b \mc{A}} = \mathbf{g}_{ab}$ and $Z_a^\mc{A} X_\mc{A} = Z_a^\mc{A} Y_\mc{A} = 0$, so that the tractor metric takes the form $h_{\mc{A} \mc{B}} = 2 X_{(\mc{A}} Y_{\mc{B})} + Z_\mc{A}^a Z_\mc{B}^b \mathbf{g}_{ab}$~-- see, e.g.,~\cite{Gover2008}. A~sec\-tion~$\Sigma^\mc{A}$ of~$\mc{O}^\mc{A}$ can be expressed as
\begin{gather}\label{eq-split-tractor}
\Sigma^\mc{A} = \sigma Y^\mc{A} + \varphi^a Z_a^\mc{A} + \rho X^\mc{A} , \qquad \text{where} \quad (\sigma,\varphi^a,\rho) \in \mc{O}[1] \oplus \mc{O}^a[-1] \oplus \mc{O}[-1].
\end{gather}
We shall denote both the tractor connection and the Levi-Civita connection of a metric in the conformal class by $\nabla_a$. The explicit formula for the tractor connection on a section~\eqref{eq-split-tractor} of~$\mc{O}^\mc{A}$ in terms of a splitting of~\eqref{eq-composition-series-V} can then be recovered from the Leibniz rule and the formulae
\begin{gather}\label{eq-progXYZ}
\nabla_a X^\mc{A} = Z^\mc{A}_a , \qquad
\nabla_a Z^\mc{A}_b = - \Rho_{ab} X^\mc{A} - \mathbf{g}_{ab} Y^\mc{A} , \qquad
\nabla_a Y^\mc{A} = \Rho \ind{_a^b} Z^\mc{A}_b ,
\end{gather}
where $\Rho_{ab}$ is the Schouten tensor of $\nabla_a$ def\/ined by the relation $2 \nabla_{[a} \nabla_{b]} V^c = 2 \Rho \ind{^c_{[a}} V \ind{_{b]}} - 2 V^d \Rho \ind{_{d[a}} \delta \ind*{_{b]}^c} $.

{\bf Complex Euclidean space.} 
Most of this paper will be concerned with the geometry on $n$-dimensional complex Euclidean space $\C \E^n$ viewed as a dense open subset of $\mc{Q}^n$, i.e., $\C \E^n = \mc{Q}^n \setminus \{\infty\}$ where $\infty$ is a point at `inf\/inity' on $\C S^n \cong \mc{Q}^n$. We choose a conformal scale $\sigma \in \mc{O}[1]$ so that $g_{ab}$ is the f\/lat metric, i.e., $\Rho_{ab}=0$. To realise $\sigma$ geometrically, we use the splitting~\eqref{eq-split-V}. Then, $\C \E^n$ arises as the intersection of the af\/f\/ine hyperplane $\mc{H} := \{ X^\mc{A} \in \V \colon X^\mc{A} \mr{Y}_\mc{A} = 1 \}$ with $\mc{C}$: $\V_1 = \langle \mr{X}^\mc{A} \rangle$ descends to the origin on $\C \E^n$, and $\V_{-1} = \langle \mr{Y}^\mc{A} \rangle$ represents~$\infty$ on~$\mc{Q}^n$. The f\/lat metric $g_{ab}$ is obtained by pulling back $h_{\mc{A} \mc{B}}$ along the local section $\mc{C} \cap \mc{H}$ of $\mc{C} \rightarrow \mc{Q}^n $. Let\-ting~$\{ x^a \}$ be f\/lat coordinates on $\C \E^n$ so that $\nabla_a = \parderv{}{x^a}$, we can integrate \eqref{eq-progXYZ} explicitly to get
\begin{gather}\label{eq-Minkowski-emb}
Y^\mc{A} = \mr{Y}^\mc{A} , \qquad Z^\mc{A}_a = \mr{Z}^\mc{A}_a - g_{ab} x^b \mr{Y}^\mc{A} , \qquad X^\mc{A} = \mr{X}^\mc{A} + x^a \mr{Z}^\mc{A}_a - \tfrac{1}{2} g_{ab} x^a x^b \mr{Y}^\mc{A} .
\end{gather}
This description is also consistent with the identif\/ication of $\C\E^n$ with the tangent space at the `origin' of $\mc{Q}^n$. In this case, the coordinates $\{ x^a \}$ arise from $\mfp_{-1} \cong \V_{-1} \otimes \V_0$ via the exponential map, which provides an embedding of $\C\E^n$ into $\mc{Q}^n$, $x^a \mapsto [ X^\mc{A}]$ where $X^\mc{A}$ is given by~\eqref{eq-Minkowski-emb}. The embedding can in fact be extended to a~\emph{conformal} embedding
\begin{gather*}
\begin{split}&
\C \E^n \rightarrow \mc{C} \rightarrow \mc{Q}^n , \\
& x^a \mapsto \Omega X^\mc{A} = \Omega \mr{X}^\mc{A} + x^a \Omega \mr{Z} _a^\mc{A} - \tfrac{1}{2} \big( \Omega^2 g_{ab} x^a x^b \big) \Omega^{-1} \mr{Y}^\mc{A} \mapsto \big[X^\mc{A}\big],
\end{split}
\end{gather*}
obtained by intersecting $\mc{C}$ with the af\/f\/ine hypersurface $\mc{H}_{\Omega} := \{ X^\mc{A} \in \V \colon X^\mc{A} \mr{Y}_\mc{A} = \Omega \}$, where~$\Omega$ is a non-vanishing holomorphic function on~$\V$.

\subsubsection{The tractor spinor bundle}
We can play the same game by considering bundles over $\mc{Q}^n$ arising from the spinor representations of $G=\Spin(n+2,\C)$. Again, we distinguish the odd- and even-dimensional cases.

{\bf Odd dimensions.} Assume $n=2m+1$.
\begin{Definition}The \emph{tractor spinor bundle} and \emph{dual tractor spinor bundle} over $\mc{Q}^n \cong G/P$ are the holomorphic homogeneous vector bundles $\mc{S} := G \times_{P} \Ss$ and $\mc{S}^* := G \times_{P} \Ss^*$ respectively.
\end{Definition}
The generators $\Gamma \ind{_{\mc{A}}_{\bm{\upalpha}}^{\bm{\upbeta}}}$ of the Clif\/ford algebra $(\V, h_{\mc{A} \mc{B}})$ induce holomorphic sections of $\mc{T}^* \otimes \mc{S}^* \otimes \mc{S}$ on $\mc{Q}^n$, which we shall also denote by $\Gamma \ind{_{\mc{A}}_{\bm{\upalpha}}^{\bm{\upbeta}}}$. The tractor connection on $\mc{Q}^n$ extends to a~\emph{tractor spinor connection} on~$\mc{S}$ preserving $\Gamma \ind{_{\mc{A}}_{\bm{\upalpha}}^{\bm{\upbeta}}}$, and thus $h_{\mc{A} \mc{B}}$.

There is a f\/iltration of $P$-submodules $\Ss =: \Ss^{-\frac{1}{2}} \supset \Ss^{\frac{1}{2}}$. These $P$-modules and their quotients give rise to $P$-invariant vector bundles on $\mc{Q}^n$ in the standard way. The splitting \eqref{eq-split-V} of $\V$ induces a splitting
\begin{gather}\label{eq-S->S1/2-odd}
\Ss \cong \Ss_{-\frac{1}{2}} \oplus \Ss_{\frac{1}{2}} ,
\end{gather}
where $\Ss_{\frac{1}{2}} \cong \V_1 \otimes \Ss_{-\frac{1}{2}}$, and we can identify $\Ss_{-\frac{1}{2}}$, and thus $\Ss_{\frac{1}{2}}$, as the spinor representation for $(\V_0,g_{ab})$. Similar considerations apply to $\Ss^*$. See, e.g., \cite{Harnad1992,Harnad1995} for details.

Elements of $\Ss_{\pm\frac{1}{2}}$ will carry bold upper case Roman indices, e.g., $\xi^\mbf{A} \in \Ss_{\pm\frac{1}{2}}$. The Clif\/ford algebra generators $\gamma \ind{_a_{\mbf{A}}^{\mbf{B}}}$ satisfy $\gamma \ind{_{(a}_{\mbf{A}}^{\mbf{C}}} \gamma \ind{_{b)}_{\mbf{C}}^{\mbf{B}}} = - g_{ab} \delta \ind*{_{\mbf{A}}^{\mbf{B}}}$, where $\delta \ind*{_{\mbf{A}}^{\mbf{B}}}$ is the identity on $\Ss_{\pm\frac{1}{2}}$. There is a spin-invariant bilinear form $\gamma \ind*{^{(0)}_{\mbf{AB}}}$ on $\Ss_{\pm\frac{1}{2}}$, by means of which we can def\/ine bilinear forms
\begin{gather*}
\gamma \ind*{^{(k)}_{a_1 \ldots a_k}_{\mbf{A} \mbf{B}}} := \gamma \ind{_{[a_1}_{\mbf{A}}^{\mbf{C_1}}} \cdots \gamma \ind{_{a_k]}_{\mbf{C_{k-1}}}^{\mbf{C_k}}} \gamma \ind*{^{(0)}_{\mbf{C_k} \mbf{B}}} ,
\end{gather*}
from $\Ss_{\pm\frac{1}{2}} \times \Ss_{\pm\frac{1}{2}}$ to $\wedge^k \V_0$ for $k=1, \ldots, n$. We introduce projectors $\mr{O}_{\bm{\upalpha}}^{\mbf{A}}\colon \Ss \rightarrow \Ss_{-\frac{1}{2}}$ and \mbox{$\mr{I}_{\bm{\upalpha}}^{\mbf{A}} \colon \Ss \rightarrow \Ss_{\frac{1}{2}}$}, and injectors $\mr{I}^{\bm{\upalpha}}_{\mbf{A}}\colon \Ss_{-\frac{1}{2}} \rightarrow \Ss$ and $\mr{O}^{\bm{\upalpha}}_{\mbf{A}}\colon \Ss_{\frac{1}{2}} \rightarrow \Ss$, which satisfy $\mr{O}_{\bm{\upalpha}}^{\mbf{B}} \mr{I}^{\bm{\upalpha}}_{\mbf{A}}= \delta_\mbf{A}^\mbf{B}$ and $\mr{O}_{\bm{\upalpha}}^{\mbf{A}} \mr{I}^{\bm{\upbeta}}_{\mbf{A}} + \mr{I}_{\bm{\upalpha}}^{\mbf{A}} \mr{O}^{\bm{\upbeta}}_{\mbf{A}} = \delta_{\bm{\upalpha}}^{\bm{\upbeta}}$. Then one can check that the relation between $\Gamma \ind{_{\mc{A}}_{\bm{\upalpha}}^{\bm{\upbeta}}} $ and $\gamma \ind{_a_{\mbf{A}}^{\mbf{B}}}$ is given by
\begin{gather}\label{eq-Gam2gam-odd}
\Gamma \ind{_{\mc{A}}_{\bm{\upalpha}}^{\bm{\upbeta}}} = \mr{Z}_\mc{A}^a \big( \mr{O}_{\bm{\upalpha}}^{\mbf{A}} \mr{I}^{\bm{\upbeta}}_{\mbf{B}} \gamma \ind{_a_{\mbf{A}}^{\mbf{B}}} - \mr{I}_{\bm{\upalpha}}^{\mbf{A}} \mr{O}^{\bm{\upbeta}}_{\mbf{B}} \gamma \ind{_a_{\mbf{A}}^{\mbf{B}}} \big) + \sqrt{2} \mr{Y}_\mc{A} \mr{O}_{\bm{\upalpha}}^{\mbf{A}} \mr{O}^{\bm{\upbeta}}_{\mbf{A}} - \sqrt{2} \mr{X}_\mc{A} \mr{I}_{\bm{\upalpha}}^{\mbf{A}} \mr{I}^{\bm{\upbeta}}_{\mbf{A}} .
\end{gather}

Sheaves of germs of holomorphic sections of $G \times_P \big( \Ss^{-\frac{1}{2}} / \Ss^{-\frac{1}{2}}\big)$ will be denoted $\mc{O}^{\mbf{A}}$, and we shall write $\mc{O}^{\mbf{A}} [-1]:= \mc{O}^{\mbf{A}} \otimes \mc{O} [-1]$, and similarly for dual bundles in the obvious way. In particular, the sheaves of germs of holomorphic sections of $\mc{S}$ and its dual are given by
\begin{gather}\label{eq-tractor-spin-bundle-odd}
\mc{O}^{\bm{\upalpha}} = \mc{O}^{\mbf{A}} + \mc{O}^{\mbf{A}} [-1] , \qquad
\mc{O}_{\bm{\upalpha}} = \mc{O}_{\mbf{A}} [1] + \mc{O}_{\mbf{A}} ,
\end{gather}
respectively. The splitting of \eqref{eq-tractor-spin-bundle-odd} can be realised by means of injectors/projectors $O_{\bm{\upalpha}}^{\mbf{A}} \in \mc{O}_{\bm{\upalpha}}^{\mbf{A}}$, $I_{\bm{\upalpha}}^{\mbf{A}} \in \mc{O}_{\bm{\upalpha}}^{\mbf{A}} [-1]$, $O^{\bm{\upalpha}}_{\mbf{A}} \in \mc{O}^{\bm{\upalpha}}_{\mbf{A}} [1]$ and $I^{\bm{\upalpha}}_{\mbf{A}} \in \mc{O}^{\bm{\upalpha}}_{\mbf{A}}$, such that $O_{\bm{\upalpha}}^{\mbf{A}} I^{\bm{\upalpha}}_{\mbf{B}} = \delta_{\mbf{B}}^{\mbf{A}}$, $I_{\bm{\upalpha}}^{\mbf{A}} O^{\bm{\upalpha}}_{\mbf{B}} = \delta_{\mbf{B}}^{\mbf{A}}$, and $O_{\bm{\upalpha}}^{\mbf{A}} I_{\mbf{A}}^{\bm{\upbeta}} + I_{\bm{\upalpha}}^{\mbf{A}} O_{\mbf{A}}^{\bm{\upbeta}} = \delta_{\bm{\upalpha}}^{\bm{\upbeta}}$, while all the other pairings are zero. In particular, we shall express a~section of $\mc{O}^{\bm{\upalpha}}$ as
\begin{gather*}
\Xi^{\bm{\upalpha}} = I^{\bm{\upalpha}}_{\mbf{A}} \xi^{\mbf{A}} + O^{\bm{\upalpha}}_{\mbf{A}} \zeta^{\mbf{A}} , \qquad \text{where} \quad \big(\xi^{\mbf{A}} , \zeta^{\mbf{A}}\big) \in \mc{O}^{\mbf{A}} + \mc{O}^{\mbf{A}} [-1],
\end{gather*}
and similarly for dual tractor spinors.

By abuse of notation, the connection on $\mc{S}$ and the spin connection associated to a metric in the conformal class will both be denoted $\nabla_a$. They satisfy
\begin{alignat}{3}
& \nabla_a O_{\bm{\upalpha}}^{\mbf{A}} = - \frac{1}{\sqrt{2}} \bm{\upgamma} \ind{_{a {\mbf{B}}} ^{\mbf{A}}} I_{\bm{\upalpha}}^{\mbf{B}} , \qquad &&
\nabla_a I_{\bm{\upalpha}}^{\mbf{A}} = - \frac{1}{\sqrt{2}} \Rho \ind{_{a b}} \bm{\upgamma} \ind{^b _{\mbf{B}} ^{\mbf{A}}} O_{\bm{\upalpha}}^{\mbf{B}} , & \nonumber\\
& \nabla_a O^{\bm{\upalpha}}_{\mbf{A}} = \frac{1}{\sqrt{2}} \bm{\upgamma} \ind{_{a {\mbf{A}}} ^{\mbf{B}}} I^{\bm{\upalpha}}_{\mbf{B}} , \qquad &&
\nabla_a I^{\bm{\upalpha}}_{\mbf{A}} = \frac{1}{\sqrt{2}} \Rho \ind{_{a b}} \bm{\upgamma} \ind{^b _{\mbf{A}} ^{\mbf{B}}} O^{\bm{\upalpha}}_{\mbf{B}}, &\label{eq-prop-OI-odd}
\end{alignat}
where $\bm{\upgamma} \ind{_a_{\mbf{A}}^{\mbf{B}}} \in \mc{O} \ind{_a_{\mbf{A}}^{\mbf{B}}} [1]$ satisfy $\bm{\upgamma}\ind{_{(a}_{\mbf{A}}^{\mbf{C}}} \bm{\upgamma} \ind{_{b)}_{\mbf{C}}^{\mbf{B}}} = - \mbf{g}_{ab} \delta \ind*{_{\mbf{A}}^{\mbf{B}}}$. The bundle analogue of~\eqref{eq-Gam2gam-odd} is
\begin{gather*}
\Gamma \ind{_{\mc{A}}_{\bm{\upalpha}}^{\bm{\upbeta}}} = Z_\mc{A}^a \big( O_{\bm{\upalpha}}^{\mbf{A}} I^{\bm{\upbeta}}_{\mbf{B}} \bm{\upgamma} \ind{_a_{\mbf{A}}^{\mbf{B}}} - I_{\bm{\upalpha}}^{\mbf{A}} O^{\bm{\upbeta}}_{\mbf{B}} \bm{\upgamma} \ind{_a_{\mbf{A}}^{\mbf{B}}} \big) + \sqrt{2} Y_\mc{A} O_{\bm{\upalpha}}^{\mbf{A}} O^{\bm{\upbeta}}_{\mbf{A}} - \sqrt{2} X_\mc{A} I_{\bm{\upalpha}}^{\mbf{A}} I^{\bm{\upbeta}}_{\mbf{A}} .
\end{gather*}

With a choice of conformal scale $\sigma \in \mc{O}[1]$ for which $g_{ab}=\sigma^{-2} \mbf{g}_{ab}$ is f\/lat, i.e., $\Rho_{ab}=0$, equations~\eqref{eq-prop-OI-odd} can be integrated explicitly to give
\begin{gather*}
I^{\bm{\upalpha}}_{\mbf{A}} = \mr{I}^{\bm{\upalpha}}_{\mbf{A}} , \qquad O^{\bm{\upalpha}}_{\mbf{A}} = \mr{O}^{\bm{\upalpha}}_{\mbf{A}} + \frac{1}{\sqrt{2}} x^a \gamma \ind{_{a {\mbf{A}}} ^{\mbf{B}}} \mr{I}^{\bm{\upalpha}}_{\mbf{B}} ,\qquad
I_{\bm{\upalpha}}^{\mbf{A}} = \mr{I}_{\bm{\upalpha}}^{\mbf{A}} , \qquad O_{\bm{\upalpha}}^{\mbf{A}} = \mr{O}_{\bm{\upalpha}}^{\mbf{A}} - \frac{1}{\sqrt{2}} x^a \gamma \ind{_{a {\mbf{A}}} ^{\mbf{B}}} \mr{I}_{\bm{\upalpha}}^{\mbf{B}} ,
\end{gather*}
where $\gamma \ind{_{a {\mbf{A}}} ^{\mbf{B}}} = \sigma^{-1} \bm{\upgamma} \ind{_{a {\mbf{A}}} ^{\mbf{B}}}$.

{\bf Even dimensions.} When $n=2m$, the story is similar, except that, by virtue of the two chiral spinor representations, we have an unprimed tractor spinor bundle and a primed tractor spinor bundle, def\/ined as $\mc{S} := G \times_{P} \Ss$ and $\mc{S}' := G \times_{P} \Ss'$ respectively. We shall view the genera\-tors~$\Gamma \ind{_{\mc{A}}_{\bm{\upalpha}}^{\bm{\upbeta'}}}$ and $\Gamma \ind{_{\mc{A}}_{\bm{\upalpha'}}^{\bm{\upbeta}}}$ as holomorphic sections of $\mc{T}^* \otimes \mc{S}^* \otimes \mc{S}'$ and $\mc{T}^* \otimes (\mc{S}')^* \otimes \mc{S}$ respectively on~$\mc{Q}^n$, both of which are preserved by the extension of the tractor connection to~$\mc{S}\oplus \mc{S}'$.

The spinor spaces $\Ss$ and $\Ss'$ admit f\/iltrations of $P$-submodules $\Ss =: \Ss^{-\frac{1}{2}} \supset \Ss^{\frac{1}{2}}$ and $\Ss' =: {\Ss'}^{-\frac{1}{2}} \supset {\Ss'}^{\frac{1}{2}}$. These $P$-modules and their quotients give rise to $P$-invariant vector bundles on $\mc{Q}^n$ in the standard way. The splitting \eqref{eq-split-V} on $\V$ induces a splitting of these f\/iltrations
\begin{gather}\label{eq-S->S1/2-even}
\Ss \cong \Ss_{-\frac{1}{2}} \oplus \Ss_{\frac{1}{2}} , \qquad \Ss' \cong \Ss_{-\frac{1}{2}}' \oplus \Ss_{\frac{1}{2}}' ,
\end{gather}
where $\Ss'_{\frac{1}{2}} \cong \V_1 \otimes \Ss_{-\frac{1}{2}}$ and $\Ss_{\frac{1}{2}} \cong \V_1 \otimes \Ss'_{-\frac{1}{2}}$, and we can identify $\Ss_{-\frac{1}{2}}$ and $\Ss'_{-\frac{1}{2}}$, and thus $\Ss'_{\frac{1}{2}}$ and $\Ss_{\frac{1}{2}}$, as the chiral spinor representations of $(\V_0,g_{ab})$. Elements of $\Ss_{-\frac{1}{2}}$ and $\Ss_{\frac{1}{2}}$ will carry unprimed and primed upper case Roman indices respectively, e.g., $\eta^{\mbf{A}}\in \Ss_{\frac{1}{2}}$ and $\xi^{\mbf{A'}} \in \Ss_{-\frac{1}{2}}$. The generators of the Clif\/ford algebra are matrices denoted $\gamma \ind{_a_{\mbf{A}}^{\mbf{B'}}}$ and $\gamma \ind{_a_{\mbf{B'}}^{\mbf{A}}}$, satisfying the Clif\/ford identities
$\gamma \ind{_{(a}_{\mbf{A}}^{\mbf{C'}}} \gamma \ind{_{b)}_{\mbf{C'}}^{\mbf{B}}} = - g_{ab} \delta \ind*{_{\mbf{A}}^{\mbf{B}}}$ and $\gamma \ind{_{(a}_{\mbf{A'}}^{\mbf{C}}} \gamma \ind{_{b)}_{\mbf{C}}^{\mbf{B'}}} = - g_{ab} \delta \ind*{_{\mbf{A'}}^{\mbf{B'}}}$, where $\delta \ind*{_{\mbf{A'}}^{\mbf{B'}}}$ and $\delta \ind*{_{\mbf{A}}^{\mbf{B}}}$ are the identity elements on $\Ss_{-\frac{1}{2}}$ and $\Ss_{\frac{1}{2}}$ respectively. We also obtain spin invariant bilinear forms $\gamma \ind*{^{(k)}_{\mbf{A'B'}}}$, $\gamma \ind*{^{(k)}_{\mbf{AB}}}$ and $\gamma \ind*{^{(k)}_{\mbf{AB'}}}$. The story for $\Ss'$ is similar.

We introduce projectors $\mr{O}_{\bm{\upalpha}}^{\mbf{A}}$, $\mr{I}_{\bm{\upalpha}}^{\mbf{A'}}$ and injectors $\mr{I}^{\bm{\upalpha}}_{\mbf{A}}$ and $\mr{O}^{\bm{\upalpha}}_{\mbf{A'}}$ for the splitting \eqref{eq-S->S1/2-even}, normalised in the obvious way.
The relation between the generators of the Clif\/ford algebra $\Cl(\V,h_{\mc{A}\mc{B}})$ and those of $\Cl(\V_0,g_{ab})$ is then given by
\begin{gather*}
\Gamma \ind{_{\mc{A}}_{\bm{\upalpha}}^{\bm{\upbeta'}}} = \mr{Z}_\mc{A}^a \big( \mr{O}_{\bm{\upalpha}}^{\mbf{A}} \mr{I}^{\bm{\upbeta'}}_{\mbf{B'}} \gamma \ind{_a_{\mbf{A}}^{\mbf{B'}}} - \mr{I}_{\bm{\upalpha}}^{\mbf{A'}} \mr{O}^{\bm{\upbeta'}}_{\mbf{B}} \gamma \ind{_a_{\mbf{A'}}^{\mbf{B}}} \big) + \sqrt{2} \mr{Y}_\mc{A} \mr{O}_{\bm{\upalpha}}^{\mbf{A}} \mr{O}^{\bm{\upbeta'}}_{\mbf{A}} - \sqrt{2} \mr{X}_\mc{A} \mr{I}_{\bm{\upalpha}}^{\mbf{A'}} \mr{I}^{\bm{\upbeta'}}_{\mbf{A'}} ,
\end{gather*}
and similar for $\Gamma \ind{_{\mc{A}}_{\bm{\upalpha'}}^{\bm{\upbeta}}}$ by interchanging primed and unprimed indices.

These algebraic objects extend to weighted tensor or spinor f\/ields just as in odd dimensions in the obvious way and notation. In particular, we have composition series of the unprimed and primed tractor spinor bundles:
\begin{alignat*}{3}
& \mc{O}^{\bm{\upalpha}} = \mc{O}^{\mbf{A}} + \mc{O}^{{\mbf{A'}}} [-1], \qquad && \mc{O}^{{\bm{\upalpha'}}} = \mc{O}^{{\mbf{A'}}} + \mc{O}^{\mbf{A}} [-1], &\\
& \mc{O}_{\bm{\upalpha}} = \mc{O}_{{\mbf{A'}}} [1] + \mc{O}_{\mbf{A}}, \qquad && \mc{O}_{{\bm{\upalpha'}}} = \mc{O}_{\mbf{A}} [1] + \mc{O}_{{\mbf{A'}}} .&
\end{alignat*}

\subsection{Twistor space}\label{sec-tw-sp}
The linear subspaces of $\mc{Q}^n$ can be described in terms of representations of $G=\Spin(n+2,\C)$. We shall be interested in those of maximal dimension, arising from maximal totally null vector subspaces of $(\V,h_{\mc{A} \mc{B}})$. In even dimensions, the complex orientation on $\V$ determines the duality of the corresponding linear subspaces, via Hodge duality, which are then described as either self-dual or anti-self-dual.

\begin{Definition}
An $m$-dimensional linear subspace of $\mc{Q}^{2m+1}$ is called a \emph{$\gamma$-plane}. A self-dual, respectively, anti-self-dual, $m$-dimensional linear subspace of $\mc{Q}^{2m}$ is called an \emph{$\alpha$-plane}, respectively, a \emph{$\beta$-plane}.

We call the space of all $\gamma$-planes in $\mc{Q}^{2m+1}$ the \emph{twistor space} of $\mc{Q}^{2m+1}$, and denote it by~$\PT_{(2m+1)}$. The space of all $\alpha$-planes, respectively, $\beta$-planes in $\mc{Q}^{2m}$ will be called the \emph{twistor space} $\PT_{(2m)}$, respectively, the \emph{primed twistor space} $\PT'_{(2m)}$.

A point in $\PT$ will be referred to as a \emph{twistor}.
\end{Definition}

We shall often write $\PT$ and $\PT'$ for $\PT_{(2m+1)}$ or $\PT_{(2m)}$, and $\PT'_{(2m)}$ respectively. We now distinguish the odd- and even-dimensional cases.

\subsubsection{Odd dimensions}
Assume $n=2m+1$. Let $ Z^{\bm{\upalpha}}$ be a non-zero spinor in $\Ss$, and def\/ine the linear map
\begin{gather}\label{eq-spinor-map}
Z_\mc{A}^{\bm{\upalpha}} := \Gamma \ind{_{\mc{A}}_{\bm{\upbeta}}^{\bm{\upalpha}}} Z^{\bm{\upbeta}} \colon \ \V \rightarrow \Ss .
\end{gather}
By \eqref{eq-Clifford_twistor_odd}, the kernel of \eqref{eq-spinor-map} is a~totally null vector subspace of~$\V$, and if it is non-trivial, descends to a linear subspace of $\mc{Q}^n$.

\begin{Definition}We say that a non-zero spinor $Z^{\bm{\upalpha}}$ in $\Ss$ is \emph{pure} if the kernel of $Z_\mc{A}^{\bm{\upalpha}} := \Gamma \ind{_{\mc{A}}_{\bm{\upbeta}}^{\bm{\upalpha}}} Z^{\bm{\upbeta}}$ has maximal dimension $m+1$.
\end{Definition}

The $(m+1)$-dimensional totally null subspace of $\V$ associated in this way to a pure spinor descends to a $\gamma$-plane in $\mc{Q}^n$. Clearly, any two pure spinors dif\/fering by a factor give rise to the same $\gamma$-plane. Further, one can show that any $\gamma$-plane in $\mc{Q}^n$ arises from a pure spinor up to scale. Hence,
\begin{Proposition}[\cite{Cartan1981}]
The twistor space $\PT$ of $\mc{Q}^{2m+1}$ is isomorphic to the projectivisation of the space of all pure spinors in $\Ss$.
\end{Proposition}

Every non-zero spinor in $\Ss$ is pure when $m=1$. Let us recall that the $\Gamma \ind*{^{(k)}_{\bm{\upalpha \upbeta}}}$ in the next theorem denote the spin bilinear forms def\/ined by \eqref{eq-spin-bilinear-form-odd}.
\begin{Theorem}[\cite{Cartan1981}]
When $m>1$, a non-zero spinor $Z^{\bm{\upalpha}}$ in $\Ss$ is pure if and only if it satisfies
\begin{gather}\label{eq-pure-spinor-odd}
\Gamma \ind*{^{(k)}_{\bm{\upalpha \upbeta}}} Z^{\bm{\upalpha}} Z^{\bm{\upbeta}} = 0 , \qquad \text{for all} \quad k < m+1 , \quad k \equiv m+2,m+1 \pmod{4},
\end{gather}
and $\Gamma \ind*{^{(m+1)}_{\bm{\upalpha \upbeta}}} Z^{\bm{\upalpha}} Z^{\bm{\upbeta}} \neq 0$.
\end{Theorem}
Alternatively, the quadratic relations \eqref{eq-pure-spinor-odd} can be expressed more succinctly by~\cite{Taghavi-Chabert2013}
\begin{gather}\label{eq-twistor-pure-cond_odd}
 Z^{\mc{A} \bm{\upalpha}} Z_\mc{A}^{\bm{\upbeta}} + Z^{\bm{\upalpha}} Z^{\bm{\upbeta}} = 0 .
\end{gather}

In analogy with the description of the quadric, we shall view $Z^{\bm{\upalpha}}$ as a position vector or coordinates on $\Ss$. The twistor space of $\mc{Q}^n$ can then be described as a complex projective variety of the projectivisation $\Pp \Ss$ of $\Ss$ with homogeneous coordinates $[Z^{\bm{\upalpha}}]$ satisfying \eqref{eq-pure-spinor-odd} or \eqref{eq-twistor-pure-cond_odd} when $m>1$. For $\mc{Q}^3$, we have $\PT_{(3)} \cong \CP^3$.

We shall adopt the following notation: if $Z$ is a point in $\PT$, with homogeneous coordina\-tes~$[Z^{\bm{\upalpha}}]$, then the corresponding $\gamma$-plane in $\mc{Q}^n$ will be denoted $\check{Z}$, i.e.,
\begin{gather*}
\check{Z} := \big\{ \big[X^\mc{A}\big] \in \mc{Q}^n \colon X^\mc{A} Z_\mc{A}^{\bm{\upalpha}} = 0 \big\} .
\end{gather*}
Let $\Xi$ be a twistor with homogeneous coordinates $[\Xi^{\bm{\upalpha}}]$ and associated $\gamma$-plane $\check{\Xi}$ in $\mc{Q}^n$. The projective tangent space of $\PT$ at $\Xi$ is the linear subspace of $\Pp\Ss$ def\/ined by
\begin{gather}\label{eq-proj-tgt-PT}
\mbf{T}_\Xi \PT := \big\{ [Z \ind{^{\bm{\upalpha}}}] \in \Pp\Ss \colon \Gamma \ind*{^{(k)}_{\bm{\upalpha}}_{\bm{\upbeta}}} Z \ind{^{\bm{\upalpha}}} \Xi \ind{^{\bm{\upbeta}}} = 0 , \, \text{for all $k<m-1$} \big\} .
\end{gather}
This is the closure of the holomorphic tangent space $\Tgt_\Xi \PT$ at $\Xi$, and contains the linear subspace
\begin{gather}\label{eq-lin-subsp}
\mbf{D}_\Xi := \big\{ [Z \ind{^{\bm{\upalpha}}}] \in \Pp\Ss \colon \Gamma \ind*{^{(k)}_{\bm{\upalpha}}_{\bm{\upbeta}}} Z \ind{^{\bm{\upalpha}}} \Xi \ind{^{\bm{\upbeta}}} = 0 ,\, \text{for all $k<m$} \big\} .
\end{gather}
This is the closure of a subspace $\mathrm{D}_\Xi$ of $\Tgt_\Xi \PT$. The smooth assignment of every point $\Xi$ of $\PT$ of $\mathrm{D}_\Xi$ yields a distribution that we shall denote $\mathrm{D}$. Another convenient way of expressing the locus in~\eqref{eq-lin-subsp} is~\cite{Taghavi-Chabert2013}
\begin{gather}\label{eq-lin-subsp-v2}
0 = Z \ind{^{\mc{A}}^{\bm{\upalpha}}} \Xi \ind*{_{\mc{A}}^{{\bm{\upbeta}}}} + 2 Z^{\bm{\upbeta}} \Xi^{\bm{\upalpha}} - Z^{\bm{\upalpha}} \Xi^{\bm{\upbeta}} ,
\end{gather}
where $Z_\mc{A}^{\bm{\upalpha}} := \Gamma \ind{_{\mc{A}}_{\bm{\upbeta}}^{\bm{\upalpha}}} Z^{\bm{\upbeta}}$ and $\Xi_\mc{A}^{\bm{\upalpha}} := \Gamma \ind{_{\mc{A}}_{\bm{\upbeta}}^{\bm{\upalpha}}} \Xi^{\bm{\upbeta}}$.

To understand $\PT$ more fully, we realise it as a Kleinian geometry. Let us f\/ix a pure spinor~$\Xi^{\bm{\upalpha}}$, and denote by $R$ the stabiliser of its span in $G$. This is a parabolic subgroup of $G$. Then, $\PT$ is isomorphic to $G/R$. One could equivalently realise~$\PT$ as the quotient of $\SO(n+2,\C)$ by the stabiliser of the corresponding $\gamma$-plane $\check{\Xi}$ in $\mc{Q}^n$. The Lie algebra $\mfr$ of $R$ induces a $|2|$-grading on $\g$, i.e., $\g = \mfr_{-2} \oplus \mfr_{-1} \oplus \mfr_0 \oplus \mfr_1 \oplus \mfr_2$, where $\mfr = \mfr_0 \oplus \mfr_1 \oplus \mfr_2$, with $\mfr_0 \cong \glie(m+1,\C)$, $\mfr_{-1} \cong \C^{m+1}$ and $\mfr_{-2} \cong \wedge^2 \C^{m+1}$, and $\mfr_{-1} \cong (\mfr_1)^*$, $\mfr_{-2} \cong (\mfr_2)^*$. In matrix form, this reads as
\begin{gather*}
\includegraphics[scale=0.93]{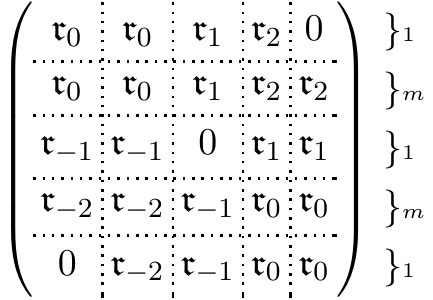}
\end{gather*}
These $\mfr_0$-modules satisfy the commutation relations $[\mfr_i,\mfr_j]\subset \mfr_{i+j}$ where $\mfr_i=\{ 0 \}$ for $|i|>2$. Further, $\g$ is equipped with a f\/iltration of $\mfr$-modules $\g := \mfr^{-2} \supset \mfr^{-1} \supset \mfr^0 \supset \mfr^1 \supset \mfr^2$ where $\mfr^i := \mfr_i \oplus \mfr^{i+1}$ satisfy $[\mfr^i , \mfr^j] \subset \mfr^{i+j}$. In particular, $\g/\mfr$ is not an irreducible $\mfr$-module, but admits a splitting into irreducible $\mfr$-submodules $\mfr^{-1} /\mfr$ and $\mfr^{-2} /\mfr^{-1}$. Since the tangent space at any point of $G/R$ can be identif\/ied with the quotient $\g/\mfr$, i.e., $\Tgt ( G/R ) \cong G \times_R \left( \g/\mfr \right)$, the tangent bundle of $\PT$ admits a f\/iltration of $R$-invariant subbundles $\Tgt \PT = \Tgt^{-2} \PT \supset \Tgt^{-1} \PT$, where the rank-$(m+1)$ distribution
\begin{gather}\label{eq-D}
 \Tgt^{-1} \PT := G \times_R \big( \mfr^{-1}/\mfr \big)
\end{gather}
is \emph{maximally non-integrable} by virtue of the commutation relations among the various graded pieces of $\g$, i.e., at every point $Z \in \PT$, $ \Tgt^{-1}_Z \PT \cong \mfr_{-1}$ and $[ \Tgt^{-1}_Z \PT , \Tgt^{-1}_Z \PT ] + \Tgt^{-1}_Z \PT \cong \mfr_{-1} \oplus \mfr_{-2}$.

We shall presently show that the distributions $\mathrm{D}$ def\/ined in terms of \eqref{eq-lin-subsp} and $\Tgt^{-1} \PT$ def\/ined by \eqref{eq-D} are the same. We f\/irst note that any spinor $Z^{\bm{\upalpha}} \in \Ss$ can be expressed as
\begin{gather}
Z^{\bm{\upalpha}} = Z_{(0)} \Xi^{\bm{\upalpha}}
+ \sum_{k=1}^{[(m+1)/2]}\left(- \frac{1}{4}\right)^k \frac{1}{k!} \left( Z_{(-2k)} \cdot \Xi \right) ^{\bm{\upalpha}}\nonumber\\
\hphantom{Z^{\bm{\upalpha}} =}{}
+ \frac{\ii}{2} \sum_{k=0}^{[(m+1)/2]}\left(- \frac{1}{4}\right)^k \frac{1}{k!} \left( Z_{(-2k-1)} \cdot \Xi \right) ^{\bm{\upalpha}} ,\label{eq-fock-spinor}
\end{gather}
where $Z_{(-i)} \in \wedge^i \mfr_{-1} \cong \wedge^i \C^{m+1}$, and $\left[\frac{m+1}{2}\right]$ is $\frac{m+1}{2}$ when $m+1$ is even, $\frac{m}{2}$ when $m+1$ is odd. Here, the $\cdot$ denotes the Clif\/ford action, i.e., $\left(\Phi \cdot \Xi \right)^{\bm{\upalpha}} = \Phi^\mc{A} \Gamma \ind{_{\mc{A}}_{\bm{\upbeta}}^{\bm{\upalpha}}} \Xi ^{\bm{\upbeta}}$, and so on as extended to the action of $\wedge^\bullet \V$ on $\Ss$. The factors have been chosen for convenience. The representation \eqref{eq-fock-spinor} is sometimes referred to as the \emph{Fock representation} \cite{Budinich1989}, and is already used implicitly in Cartan's work \cite{Cartan1981}, where the $Z_{(-i)}$ are viewed as the components of a spinor.

Now, using \eqref{eq-Clifford_twistor_odd} and \eqref{eq-spin-bilinear-form-odd}, together with \eqref{eq-pure-spinor-odd} applied to $\Xi^{\bm{\upalpha}}$, we compute $\Gamma \ind*{^{(m+2)}_{\bm{\upalpha \upbeta}}} Z^{\bm{\upalpha}} \Xi^{\bm{\upalpha}} = Z_{(0)} \Gamma \ind*{^{(m+2)}_{\bm{\upalpha \upbeta}}} \Xi^{\bm{\upalpha}} \Xi^{\bm{\upalpha}}$ and, for $k \geq0$,
\begin{gather}
\Gamma \ind*{^{(m-2k+1)}_{\mc{A}_1 \ldots \mc{A}_{m-2k+1}}_{\bm{\upalpha \upbeta}}} Z^{\bm{\upalpha}} \Xi^{\bm{\upalpha}} = \left( -\frac{1}{4} \right)^k \frac{1}{k!} Z \ind*{_{(-2k)}^{\mc{B}_{1} \ldots \mc{B}_{2k}}} \Gamma \ind*{^{(m+1)}_{\mc{B}_{1} \ldots \mc{B}_{2k} \mc{A}_1 \ldots \mc{A}_{m-2k+1}}_{\bm{\upalpha \upbeta}}} \Xi^{\bm{\upalpha}} \Xi^{\bm{\upalpha}} \nonumber\\
\hphantom{\Gamma \ind*{^{(m-2k+1)}_{\mc{A}_1 \ldots \mc{A}_{m-2k+1}}_{\bm{\upalpha \upbeta}}} Z^{\bm{\upalpha}} \Xi^{\bm{\upalpha}} =}{}
 + \frac{\ii}{2} \left( -\frac{1}{4} \right)^k \frac{1}{k!} Z \ind*{_{(-2k-1)}^{\mc{B}_{1} \ldots \mc{B}_{2k+1}}} \Gamma \ind*{^{(m+2)}_{\mc{B}_{1} \ldots \mc{B}_{2k+1} \mc{A}_1 \ldots \mc{A}_{m-2k+1}}_{\bm{\upalpha \upbeta}}} \Xi^{\bm{\upalpha}} \Xi^{\bm{\upalpha}} , \nonumber\\
\Gamma \ind*{^{(m-2k)}_{\mc{A}_1 \ldots \mc{A}_{m-2k}}_{\bm{\upalpha \upbeta}}} Z^{\bm{\upalpha}} \Xi^{\bm{\upalpha}} = \frac{\ii}{2} \left( -\frac{1}{4} \right)^k \frac{1}{k!} Z \ind*{_{(-2k-1)}^{\mc{B}_{1} \ldots \mc{B}_{2k+1}}} \Gamma \ind*{^{(m+1)}_{\mc{B}_{1} \ldots \mc{B}_{2k+1} \mc{A}_1 \ldots \mc{A}_{m-2k}}_{\bm{\upalpha \upbeta}}} \Xi^{\bm{\upalpha}} \Xi^{\bm{\upalpha}}\nonumber \\
\hphantom{\Gamma \ind*{^{(m-2k)}_{\mc{A}_1 \ldots \mc{A}_{m-2k}}_{\bm{\upalpha \upbeta}}} Z^{\bm{\upalpha}} \Xi^{\bm{\upalpha}} =}{}
 + \left( -\frac{1}{4} \right)^{k+1} \frac{1}{(k+1)!} Z \ind*{_{(-2k-2)}^{\mc{B}_{1} \ldots \mc{B}_{2k+2}}} \Gamma \ind*{^{(m+2)}_{\mc{B}_{1} \ldots \mc{B}_{2k+2} \mc{A}_1 \ldots \mc{A}_{m-2k}}_{\bm{\upalpha \upbeta}}} \Xi^{\bm{\upalpha}} \Xi^{\bm{\upalpha}} .\label{eq-GZXi}
\end{gather}
Here, we have added tractor indices to the $Z_{(-i)}$. We can immediately conclude
\begin{Lemma}\label{lem-Z-orbit-trc}
The conditions that $[Z^{\bm{\upalpha}}] \in \Pp \Ss$ lies in $\mbf{T}_\Xi \PT$ and $\mbf{D}_\Xi$ respectively are equivalent to
\begin{gather}
Z^{\bm{\upalpha}} = Z_{(0)} \Xi^{\bm{\upalpha}} + \tfrac{\ii}{2} \big( Z_{(-1)} \cdot \Xi \big)^{\bm{\upalpha}} - \tfrac{1}{4} \big( Z_{(-2)} \cdot \Xi\big)^{\bm{\upalpha}} , \label{eq-Z-orbit-trc-2} \\
Z^{\bm{\upalpha}} = Z_{(0)} \Xi^{\bm{\upalpha}} + \tfrac{\ii}{2} \big( Z_{(-1)} \cdot \Xi \big)^{\bm{\upalpha}} , \label{eq-Z-orbit-trc-1}
\end{gather}
respectively, up to overall factors. When $Z_{(0)}$ is non-zero, $[Z^{\bm{\upalpha}}]$ given by \eqref{eq-Z-orbit-trc-2} and \eqref{eq-Z-orbit-trc-1} lies in $\Tgt_\Xi \PT$ and $\mathrm{D}_\Xi$ respectively. In particular, $\mathrm{D} \cong \Tgt^{-1} \PT$.
\end{Lemma}

\begin{proof}
Equations \eqref{eq-Z-orbit-trc-2} and \eqref{eq-Z-orbit-trc-1} follow from def\/initions \eqref{eq-proj-tgt-PT} and \eqref{eq-lin-subsp} using \eqref{eq-GZXi}. Equation \eqref{eq-Z-orbit-trc-1} with $Z_{(0)} =1$ coincides with the exponential of an element of $\mfr_{-1}$ and thus describes a point in $\Tgt^{-1}_\Xi \PT$. The story for \eqref{eq-Z-orbit-trc-2} is similar.
\end{proof}

On the other hand, using \eqref{eq-pure-spinor-odd} or referring to \cite{Cartan1981}, the condition that $Z^{\bm{\upalpha}}$ be pure is that
\begin{gather}
Z_{(0)} Z_{(-2k-1)} = Z_{(-1)} \wedge Z_{(-2k)} , \nonumber\\
Z_{(0)} Z_{(-2k)} = Z_{(-2)} \wedge Z_{(-2k+2)} ,\qquad k=1, \ldots , [(m+1)/2] .\label{eq-purity-cond-twistor}
\end{gather}
A dense open subset of $\PT$ containing $[\Xi^{\bm{\upalpha}}]$ can be obtained by intersecting the locus~\eqref{eq-purity-cond-twistor} with the af\/f\/ine subspace $Z_{(0)} =1$ in $\Ss$. Summarising,

\begin{Proposition}\label{prop-canonical-distribution}
The twistor space $\PT$ of a $(2m+1)$-dimensional smooth quadric $\mc{Q}^{2m+1}$ has dimension $\frac{1}{2}(m+1)(m+2)$, and is equipped with a maximally non-integrable distribution $\mathrm{D}$ of rank $m+1$, i.e., $\Tgt \PT = \mathrm{D} + [ \mathrm{D} , \mathrm{D} ]$, where, for any $\Xi \in \PT$, $\mathrm{D}_\Xi$ is a dense open subset of $\mbf{D}_\Xi$ as defined by~\eqref{eq-lin-subsp}.

Further, for any $\Xi \in \PT$, the projective tangent space $\mbf{T}_\Xi \PT$ intersects $\PT$ in a $(2m+1)$-dimensional linear subspace of $\PT$, and $\mbf{D}_\Xi$ is an $(m+1)$-dimensional linear subspace of~$\PT$.
\end{Proposition}

\begin{proof}The f\/irst part has already been explained and stems from the general theory of~\cite{vCap2009}.

For the second part, we f\/ix a pure spinor $\Xi^{\bm{\upalpha}}$, and let $[Z^{\bm{\upalpha}}]$ be an element of the projective tangent space $\mbf{T}_\Xi \PT$ so that $Z^{\bm{\upalpha}}$ takes the form \eqref{eq-Z-orbit-trc-2}. If $[Z^{\bm{\upalpha}}]$ also lies in $\PT$, then, with reference to~\eqref{eq-purity-cond-twistor}, $Z_{(-1)} \wedge Z_{(-2)} = 0$ and $Z_{(-2)} \wedge Z_{(-2)} = 0$. Generically, $Z_{(-1)}$ is non-zero, so $Z_{(-2)} = Z_{(-1)} \wedge \Phi_{(-1)}$ for some $\Phi_{(-1)} \in \mfr_{-1}$. The form of $Z_{(-2)}$ remains invariant under the transformation $\Phi_{(-1)} \mapsto \Phi_{(-1)} + a Z_{(-1)}$ for any $a \in \C$. The choice of $Z_{(0)}$ is cancelled out by the freedom in the choice of scale of~\eqref{eq-Z-orbit-trc-2}. Thus, $\dim\left(\mbf{T}_\Xi \PT \cap \PT \right)= 2 \times (m+1) - 1 = 2m+1$. If~$[Z^{\bm{\upalpha}}]$ lies in $\mbf{D}_\Xi \PT$ , then it takes the form \eqref{eq-Z-orbit-trc-1}. In this case, the purity conditions~\eqref{eq-purity-cond-twistor} do not yield any further constraints, and thus $[Z^{\bm{\upalpha}}]$ must also lie in~$\PT$.
\end{proof}

\begin{Definition}\label{defn-canonical-distribution}The rank-$(m+1)$ distribution $\mathrm{D}$ will be referred to as the \emph{canonical distribution} of~$\PT$.
\end{Definition}

When $m=1$, the twistor space of $\mc{Q}^3$ is simply $\CP^3$ and the canonical distribution $\Drm$ is the rank-$2$ contact distribution annihilated by the contact $1$-form $\bm{\alpha} := \Gamma^{(0)}_{\bm{\upalpha}\bm{\upbeta}} Z^{\bm{\upalpha}} \dd Z^{\bm{\upbeta}}$.
The appropriate generalisation of this contact $1$-form to dimension $2m+1$ is then the set of $1$-forms
\begin{gather}\label{eq-hi-contact}
\bm{\alpha}^{\bm{\upalpha}\bm{\upbeta}} := Z \ind{^{\mc{A}}^{\bm{\upalpha}}} \dd Z \ind*{_{\mc{A}}^{{\bm{\upbeta}}}} + 2 Z^{\bm{\upbeta}} \dd Z ^{\bm{\upalpha}} - Z^{\bm{\upalpha}} \dd Z ^{\bm{\upbeta}} ,
\end{gather}
annihilating the canonical distribution $\Drm$. Here, the homogeneous coordinates $[Z^{\bm{\upalpha}}]$ are assumed to satisfy~\eqref{eq-pure-spinor-odd} or \eqref{eq-twistor-pure-cond_odd}.

The following lemma follows directly from the exponential map from a given complement of~$\mfr$ in~$\g$ to a dense open subset of $\PT$.

\begin{Lemma}\label{lem-fol-D}
Let $\Xi$ be a point in $\PT$, and let $\mfr$ be its stabiliser in $\g$. Then $\mbf{D}_\Xi$ is foliated by a~family of distinguished curves passing through~$\Xi$ parametrised by the points of the $(m+1)$-dimensional module~$\mfr_{-1}$, for any decomposition $\mfr = \mfr_{-2} \oplus \mfr_{-1} \oplus \mfr_0 \oplus \mfr_1 \oplus \mfr_2$.
\end{Lemma}

{\bf Geometric correspondences.} The bilinear forms \eqref{eq-spin-bilinear-form-odd} can also be used to characterise the intersections of $\gamma$-planes in terms of their corresponding pure spinors.
\begin{Theorem}[\cite{Cartan1981,Harnad1995}]\label{thm-gam-inters}
Let $Z$ and $W$ be two twistors with homogeneous coordinates~$[Z^{\bm{\upalpha}}]$ and~$[W^{\bm{\upalpha}}]$, and corresponding $\gamma$-planes $\check{Z}$ and~$\check{W}$ in $\mc{Q}^n$ respectively. Then
\begin{gather*}
\dim \big( \check{Z} \cap \check{W} \big) \geq k \quad \Longleftrightarrow \quad \Gamma \ind*{^{(\ell)}_{\bm{\upalpha \upbeta}}} Z^{\bm{\upalpha}} W^{\bm{\upbeta}} = 0 , \qquad \mbox{for all} \quad \ell \leq k.
\end{gather*}
Further, $\dim ( \check{Z} \cap \check{W} ) = k$ if and only if in addition $\Gamma \ind*{^{(k+1)}_{\bm{\upalpha \upbeta}}} Z^{\bm{\upalpha}} W^{\bm{\upbeta}} \neq 0$.
\end{Theorem}
A direct application leads to
\begin{Proposition}\label{prop-XiZ}
Let $\Xi$ and $Z$ be two twistors with corresponding $\gamma$-planes $\check{\Xi}$ and $\check{Z}$ respectively. Then
\begin{enumerate}\itemsep=0pt
\item[$1.$] $\dim(\check{\Xi} \cap \check{Z}) \geq m-3$ if and only if there exists $W \in \PT$ such that $W \in \mbf{D}_\Xi \cap \mbf{T}_Z \PT$ or $W \in \mbf{D}_Z \cap \mbf{T}_\Xi \PT$.
\item[$2.$] $\dim(\check{\Xi} \cap \check{Z}) \geq m-2$ if and only if $\Xi \in \mbf{T}_Z \PT$ if and only if $Z \in \mbf{T}_\Xi \PT$ if and only if there exists $W \in \PT$ such that $Z, \Xi \in \mbf{D}_W$, or equivalenly $W \in \mbf{D}_Z \cap \mbf{D}_\Xi$.
\item[$3.$] $\dim(\check{\Xi} \cap \check{Z}) \geq m-1$ if and only if $Z \in \mbf{D}_\Xi$ if and only if $\Xi \in \mbf{D}_Z$.
\end{enumerate}
\end{Proposition}

\begin{proof} We f\/ix $\Xi^{\bm{\upalpha}}$ and we assume that $Z^{\bm{\upalpha}}$ is given by~\eqref{eq-fock-spinor} with components $Z_{(-i)}$ satis\-fying~\eqref{eq-purity-cond-twistor}. In each case, we apply Theorem~\ref{thm-gam-inters} and compute $\Gamma \ind*{^{(\ell)}_{\bm{\upalpha \upbeta}}} Z^{\bm{\upalpha}} W^{\bm{\upbeta}} = 0$ to derive conditions on~$Z_{(-i)}$. With no loss of generality, we may assume $Z_{(0)} = 1$.

1.~We have $Z_{(-i)}$ for all $i\geq4$ and $Z_{(-2)} \wedge Z_{(-2)} = 0$, i.e., $Z_{(-2)} = \Phi_{(-1)} \wedge \Psi_{(-1)}$ for some $\Phi_{(-1)} , \Psi_{(-1)} \in \mfr_{-1}$, and $Z_{(-3)} = Z_{(-1)} \wedge Z_{(-2)} = Z_{(-1)} \wedge \Phi_{(-1)} \wedge \Psi_{(-1)}$. A suitable $W \in \mbf{D}_\Xi \cap \mbf{T}_Z \PT$ is given by $W^{\bm{\upalpha}} = \Xi^{\bm{\upalpha}} + \frac{\ii}{2} ( Z_{(-1)} \cdot \Xi)^{\bm{\upalpha}}$ and $W^{\bm{\upalpha}} = Z^{\bm{\upalpha}} + \frac{1}{4} (( \Phi_{(-1)} \wedge \Psi_{(-1)} ) \cdot Z)^{\bm{\upalpha}}$, and similarly for a suitable $W \in \mbf{D}_Z \cap \mbf{T}_\Xi \PT$.

2.~The f\/irst two equivalences follow immediately from Proposition~\ref{prop-canonical-distribution} and Theorem~\ref{thm-gam-inters}. For the last equivalence, we have $Z_{(-i)}$ for all $i\geq3$, so that $Z_{(-2)} \wedge Z_{(-2)} = 0$ and $Z_{(-1)} \wedge Z_{(-2)} = 0$, i.e., $Z_{(-2)} = Z_{(-1)} \wedge \Phi_{(-1)}$ for some $\Phi_{(-1)} \in \mfr_{-1}$. A suitable $W \in \mbf{D}_Z \cap \mbf{D}_\Xi$ is given by $W^{\bm{\upalpha}} = \Xi^{\bm{\upalpha}} + \frac{\ii}{2} ( Z_{(-1)} \cdot \Xi )^{\bm{\upalpha}}$ and $W^{\bm{\upalpha}} = Z^{\bm{\upalpha}} - \frac{\ii}{2} ( \Phi_{(-1)} \cdot Z )^{\bm{\upalpha}}$.

3.~This follows immediately from Proposition \ref{prop-canonical-distribution} and Theorem \ref{thm-gam-inters}.
\end{proof}

In a similar vein, we obtain
\begin{Proposition}\label{prop-geometric-T}
Fix a twistor $\Xi$ in $\PT$ and let $\check{\Xi }$ be its corresponding $\gamma$-plane in $\mc{Q}^n$. Let $Z$ and $W$ be two twistors in $\mbf{T}_\Xi \PT$, corresponding to $\gamma$-planes $\check{Z}$ and $\check{W}$. Then $\dim ( \check{Z} \cap \check{W} ) \geq m-4$.

Further, if $Z$ and $W$ take the respective forms
\begin{gather*}
Z^{\bm{\upalpha}} = Z_{(0)} \Xi^{\bm{\upalpha}} + \tfrac{\ii}{2} \big( Z_{(-1)} \cdot \Xi \big)^{\bm{\upalpha}} - \tfrac{1}{4} \big( Z_{(-2)} \cdot \Xi\big)^{\bm{\upalpha}} , \\
W^{\bm{\upalpha}} = W_{(0)} \Xi^{\bm{\upalpha}} + \tfrac{\ii}{2} \big( W_{(-1)} \cdot \Xi \big)^{\bm{\upalpha}} - \tfrac{1}{4} \big( W_{(-2)} \cdot \Xi\big)^{\bm{\upalpha}} ,
\end{gather*}
where $Z_{(0)} Z_{(-2)} = Z_{(-1)} \wedge Z_{(-1)}$, $Z_{(-2)} \wedge Z_{(-2)} = 0$, $W_{(0)} W_{(-2)} = W_{(-1)} \wedge W_{(-1)}$ and $W_{(-2)} \wedge W_{(-2)} = 0$, then
\begin{subequations}
\begin{alignat}{3}
& \dim \big( \check{Z} \cap \check{W} \big) \geq m-3 \quad \Longleftrightarrow \quad & & Z_{(-2)} \wedge W_{(-2)} = 0 ,& \label{eq-inter3} \\
& \dim \big( \check{Z} \cap \check{W} \big) \geq m-2 \quad \Longleftrightarrow \quad & & Z_{(-1)} \wedge W_{(-2)} + W_{(-1)} \wedge Z_{(-2)} = 0 , & \label{eq-inter2} \\
& \dim \big( \check{Z} \cap \check{W} \big) \geq m-1\quad \Longleftrightarrow \quad & & W_{(-2)} - Z_{(-2)} - W_{(-1)} \wedge Z_{(-1)} = 0 . & \label{eq-inter1}
\end{alignat}
\end{subequations}
\end{Proposition}

\begin{proof}Let us rewrite
\begin{gather*}
W^{\bm{\upalpha}} = Z^{\bm{\upalpha}} + \tfrac{\ii}{2} \big( \Phi_{(-1)} \cdot Z \big)^{\bm{\upalpha}} - \tfrac{1}{4} \big( \Phi_{(-2)} \cdot Z \big)^{\bm{\upalpha}} \\
\hphantom{W^{\bm{\upalpha}} =}{} - \tfrac{\ii}{8} \big( \big( \Phi_{(-1)} \wedge \Phi_{(-2)} \big) \cdot Z \big)^{\bm{\upalpha}} + \tfrac{1}{32} \big( \big( \Phi_{(-2)} \wedge \Phi_{(-2)} \big) \cdot Z \big)^{\bm{\upalpha}} ,
\end{gather*}
where $\Phi_{-1} := W_{-1} - Z_{-1}$ and $\Phi_{-2} := W_{-2} - Z_{-2} - W_{-1} \wedge Z_{-1}$. It suf\/f\/ices to compute $\Gamma \ind*{^{(m-k)}_{\bm{\upalpha \upbeta}}} Z^{\bm{\upalpha}} W^{\bm{\upbeta}} = 0$ for all $k \geq 4$, and
\begin{enumerate}\itemsep=0pt
\item[1)] $\Gamma \ind*{^{(m-k)}_{\bm{\upalpha \upbeta}}} Z^{\bm{\upalpha}} W^{\bm{\upbeta}} = 0$ for all $k \geq 3$ if and only if $\Phi_{(-2)} \wedge \Phi_{(-2)} = 0$;
\item[2)] $\Gamma \ind*{^{(m-k)}_{\bm{\upalpha \upbeta}}} Z^{\bm{\upalpha}} W^{\bm{\upbeta}} = 0$ for all $k \geq 2$ if and only if $\Phi_{(-1)} \wedge \Phi_{(-2)} = 0$;
\item[3)] $\Gamma \ind*{^{(m-k)}_{\bm{\upalpha \upbeta}}} Z^{\bm{\upalpha}} W^{\bm{\upbeta}} = 0$ for all $k \geq 1$ if and only if $\Phi_{(-2)}= 0$.
\end{enumerate}
Equivalences \eqref{eq-inter3}, \eqref{eq-inter2} and \eqref{eq-inter1} now follow from the def\/initions of $\Phi_{(-1)}$ and $\Phi_{(-2)}$.
\end{proof}

A special case of this proposition is given below.
\begin{Corollary}\label{cor-D-geom}
Fix a twistor $\Xi$ in $\PT$ and let $\check{\Xi }$ be its corresponding $\gamma$-plane in~$\mc{Q}^n$. Let~$Z$ and~$W$ be two twistors in~$\mbf{D}_\Xi$, corresponding to $\gamma$-planes~$\check{Z}$ and~$\check{W}$. Then $\dim ( \check{Z} \cap \check{W} ) \geq m-2$.

Further, $Z$ and $W$ belong to the same distinguished curve in $\mbf{D}_\Xi$, as defined in Lemma~{\rm \ref{lem-fol-D}}, if and only if $\dim(\check{Z} \cap \check{W}) \geq m-1$.
\end{Corollary}

\begin{proof}
This is a direct consequence of Proposition \ref{prop-geometric-T} with $Z_{(-2)} = W_{(-2)} = 0$, and Lem\-ma~\ref{lem-fol-D}.
\end{proof}

\subsubsection{Even dimensions}

Assume $n=2m$. Any non-zero chiral spinor $ Z^{\bm{\upalpha}}$ def\/ines a linear map $Z_\mc{A}^{\bm{\upalpha'}} := \Gamma \ind{_{\mc{A}}_{\bm{\upbeta}}^{\bm{\upalpha'}}} Z^{\bm{\upbeta}} \colon \V \rightarrow \Ss$, and similarly for primed spinors. Again, any non-trivial kernel of this map descends to a linear subspace of $\mc{Q}^n$. A non-zero chiral spinor $Z^{\bm{\upalpha}}$ is \emph{pure} if the kernel of $Z_\mc{A}^{\bm{\upalpha}}$ has maximal dimension $m+1$, and similarly for primed spinors.

\begin{Proposition}[\cite{Cartan1981}]
The twistor space $\PT$ and the primed twistor space $\PT'$ of $\mc{Q}^{2m}$ are isomorphic to the projectivations of the spaces of all pure spinors in $\Ss$ and $\Ss'$ respectively.
\end{Proposition}

When $m=2$, all spinors in $\Ss$ and $\Ss'$ are pure. When $m>2$, the analogue of the purity condition \eqref{eq-pure-spinor-odd} is now \cite{Cartan1981}
\begin{gather}\label{eq-pure-spinor-even}
\Gamma \ind*{^{(k)}_{\bm{\upalpha}\bm{\upbeta}}} Z^{\bm{\upalpha}} Z^{\bm{\upbeta}} = 0 , \qquad \text{for all} \quad k < m+1 ,\quad k \equiv m+1 \pmod{4},
\end{gather}
or alternatively, \cite{Hughston1988,Taghavi-Chabert2016}, $ Z^{\mc{A} \bm{\upalpha'}} Z_\mc{A}^{\bm{\upbeta'}} = 0$. Again, we will think of $\PT$ and $\PT'$ as complex projective varieties of~$\Pp \Ss$ and~$\Pp \Ss'$ respectively, when $m>2$, while for $\mc{Q}^4$, we have $\PT_{(4)} \cong \CP^3$.

The Kleinian model is again a homogeneous space $G/R$, where $R$ is parabolic. But its parabolic Lie algebra~$\mfr$ this time induces a $|1|$-grading $\g = \mfr_{-1} \oplus \mfr_0 \oplus \mfr_1$ on $\g$, where $\mfr_0 \cong \glie(m+1,\C)$, $\mfr_{-1} \cong \wedge^2 \C^{m+1}$ and $\mfr_1 \cong \wedge^2 (\C^{m+1})^*$, and $\mfr = \mfr_0 \oplus \mfr_1$, as given in matrix form by
\begin{gather*}
\includegraphics{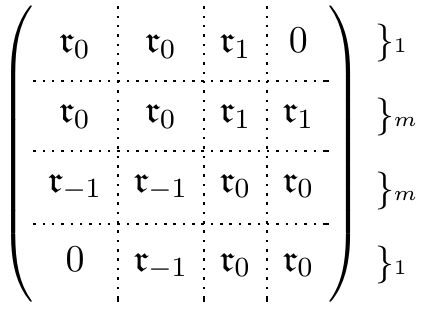}
\end{gather*}
Again, the one-dimensional center of $\mfr_0$ is spanned by a unique grading element with eigenva\-lues~$i$ on~$\mfr_i$. In this case, the tangent space of any point of~$G/R$ is irreducible and linearly isomorphic to~$\mfr_{-1}$.

Unlike in odd dimensions, the twistor space of $\mc{Q}^{2m}$ is not equipped with any canonical \mbox{rank-$m$} distribution. As we shall see in Section~\ref{sec-even-odd}, one requires an additional structure to endow~$\PT_{(2m)}$ with one.

\begin{Proposition}The twistor space $\PT$ of a $2m$-dimensional smooth quadric $\mc{Q}^{2m}$ has dimension $\frac{1}{2}m(m+1)$. Further, for any~$Z$ of~$\PT$, the projective tangent space $\mbf{T}_Z \PT$ intersects $\PT$ in a~$(2m-1)$-dimensional linear subspace of~$\PT$.
\end{Proposition}

Arguments similar to those used in odd dimensions lead to the following proposition.
\begin{Proposition}Let $Z$ and $W$ be two twistors corresponding to $\alpha$-planes $\check{Z}$ and $\check{W}$. Then $\dim ( \check{Z} \cap \check{W} ) \in \{m-2, m \}$ if and only if $Z \in \mbf{T}_W \PT$, or equivalently, $W \in \mbf{T}_Z \PT$. Further, if~$Z$ and~$W$ lie in~$\mbf{T}_\Xi \PT$ for some twistor $\Xi$ in $\PT$, then $\dim ( \check{Z} \cap \check{W}) \in \{m-4, m-2, m \}$.
\end{Proposition}

\subsubsection{From even to odd dimensions}\label{sec-even-odd}
We note that as $\frac{1}{2}(m+1)(m+2)$-dimensional projective complex varieties of $\CP^{2^{m+1}-1}$, the respective twistor spaces $\PT:=\PT_{(2m+1)}$ and $\widetilde{\PT}:=\PT_{(2m+2)}$ of $\mc{Q}^{2m+1}$ and $\mc{Q}^{2m+2}$ are isomorphic. The only geometric structure that distinguishes the former from the latter is the rank-$(m+1)$ canonical distribution. It is shown in~\cite{Doubrov2010} how $\widetilde{\PT}$ can be viewed as a `Fef\/ferman space' over $\PT$ -- in fact, this reference deals with a more general, curved, setting. Here, we explain how the canonical distribution on $\PT$ arises as one `descends' from $\widetilde{\PT}$ to $\PT$.

Let $\widetilde{\V}$ be a $(2m+4)$-dimensional oriented complex vector space equipped with a non-degenerate symmetric bilinear form $\tilde{h}_{\mc{A} \mc{B}}$. Denote by $X^\mc{A}$ the standard coordinates on $\widetilde{\V}$. As before, we realise $\mc{Q}^{2m+2}$ as a smooth quadric of $\Pp \widetilde{\V}$ with twistor spaces $\widetilde{\PT}$ and $\widetilde{\PT}'$ induced from the irreducible spinor representations $\widetilde{\Ss}$ and $\widetilde{\Ss}'$ of $(\widetilde{\V}, \tilde{h}_{\mc{A} \mc{B}})$. Now, f\/ix a~unit vector~$U^\mc{A}$ in~$\widetilde{\V}$, so that $\widetilde{\V} = \mathbb{U} \oplus \V$, where $\mathbb{U} := \langle U^\mc{A} \rangle$, and $\V:=\mathbb{U}^\perp$ is its orthogonal complement in~$\widetilde{\V}$. Then~$\V$ is equipped with a non-degenerate symmetric bilinear form $h_{\mc{A} \mc{B}} := \tilde{h}_{\mc{A} \mc{B}} - U_\mc{A} U_\mc{B}$, and we can realise $\mc{Q}^{2m+1}$ as a smooth quadric of~$\Pp \V$ with twistor space~$\PT$ induced from the irreducible spinor representation~$\Ss$ of~$(\V, h_{\mc{A} \mc{B}})$.

Observe that $U^\mc{A}$ def\/ines two invertible linear maps,
\begin{gather*}
U_{\bm{\upalpha'}}^{\bm{\upbeta}} := U^\mc{A} \widetilde{\Gamma} \ind{_{\mc{A} \bm{\upalpha'}}^{\bm{\upbeta}}} \colon \ \widetilde{\Ss}' \rightarrow \widetilde{\Ss} , \qquad
U_{\bm{\upalpha}}^{\bm{\upbeta'}} := U^\mc{A} \widetilde{\Gamma} \ind{_{\mc{A} \bm{\upalpha}}^{\bm{\upbeta'}}} \colon \ \widetilde{\Ss} \rightarrow \widetilde{\Ss}' ,
\end{gather*}
where $\widetilde{\Gamma} \ind{_{\mc{A} \bm{\upalpha'}}^{\bm{\upbeta}}}$ and $\widetilde{\Gamma} \ind{_{\mc{A} \bm{\upalpha}}^{\bm{\upbeta'}}}$ generate the Clif\/ford algebra $\Cl(\widetilde{\V},\tilde{h}_{\mc{A} \mc{B}})$. These maps allow us to identi\-fy~$\widetilde{\Ss}$ with~$\widetilde{\Ss}'$, and thus $\widetilde{\PT}$ with $\widetilde{\PT}'$. Further, using the Clif\/ford property, it is straightforward to check that $\Gamma \ind{_{\mc{A} \bm{\upalpha}}^{\bm{\upbeta}}} := h_\mc{A}^\mc{B} \widetilde{\Gamma} \ind{_{\mc{B} \bm{\upalpha}}^{\bm{\upgamma'}}} U_{\bm{\gamma'}}^{\bm{\upbeta}} = - h_\mc{A}^\mc{B} U_{\bm{\upalpha}}^{\bm{\upgamma'}} \widetilde{\Gamma} \ind{_{\mc{B} \bm{\upgamma'}}^{\bm{\upbeta}}} = U^\mc{B} \widetilde{\Gamma} \ind{_{\mc{A} \mc{B} \bm{\upalpha}}^{\bm{\upbeta}}}$ generate the Clif\/ford algebra $\Cl(\V, h_{\mc{A} \mc{B}})$. More generally, the relation between the spanning elements of $\Cl(\V, h_{\mc{A} \mc{B}})$ and those of $\Cl(\widetilde{\V}, \widetilde{h}_{\mc{A} \mc{B}})$ is given by
\begin{gather}\label{eq-Cl-even-odd}
\Gamma \ind*{^{(k)}_{\mc{A}_1 \ldots \mc{A}_k} _{\bm{\upalpha \upbeta}}} = h_{\mc{A}_1}^{\mc{B}_1} \cdots h_{\mc{A}_k}^{\mc{B}_k} \widetilde{\Gamma} \ind*{^{(k)}_{\mc{B}_1 \ldots \mc{B}_k} _{\bm{\upalpha \upbeta}}} , \qquad k \equiv m+2 \pmod{2} , \\
\Gamma \ind*{^{(k)}_{\mc{A}_1 \ldots \mc{A}_k} _{\bm{\upalpha \upbeta}}} = U^\mc{B} \widetilde{\Gamma} \ind*{^{(k)}_{\mc{A}_1 \cdots \mc{A}_k \mc{B}} _{\bm{\upalpha \upbeta}}} = (-1)^k h_{\mc{A}_1}^{\mc{B}_1} \cdots h_{\mc{A}_k}^{\mc{B}_k} U _{\bm{\upalpha}}^{\bm{\upgamma'}} \widetilde{\Gamma} \ind*{^{(k)}_{\mc{B}_1 \cdots \mc{B}_k} _{\bm{\upgamma' \upbeta}}} , \qquad k \equiv m+1 \pmod{2} .\nonumber
\end{gather}
If we now introduce homogeneous coordinates $[Z^{\bm{\upalpha}}]$ on $\Pp \widetilde{\Ss}$, we can identify the twistor space $\PT$ equipped with its canonical distribution with the twistor space $\widetilde{\PT}$, as can be seen by inspection of~\eqref{eq-pure-spinor-odd} and~\eqref{eq-pure-spinor-even}. Note that we could have played the same game with~$\widetilde{\PT}'$.

Let us interpret this more geometrically. Clearly, the embedding of $\mc{Q}^{2m+1}$ into $\mc{Q}^{2m+2}$ arises as the intersection of the hyperplane $U_\mc{A} X^\mc{A} = 0$ in $\Pp \widetilde{\V}$ with the cone over $\mc{Q}^{2m+2}$. A $\gamma$-plane of~$\mc{Q}^{2m+1}$ then arises as the intersection of an $\alpha$-plane of $\mc{Q}^{2m+2}$ with $\mc{Q}^{2m+1}$, and similarly for $\beta$-planes. An $\alpha$-plane $\check{Z}$ and a $\beta$-plane $\check{W}$ def\/ine the same $\gamma$-plane if and only if their corresponding twistors satisfy $Z^{\bm{\upalpha}} = U_{\bm{\upbeta'}}^{\bm{\upalpha}} W^{\bm{\upbeta'}}$. In particular, such a pair must intersect maximally, i.e., in an $m$-plane in $\mc{Q}^{2m+2}$. This much is already outlined in the appendix of~\cite{Penrose1986}.

Finally, we can see how the canonical distribution $\mathrm{D}$ on $\PT$ arises geometrically from~$\widetilde{\PT}$ and~$\widetilde{\PT}'$. Fix a point $[\Xi^{\bm{\upalpha}}]$ in $\widetilde{\PT}$. This represents an $\alpha$-plane $\check{\Xi}$ in~$\mc{Q}^{2m+2}$, and so a $\gamma$-plane in~$\mc{Q}^{2m+1}$, which also corresponds to the unique $\beta$-plane with associated primed twistor $[U_{\bm{\upbeta}}^{\bm{\upalpha'}} \Xi^{\bm{\upbeta}}]$ in~$\widetilde{\PT}'$. We claim that the $\beta$-planes intersecting $\check{\Xi}$ maximally are in one-to-one correspondence with the points of $\mbf{D}_\Xi$. To see this, let $[Z^{\bm{\upalpha}}]$ be a point in $\mathbf{T}_\Xi \widetilde{\PT} \subset \Pp \widetilde{\Ss}$ so that
\begin{gather*}
\widetilde{\Gamma} \ind*{^{(k)}_{\bm{\upalpha}}_{\bm{\upbeta}}} Z \ind{^{\bm{\upalpha}}} \Xi \ind{^{\bm{\upbeta}}} = 0 , \qquad \text{for all}\quad k<m , \quad k \equiv m \pmod{2}.
\end{gather*}
We can then conclude $[Z^{\bm{\upalpha}}] \in \mbf{T}_\Xi \PT$ by virtue of~\eqref{eq-proj-tgt-PT} and~\eqref{eq-Cl-even-odd} as expected. Now, consider the set of all $\beta$-planes intersecting $\check{\Xi}$ maximally: these correspond to all primed twistors $[W^{\bm{\upalpha'}}] \in \widetilde{\PT}'$ satisfying
\begin{gather*}
\widetilde{\Gamma} \ind*{^{(k)}_{\bm{\upalpha'}}_{\bm{\upbeta}}} W \ind{^{\bm{\upalpha'}}} \Xi \ind{^{\bm{\upbeta}}} = 0 , \qquad \text{for all} \quad k<m+1 , \quad k \equiv m+1 \pmod{2}.
\end{gather*}
Identifying $\beta$-planes and $\alpha$-planes on $\mc{Q}^{2m+1}$, i.e., setting $Z^{\bm{\upalpha}} = U_{\bm{\upbeta'}}^{\bm{\upalpha}} W^{\bm{\upbeta'}}$, and using~\eqref{eq-Cl-even-odd} again precisely yield that $[Z^{\bm{\upalpha}}] \in \mbf{D}_\Xi$ by virtue of~\eqref{eq-lin-subsp}.

\subsection{Correspondence space} \label{sec-correspondence-space}
We now formalise the correspondence between $\mc{Q}^n$ and $\PT$.
\subsubsection{Odd dimensions}
Assume $n=2m+1$.
\begin{Definition} The \emph{correspondence space $\F$} of $\mc{Q}^n$ and $\PT$ is the projective complex subvariety of $\mc{Q}^n \times \PT$ def\/ined as the set of points $([X^\mc{A}] , [Z^{\bm{\upalpha}}])$ satisfying \emph{the incidence relation}
\begin{gather}\label{eq-incidence_relation-twistor}
X ^\mc{A} Z \ind*{_{\mc{A}}^{\bm{\upbeta}}} = 0 ,
\end{gather}
where $Z \ind*{_{\mc{A}}^{\bm{\upbeta}}} := \Gamma \ind{_{\mc{A}}_{\bm{\upalpha}}^{\bm{\upbeta}}} Z^{\bm{\upalpha}}$.
\end{Definition}
The usual way of understanding the twistor correspondence is by means of the double f\/ibration
\begin{gather*}
\xymatrix{& \F \ar[dr]^\mu \ar[dl]_\nu & \\
\mc{Q}^n & & \PT, }
\end{gather*}
where $\mu$ and $\nu$ denote the usual projections of maximal rank.

Clearly, since, by def\/inition, a twistor $[Z^{\bm{\upalpha}}]$ in $\PT$ corresponds to a~$\gamma$-plane of $\mc{Q}^n$, namely the set of points $[X^\mc{A}]$ in $\mc{Q}^n$ satisfying~\eqref{eq-incidence_relation-twistor}, we see that each f\/iber of~$\mu$ is isomorphic to~$\CP^m$.

Now, a point $x$ of $\mc{Q}^n$ is sent to a compact complex submanifold $\hat{x}$ of $\PT$ isomorphic to the f\/iber $\F_x$ of $\F$ over $x$, and similarly, a subset $\mc{U}$ of $\mc{Q}^n$ will correspond to a subset $\widehat{\mc{U}}$ of $\PT$ swept out by those complex submanifolds $\{ \hat{x} \}$ parametrised by the points $x \in \mc{U}$, i.e.,
\begin{gather*}
x \in \mc{Q}^n \mapsto \F_x := \nu^{-1}(x) \mapsto \hat{x} := \mu(\F_x) , \qquad
\mc{U} \subset \mc{Q}^n \mapsto \F_\mc{U} := \!\bigcup_{x \in \mc{U}}\! \nu^{-1}(x) \mapsto \widehat{\mc{U}} := \!\bigcup_{x \in \mc{U}} \! \mu(\F_x) .
\end{gather*}

To describe $\hat{x}$, it is enough to describe the f\/iber $\F_x$. By def\/inition, this is the set of all $\gamma$-planes incident on $x$. If $\check{Z}$ is a $\gamma$-plane incident on $x$, the intersection~$\check{Z} \cap \Tgt_x \mc{Q}^n$ is an $m$-dimensional subspace totally null with respect to the bilinear form on $\Tgt_x \mc{Q}^n \cong \C \E^n$, which we shall also refer to a~$\gamma$-plane. This descends to a $\gamma$-plane in $\mc{Q}^{2m-1}$ viewed as the projectivisation of the null cone through~$x$. Thus, $\hat{x} \cong \F_x$ is isomorphic to the $\frac{1}{2}m(m+1)$-dimensional twistor space~$\PT_{(2m-1)}$ of~$\mc{Q}^{2m-1}$.

We can get a little more information about $\F$ by viewing it as the homogeneous space $G/Q$ where $Q:=P \cap R$ is the intersection of $P$, the stabiliser of a null line in $\V$, and $R$ the stabiliser of a totally null $(m+1)$-plane containing that line. The Lie algebra $\mfq$ of $Q$ induces a~$|3|$-grading on~$\g$, i.e., $\g = \mfq_{-3} \oplus \mfq_{-2} \oplus \mfq_{-1} \oplus \mfq_0 \oplus \mfq_1 \oplus \mfq_2 \oplus \mfq_3$, where $\mfq = \mfq_0 \oplus \mfq_1 \oplus \mfq_2 \oplus \mfq_3$. For convenience, we split $\mfq_{\pm1}$ and $\mfq_{\pm2}$ further as $\mfq_{\pm1} = \mfq_{\pm1}^E \oplus \mfq_{\pm1}^F$ and $\mfq_{\pm2} = \mfq_{\pm2}^E \oplus \mfq_{\pm2}^F$. Also, $\mfq_0 \cong \glie(m,\C) \oplus \C$, $\mfq_{-1}^E \cong \C^m$, $\mfq_{-1}^F \cong (\C^m)^*$, $\mfq_{-2}^E \cong \C$, $\mfq_{-2}^F \cong \wedge^2 \C^m$ and $\mfq_{-3} \cong (\C^m)^*$ with $(\mfq_i)^* \cong \mfq_{-i}$. In matrix form, $\g$ reads as
\begin{gather*}
\includegraphics{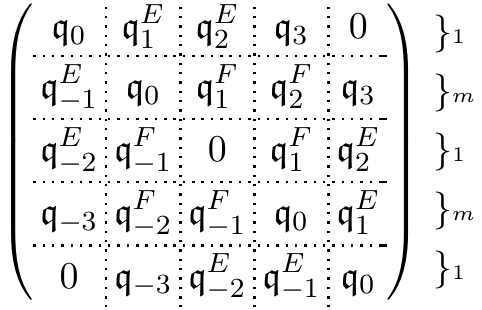}
\end{gather*}
These modules satisfy the commutation relations $[\mfq_i,\mfq_j] \subset \mfq_{i+j}$ where $\mfq_i=\{ 0 \}$ for $|i|>3$. More precisely, the action of $\mfq_1$ on these modules, carefully distinguishing $\mfq_1^E$ and $\mfq_1^F$, can be recorded in the form of a diagram:
\begin{gather*}
\xymatrix @C=1pc @R=1pc{ & & & & \mfq_{-1}^E & & \\
& & & \mfq_{-2}^E \ar[dr]^{\mfq_1^E} \ar[ur]^{\mfq_1^F} & & & \\
& & \mfq_{-3} \ar[dr]^{\mfq_1^E} \ar[ur]^{\mfq_1^F} & & \mfq_{-1}^F & & \\
& & & \mfq_{-2}^F \ar[ur]^{\mfq_1^F} & & & \\
\mfp_{-1} \ar@{<.}[uurr] & & & & \mfr_{-2} \ar[rr]^{\mfr_1} \ar@{<.}[ul] & & \mfr_{-1} \ar@{<.}[uull]
}
\end{gather*}
where the dotted arrows give the relations between $\mfq_0$-modules, and $\mfp_0$- and $\mfr_0$-modules. Invariance follows from the inclusions $\mfq_1^E \subset \mfr_0$, $\mfq_1^F \subset \mfp_0$, $\mfq_1^E \subset \mfp_1$ and $\mfq_1^F \subset \mfr_1$.

Beside the f\/iltration of vector subbundles of $\Tgt \F$ determined by the grading on $\g$, we distinguish three $Q$-invariant distributions of interest on $\F$:
\begin{itemize}\itemsep=0pt
\item the rank-$\frac{1}{2}m(m+1)$ distribution $\Tgt^{-2}_F \F$ corresponding to $\mfq_{-2}^F \oplus \mfq_{-1}^F$. It is integrable and tangent to the f\/ibers of $\nu:G/Q \rightarrow G/P$, each isomorphic to the homogeneous space $P/Q$. This follows from the relations $[ \mfq_{-1}^F , \mfq_{-1}^F ] \subset \mfq_{-2}^F$, $[ \mfq_{-1}^F , \mfq_{-2}^F ] = 0$, and $[ \mfq_{-2}^F , \mfq_{-2}^F ] = 0$, and the fact that the kernel of the projection $\g/\mfq \rightarrow \g/\mfp$ is precisely $\mfq_{-2}^F \oplus \mfq_{-1}^F \cong \mfp/\mfq$. In fact, since $[ \mfq_{-1}^F , \mfq_{-1}^F ] \subset \mfq_{-2}^F$, each f\/iber is itself equipped with the canonical distribution on $\PT_{(2m-1)}$.

\item the rank-$m$ distribution $\Tgt^{-1}_E \F$ corresponding to $\mfq_{-1}^E$. It is integrable and tangent to the f\/ibers of $\mu:G/Q \rightarrow G/R$, each isomorphic to the homogeneous space $R/Q$. This follows from the relations $[ \mfq_{-1}^E , \mfq_{-1}^E ] = 0$ and the fact that the kernel of the projection $\g/\mfq \rightarrow \g/\mfr$ is precisely $\mfq_{-1}^E \cong \mfr/\mfq$.

\item the rank-$(2m+1)$ distribution $\Tgt^{-2}_E \F$ corresponding to $\mfq_{-2}^E \oplus \mfq_{-1}^F \oplus \mfq_{-1}^E$. It is non-integrable and bracket generates $\Tgt \F$ since we have $[ \mfq_{-1}^E , \mfq_{-1}^F ] \subset \mfq_{-2}^E$, $[ \mfq_{-1}^E , \mfq_{-2}^E ] = 0$, $[ \mfq_{-1}^E , \mfq_{-2}^F ] \subset \mfq_{-3}$, $[ \mfq_{-1}^F , \mfq_{-2}^E ] \subset \mfq_{-3}$. Further, the quotient $\Tgt^{-2}_E \F / \Tgt^{-1}_E \F$ descends to the canonical distribution $\Tgt^{-1} \PT$.
\end{itemize}

{\bf The twistor space and correspondence space of $\C \E^{2m+1}$.}
At this stage, we introduce a~splitting \eqref{eq-split-V} of $\V$, and as before denote by $\mr{X}^\mc{A}$, $\mr{Y}^\mc{A}$ and $\mr{Z}^\mc{A}_a$ vectors in $\V_1$, $\V_{-1}$ and $\V_0$ respectively. There is an induced splitting~\eqref{eq-S->S1/2-odd} of~$\Ss$, and we shall accordingly split the homogeneous twistor coordinates as $Z^{\bm{\upalpha}}=(\omega^{\mbf{A}} , \pi^{\mbf{A}})$, or, using the injectors, as
\begin{gather}\label{eq-Zompi}
Z^{\bm{\upalpha}} = \mr{I}^{\bm{\upalpha}}_{\mbf{A}} \omega^\mbf{A} + \mr{O}^{\bm{\upalpha}}_{\mbf{A}} \pi^\mbf{A} .
\end{gather}
Needless to say that Cartan's theory of spinors applies to $\Ss_{-\frac{1}{2}}$ and $\Ss_{\frac{1}{2}}$ in the obvious way and notation, as we have done in Section~\ref{sec-tw-sp}. In particular, a spinor $\pi^{\mbf{A}}$ is pure if and only if the kernel of the map $\pi _a^{\mbf{A}} := \pi^{\mbf{B}} \gamma \ind{_a_{\mbf{B}}^{\mbf{A}}}$ is of maximal dimension $m$, and so on. The purity condition on $Z^{\bm{\upalpha}}$ can then be re-expressed as follows.

\begin{Lemma}\label{lem-pure-ompi}
Let $Z^{\bm{\upalpha}}=(\omega^{\mbf{A}},\pi^\mbf{A})$ be a non-zero spinor in $\Ss \cong \Ss_{-\frac{1}{2}} \oplus \Ss_{\frac{1}{2}}$. Then $Z^{\bm{\upalpha}}$ is pure, i.e., satisfies~\eqref{eq-pure-spinor-odd}, if and only if $\omega^{\mbf{A}}$ and $\pi^{\mbf{A}}$ satisfy
\begin{subequations}\label{eq-Z-pure}
\begin{gather}
\gamma \ind*{^{(k)}_{\mbf{AB}}} \pi^{\mbf{A}} \pi^{\mbf{B}} = 0, \qquad \text{for all} \quad k < m, \quad k \equiv m+1, m \pmod{4}, \label{eq-pi-pure} \\
\gamma \ind*{^{(k)}_{\mbf{AB}}} \omega^{\mbf{A}} \omega^{\mbf{B}} = 0, \qquad \text{for all} \quad k < m, \quad k \equiv m+1, m \pmod{4},\label{eq-om-pure} \\
\gamma \ind*{^{(k)}_{\mbf{AB}}} \omega^{\mbf{A}} \pi^{\mbf{B}} = 0, \qquad \text{for all} \quad k < m-1. \label{eq-T-1PT}
\end{gather}
\end{subequations}
\end{Lemma}

\begin{proof}
This is a direct computation using \eqref{eq-twistor-pure-cond_odd}, \eqref{eq-Gam2gam-odd} and \eqref{eq-Zompi}. Writing $\pi _a^{\mbf{A}} := \pi^{\mbf{B}} \gamma \ind{_a_{\mbf{B}}^{\mbf{A}}}$ and $\omega _a^{\mbf{A}} := \omega^{\mbf{B}} \gamma \ind{_a_{\mbf{B}}^{\mbf{A}}}$, we f\/ind
\begin{gather*}
\pi \ind{^a^{\mbf{A}}} \pi \ind*{_a^{\mbf{B}}} + \pi^{\mbf{A}} \pi^{\mbf{B}} = 0 , \qquad \omega \ind{^a^{\mbf{A}}} \omega \ind*{_a^{\mbf{B}}} + \omega^{\mbf{A}} \omega^{\mbf{B}} = 0 , \qquad \pi \ind{^a^{\mbf{A}}} \omega \ind*{_a^{\mbf{B}}} - \pi^{\mbf{A}} \omega^{\mbf{B}} + 2 \omega^{\mbf{A}} \pi^{\mbf{B}} = 0 ,
\end{gather*}
which are equivalent to \eqref{eq-pi-pure}, \eqref{eq-om-pure} and \eqref{eq-T-1PT} respectively \cite{Taghavi-Chabert2013}.
\end{proof}

By Cartan's theory of spinors, condition \eqref{eq-pi-pure} is equivalent to $\pi^{\mbf{A}}$ being pure provided it is non-zero, and similarly for condition \eqref{eq-om-pure} and $\omega^{\mbf{A}}$. Condition \eqref{eq-T-1PT} is equivalent to the $\gamma$-planes of $\pi^{\mbf{A}}$ and $\omega^{\mbf{A}}$ intersecting in an $m$- or $(m-1)$-plane in $\V_0$ provided these are non-zero.

\begin{Remark}
The annihilator \eqref{eq-hi-contact} of the canonical distribution of $\PT$ can be re-expressed as
\begin{gather}\label{eq-hi-contact-ompi}
\begin{split}
& \bm{\alpha}^{\mbf{AB}}_{(\omega,\omega)} := \omega^{a\mbf{A}} \dd \omega_a^{\mbf{B}} + 2 \omega^{\mbf{B}} \dd \omega^{\mbf{A}} - \omega^{\mbf{A}} \dd \omega^{\mbf{B}} , \\
& \bm{\alpha}^{\mbf{AB}}_{(\pi,\pi)} := \pi^{a\mbf{A}} \dd \pi_a^{\mbf{B}} + 2 \pi^{\mbf{B}} \dd \pi^{\mbf{A}} - \pi^{\mbf{A}} \dd \pi^{\mbf{B}} , \\
& \bm{\alpha}^{\mbf{AB}}_{(\omega,\pi)} := \omega^{a\mbf{A}} \dd \pi_a^{\mbf{B}} + \omega^{\mbf{A}} \dd \pi^{\mbf{B}} + 4 \pi^{[\mbf{A}} \dd \omega^{\mbf{B}]} , \\
& \bm{\alpha}^{\mbf{AB}}_{(\pi,\omega)} := \pi^{a\mbf{A}} \dd \omega_a^{\mbf{B}} + \pi^{\mbf{A}} \dd \omega^{\mbf{B}} + 4 \omega^{[\mbf{A}} \dd \pi^{\mbf{B}]} ,
\end{split}
\end{gather}
where we have used \eqref{eq-Zompi} and \eqref{eq-Gam2gam-odd}, and it is understood that $\omega^{\mbf{A}}$ and $\pi^\mbf{A}$ satisfy~\eqref{eq-Z-pure}.
\end{Remark}

The twistor correspondence associates to the point $\infty$ in $\mc{Q}^n$, with coordinates $[\mr{Y}^\mc{A}]$, a complex submanifold $\widehat{\infty}$ of $\PT$ def\/ined by the locus $\mr{Y}^\mc{A} Z_\mc{A}^{\bm{\upalpha}} = 0$ in $\PT$, i.e.,
\begin{gather*}
\infty \in \mc{Q}^n \mapsto \F_\infty := \nu^{-1}(\infty) \mapsto \widehat{\infty} := \mu(\F_\infty) = \mu \circ \nu^{-1} (\infty).
\end{gather*}
Points of $\widehat{\infty}$ are parametrised by $[\omega^{\mbf{A}},0]$. Since removing $\infty$ from $\mc{Q}^n$ yields complex Euclidean space~$\C \E^n$, we accordingly remove $\widehat{\infty}$ to obtain the twistor space $\PT \setminus \{ \widehat{\infty} \} = \mu \circ \nu^{-1} ( \C \E^n )$ of~$\C \E^n$. This will be denoted by $\PT_{\setminus \widehat{\infty}}$. This region of twistor space is parametrised by $\{[\omega^{\mbf{A}},\pi^{\mbf{A}}] \colon$ \mbox{$\pi^{\mbf{A}} \neq 0\}$}.

The correspondence space of $\C \E^n$ will be denoted $\F_{\C\E^n}$, and is parametrised by the coordinates $(x^a , [ \pi^{\mbf{A}}])$, where $\{ x^a \}$ are the f\/lat standard coordinates on $\C\E^n$ and $[ \pi^{\mbf{A}}]$ are homogeneous pure spinor coordinates on the f\/ibers of $\F$. These parametrise the $\gamma$-planes of the tangent space $\Tgt_x \C \E^n$ at a point $x$ in $\C \E^n$, and are related to $[\omega^\mbf{A},\pi^\mbf{B}]$ by means of the incidence relation \eqref{eq-incidence_relation-twistor}
\begin{gather}
 \omega^{\mbf{A}} = \tfrac{1}{\sqrt{2}} x^a \pi_a^{\mbf{A}} , \label{eq-incidence_relation-spinor_odd}
\end{gather}
which can be obtained from \eqref{eq-Minkowski-emb}, \eqref{eq-Gam2gam-odd} and \eqref{eq-Zompi}. Indeed, the $\gamma$-plane def\/ined by $[\pi^{\mbf{A}}]$ through the origin is given by the locus $\frac{1}{\sqrt{2}} x^a \pi_a^{\mbf{A}}$, so that the $\gamma$-plane def\/ined by $[\pi^{\mbf{A}}]$ through any other point $\mr{x}^a$ is given by \eqref{eq-incidence_relation-spinor_odd} with $\omega^{\mbf{A}} = \frac{1}{\sqrt{2}} \mr{x}^a \pi_a^{\mbf{A}}$.

\begin{Remark}By \eqref{eq-incidence_relation-spinor_odd} and \eqref{eq-pi-pure}, for a holomorphic function $f$ on $\F$ to descend to $\PT$, it must be annihilated by the dif\/ferential operator $\pi^{[\mbf{A}} \pi \ind{^a^{\mbf{B}]}} \nabla_a$.
\end{Remark}

\subsubsection{Even dimensions}
The double f\/ibration picture in dimension $n=2m$ is very similar to the odd-dimensional case, and we only summarise the discussion here.

We realise $\F$ as a homogeneous space $G/Q$. Here, the Lie algebra $\mfq$ of $Q$ induces a $|2|$-grading $\g = \mfq_{-2} \oplus \mfq_{-1} \oplus \mfq_0 \oplus \mfq_1 \oplus \mfq_2$ on $\g$, where $\mfq = \mfq_0 \oplus \mfq_1 \oplus \mfq_2$. We split $\mfq_{\pm1}$ further as $\mfq_{\pm1} = \mfq_{\pm1}^E \oplus \mfq_{\pm1}^F$, and we have $\mfq_0 \cong \glie(m,\C) \oplus \C$, $\mfq_{-1}^E \cong \C^m$, $\mfq_{-1}^F \cong \wedge^2 \C^m$ and $\mfq_{-2} \cong (\C^m)^*$ with $(\mfq_i)^* \cong \mfq_{-i}$. The action of $\mfq_1$ on these $\mfq_0$-modules is recorded below together with the matrix form of the splitting:
\begin{gather*}
\includegraphics{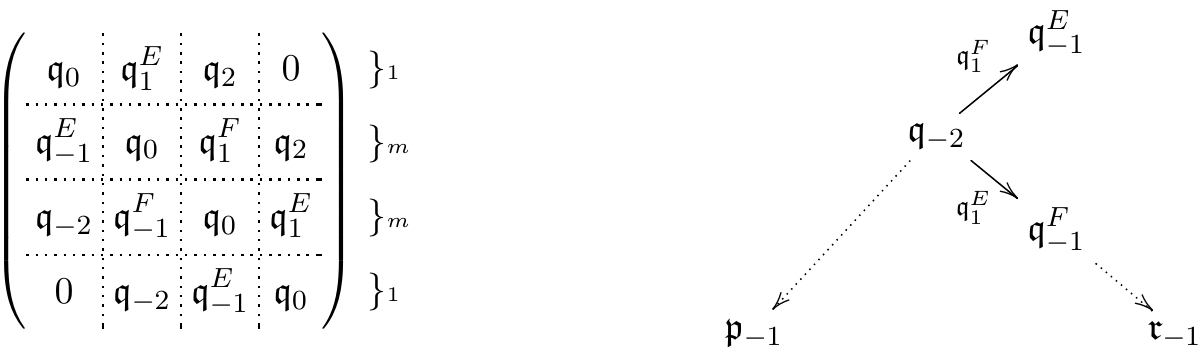}
\end{gather*}

The modules $\mfq_{-1}^F$ and $\mfq_{-1}^E$ give rise to two integrable $Q$-invariant distributions $\Tgt^{-1}_F \F$ and $\Tgt^{-1}_E \F$ on $\F$ of rank $\frac{1}{2}m(m-1)$ and $m$ respectively, and tangent to the f\/ibers of $G/Q \rightarrow G/P$ and $G/Q \rightarrow G/R$ respectively.

{\bf The twistor space and correspondence space of $\C \E^{2m}$.}
The even-dimensional analogue of Lemma~\ref{lem-pure-ompi} is recorded below.
\begin{Lemma}\label{lem-pure-ompi-even}
Let $Z^{\bm{\upalpha}} = ( \omega^\mbf{A} , \pi^{\mbf{A'}})$ be a spinor in $\Ss \cong \Ss_{-\frac{1}{2}} \oplus \Ss_{\frac{1}{2}}'$. Then $Z^{\bm{\upalpha}}$ is pure if and only if $\omega^{\mbf{A}}$ and $\pi^{\mbf{A'}}$ satisfy
\begin{subequations}\label{eq-Z-pure-even}
\begin{alignat}{3}
&\gamma \ind*{^{(k)}_{\mbf{A'B'}}} \pi^{\mbf{A'}} \pi^{\mbf{B'}} = 0, \qquad &&\text{for all}\quad k < m , \quad k \equiv m \pmod{4}, & \label{eq-pi-pure-even}\\
&\gamma \ind*{^{(k)}_{\mbf{AB}}} \omega^{\mbf{A}} \omega^{\mbf{B}} = 0, \qquad &&\text{for all}\quad k < m, \quad k \equiv m \pmod{4}, & \label{eq-om-pure-even}\\
&\gamma \ind*{^{(k)}_{\mbf{AB'}}}\omega^{\mbf{A}} \pi^{\mbf{B'}} = 0, \qquad &&\text{for all}\quad k < m-1, \quad k \equiv m-1 \pmod{2}. &\label{eq-T-1PT-even}
\end{alignat}
\end{subequations}
\end{Lemma}
Conditions \eqref{eq-pi-pure-even}, \eqref{eq-om-pure-even} and \eqref{eq-T-1PT-even} can equivalently be expressed as
\begin{gather*}
\pi \ind{^a^{\mbf{A}}} \pi \ind*{_a^{\mbf{B}}} = 0 ,\qquad \omega \ind{^a^{\mbf{A'}}} \omega \ind*{_a^{\mbf{B'}}} = 0 , \qquad \pi \ind{^a^{\mbf{A}}} \omega \ind*{_a^{\mbf{B'}}} + 2 \omega^{\mbf{A}} \pi^{\mbf{B'}} = 0 ,
\end{gather*}
respectively. By Cartan's theory of spinors, condition \eqref{eq-pi-pure-even} is equivalent to $\pi^{\mbf{A'}}$ being pure provided it is non-zero, and similarly for conditions \eqref{eq-om-pure-even} for $\omega^{\mbf{A}}$. Condition~\eqref{eq-T-1PT-even} is equivalent to the $\alpha$-plane of $\pi^{\mbf{A'}}$ and the $\beta$-plane of $\omega^{\mbf{A}}$ intersecting in an $(m-1)$-plane in~$\V_0$ provided these are non-zero.

Just as in the odd-dimensional case, the twistor space of $\C \E^{2m}$ is obtained by removing the $\frac{1}{2}m(m-1)$-dimensional complex submanifold $\widehat{\infty}$ corresponding to $\infty$ on $\mc{Q}^{2m}$ from $\PT$. We can use $[\pi^{\mbf{A'}}]$ as homogeneous coordinates on the f\/ibers of $\F_{\C\E^n}$, and the incidence relation \eqref{eq-incidence_relation-twistor} can be expressed as $\omega^{\bm{A}} = \frac{1}{\sqrt{2}} x^a \pi_a^{\mbf{A}}$.

\subsection[Co-gamma-planes and mini-twistor space]{Co-$\boldsymbol{\gamma}$-planes and mini-twistor space}\label{sec-co-gam}
In odd dimensions, there is an additional geometric object of interest.
\begin{Definition}\label{defn-co-null-plane}
A \emph{co-$\gamma$-plane} is an $(m+1)$-dimensional af\/f\/ine subspace of $\C\E^{2m+1}$ with the property that the orthogonal complement of its tangent space at any of its point is totally null with respect to the metric.

The space of all co-$\gamma$-planes in $\C\E^{2m+1}$ is called the \emph{mini-twistor space} of $\C\E^{2m+1}$, and is denoted~$\M\T$.
\end{Definition}

Viewed as a vector subspace of $\Tgt_x \C \E^n \cong \C \E^n$, a co-$\gamma$-plane through a point $x$ in $\C \E^n$ is the orthogonal complement of a $\gamma$-plane through $x$. Consider a co-$\gamma$-plane through the origin, and let $[\pi^{\mbf{A}}]$ be a projective pure spinor associated to the $\gamma$-plane orthogonal to it. Then, it is easy to check that this co-$\gamma$-plane consists of the set of points $x^a$ satisfying $t \pi^{\mbf{A}} = \frac{1}{\sqrt{2}} x^a \pi_a^{\mbf{A}}$ where $t \in \C$ with $x^a x_a = -2 t^2$. Shifting the origin to $\mr{x}^a$ say, a point in a co-$\gamma$-plane containing $\mr{x}^a$ now satisf\/ies $\omega^{\mbf{A}} + \pi^{\mbf{A}} t = \frac{1}{\sqrt{2}} x^a \pi_a^{\mbf{A}}$ for some $t \in \C$, and where $\omega^{\mbf{A}} := \frac{1}{\sqrt{2}} \mr{x}^a \pi_a^{\mbf{A}}$. Thus, a co-$\gamma$-plane through $\mr{x}^a$ consists of the set of points satisfying the incidence relation
\begin{gather}\label{eq-incidence-co-null}
\omega^{[\mbf{A}} \pi^{\mbf{B}]} = \tfrac{1}{\sqrt{2}} x^a \pi_a^{[\mbf{A}} \pi^{\mbf{B}]} ,
\end{gather}
where $[\pi^{\mbf{C}}]$ is a projective pure spinor and $\omega^{\mbf{A}} := \frac{1}{\sqrt{2}} \mr{x}^a \pi_a^{\mbf{A}}$. In particular, a co-$\gamma$-plane consists of a $1$-parameter family of $\gamma$-planes, and thus corresponds to the curve
\begin{gather}\label{eq-MT-PT-curve}
\C \ni t \mapsto \big[\omega^{\mbf{A}} + \pi^{\mbf{A}} t , \pi^{\mbf{A}} \big] \in \PT_{\setminus \widehat{\infty}} .
\end{gather}

The relation between $\M\T$ and $\PT_{\setminus \widehat{\infty}}$ can be made precise by involving our choice of `inf\/ini\-ty'~$[\mr{Y}^\mc{A}]$ to def\/ine $\C \E^n$. Let us write $(\mr{Y} \cdot Z)^{\bm{\upalpha}} := \mr{Y}^\mc{A} \Gamma \ind{_{\mc{A}}_{\bm{\upbeta}}^{\bm{\upalpha}}} Z^{\bm{\upbeta}}$. We can then def\/ine the vector f\/ield
\begin{gather}\label{eq-Infinity-Vec-field}
\bm{Y} := - \frac{\ii}{2} \big(\mr{Y} \cdot Z\big)^{\bm{\upalpha}} \parderv{}{Z^{\bm{\upalpha}}} = \frac{\ii}{\sqrt{2}} \pi^{\mbf{A}} \parderv{}{\omega^{\mbf{A}}} ,
\end{gather}
on $\PT_{\setminus \widehat{\infty}}$, the factors having been added for later convenience. It is now pretty clear that the curve~\eqref{eq-MT-PT-curve} is an integral curve of the vector f\/ield \eqref{eq-Infinity-Vec-field} passing through the point $[\omega^{\mbf{A}},\pi^{\mbf{A}}]$. We therefore conclude
\begin{Lemma}\label{lem-MT-flow}
Mini-twistor space $\M\T$ is the quotient of $\PT_{\setminus \widehat{\infty}}$ by the flow of $\bm{Y}$ defined by~\eqref{eq-Infinity-Vec-field}.
\end{Lemma}

An alternative geometric interpretation can be obtained by introducing weighted homogeneous coordinates on $\M\T$ as follows. Since $\pi^\mbf{A}$ is pure, we can view $[\pi^{\mbf{A}}]$ as homogeneous coordinates on $\PT_{(2m-1)}$. Let $\underline{\omega}_{a_1 \ldots a_{m-1}}$ be an $(m-1)$-form satisfying
\begin{gather}\label{eq-deg-om-MT}
\pi^{a_1 \mbf{A}} \underline{\omega}_{a_1 a_2\ldots a_{m-1}} = 0 , \qquad m>1 .
\end{gather}
Write $[ \underline{\omega}_{a_1 \ldots a_{m-1}} , \pi^{\mbf{A}}]_{2,1}$ for the equivalence class of pairs $( \underline{\omega}_{a_1 \ldots a_{m-1}} , \pi^{\mbf{A}} )$ def\/ined by the relation
\begin{gather*}
\big( \underline{\omega}_{a_1 \ldots a_{m-1}} , \pi^{\mbf{A}}\big ) \sim \big( \lambda^2 \underline{\omega}_{a_1 \ldots a_{m-1}} , \lambda \pi^{\mbf{A}} \big) \qquad \text{for some}\quad \lambda \in \C^*.
\end{gather*}
Then $[ \underline{\omega}_{a_1 \ldots a_{m-1}} , \pi^{\mbf{A}} ]_{2,1}$ constitute weighted homogeneous coordinates on $\M\T$. To see this, we note that for any choice of representative, the condition \eqref{eq-deg-om-MT} is equivalent to
\begin{gather}\label{eq-coord-MT-PT}
\underline{\omega}_{a_1 \ldots a_{m-1}} = \gamma \ind*{^{(m-1)}_{a_1 \ldots a_{m-1}}_{\mbf{A} \mbf{B}}} \pi^{\mbf{A}} \omega^{\mbf{B}}
\end{gather}
for some pure spinor $\omega^{\mbf{A}}$ satisfying~\eqref{eq-om-pure}. Then, the projection of any $[ \omega^{\mbf{A}} , \pi^\mbf{A}]$ in $\PT_{\setminus \widehat{\infty}}$ to $[ \underline{\omega}_{a_1 \ldots a_{m-1}} , \pi^{\mbf{A}} ]_{2,1}$ is independent of the choice of representative of $[ \omega^{\mbf{A}} , \pi^{\mbf{A}} ]$, and further, since~$\pi^{\mbf{A}}$ is pure, i.e., satisf\/ies~\eqref{eq-pi-pure}, sending $\omega^{\mbf{A}}$ to $\omega^{\mbf{A}} + t \pi^{\mbf{A}}$ for any $t \in \C$ leaves~\eqref{eq-coord-MT-PT} unchanged.

With these coordinates, we can rewrite the incidence relation \eqref{eq-incidence-co-null} as
\begin{gather}\label{eq-incidence-co-null-2}
\underline{\omega}_{a_1 \ldots a_{m-1}} = \tfrac{1}{\sqrt{2}} x^a \gamma \ind*{^{(m)}_{a a_1 \ldots a_{m-1}}_{\mbf{A} \mbf{B}}} \pi^{\mbf{A}} \pi^{\mbf{B}} .
\end{gather}

Now, turning to the geometrical interpretation, we f\/ix a point $\pi$ in $\PT_{(2m-1)}$ with a choice of pure spinor $\pi^{\mbf{A}}$. Since $\Tgt^{-1}_\pi \PT_{(2m-1)}$ is a dense open subset of an $m$-dimensional linear subspace of $\PT_{(2m-1)}$ containing $\pi$, we can identify a vector in $\Tgt^{-1}_\pi \PT_{(2m-1)}$ with a point in this subspace, which can be represented by a pure spinor $\omega^{\mbf{A}}$ satisfying \eqref{eq-om-pure}. At this stage, this identif\/ication is valid provided the scale of $\pi^{\mbf{A}}$ is \emph{fixed}. Clearly the origin in $\Tgt^{-1}_\pi \PT_{(2m-1)}$ is $\pi^{\mbf{A}}$ itself, so that $(\omega^{\mbf{A}}, \pi^{\mbf{A}})$ maps injectively to $(\underline{\omega}_{a_1 a_2 \ldots a_{m-1}}, \pi^{\mbf{A}})$. That this map is also surjective follows immediately from \eqref{eq-coord-MT-PT}. Hence, we can conclude

\begin{Proposition}\label{prop-MT-char}
The mini-twistor space $\M\T$ of $\C \E^{2m+1}$ is a $\frac{1}{2}m(m+3)$-dimensional complex manifold isomorphic to the total space of the canonical rank-$m$ distribution $\Tgt^{-1} \PT_{(2m-1)}$ of the twistor space $\PT_{(2m-1)}$ of $\mc{Q}^{2m-1}$.
\end{Proposition}

For clarity, we represent $\M\T$ by means of an extended double f\/ibration
\begin{align*}
\xymatrix{& \F_{\C\E^n} \ar[dr]^{\mu} \ar@{.>}[ddr]_{\eta} \ar[dl]_{\nu} & \\
\C \E^n & & \PT_{\setminus \widehat{\infty}} \ar[d]^{\tau} \\
& & \M\T }
\end{align*}
where $\mu$, $\nu$, $\tau$ and $\eta$ are the usual projections. We shall introduce the following notation for submanifolds of $\M\T$ corresponding to points in $\C\E^n$:
\begin{gather*}
x \in \C \E^n \mapsto \F_x := \nu^{-1}(x) \mapsto \underline{\hat{x}} := \tau (\hat{x}) = \eta (\F_x) , \\
\mc{U} \subset \C \E^n \mapsto \F_\mc{U} := \bigcup_{x \in \mc{U}} \nu^{-1}(x) \mapsto \underline{\widehat{\mc{U}}} := \tau \big(\widehat{\mc{U}}\big) = \eta ( \F_\mc{U} ) .
\end{gather*}

\begin{Remark}
For a holomorphic function on $\F$ to descend to $\M\T$, it must be annihilated by the dif\/ferential operator $\pi \ind{^a^{\mbf{A}}} \nabla_a$.
\end{Remark}

\subsection{Normal bundles}\label{sec-normal}
It will also be convenient to think of the correspondence space as an analytic family $\{ \hat{x} \}$ of compact complex submanifolds of twistor space parametrised by the points $x$ of $\mc{Q}^n$. The way each $\hat{x}$ is embedded in $\PT$ is described by its (holomorphic) \emph{normal bundle} $\Nrm \hat{x}$ in
$\PT$, which is the rank-$(m+1)$ vector bundle def\/ined by the short exact sequence
\begin{gather*}
0 \rightarrow \Tgt \hat{x} \rightarrow \Tgt \left. \PT \right|_{\hat{x}} \rightarrow \Nrm \hat{x} \rightarrow 0 .
\end{gather*}
As we shall see there are some crucial dif\/ference between the odd- and even-dimensional cases.

\subsubsection{Odd dimensions}\label{sec-important}
Assume $n=2m+1$. We f\/irst note that the canonical distribution $\mathrm{D}$ on $\PT$ def\/ines a subbundle $\mathrm{D} |_{\hat{x}} + \Tgt \hat{x}$ of $\Tgt \PT |_{\hat{x}}$ containing $\Tgt \hat{x}$. How much of this subbundle descends to $\Nrm \hat{x}$ is answered by the following lemma.
\begin{Lemma}\label{lem-D-Nx}
Let $x$ be a point in $\mc{Q}^{2m+1}$. Then, for any $Z \in \hat{x} \subset \PT$, the intersection of $\mathrm{D}_Z$ and $\Tgt_Z \hat{x}$ has dimension $m$. In particular, $\hat{x}$ is equipped with a maximally non-integrable rank-$m$ distribution $\Tgt^{-1} \hat{x} := \mathrm{D} |_{\hat{x}} \cap \Tgt \hat{x}$. Further, there is a distinguished line subbundle of the normal bundle $\mathrm{N} \hat{x}$ of $\hat{x}$ given by $\Nrm^{-1} \hat{x} := ( \mathrm{D} |_{\hat{x}} + \Tgt \hat{x})/\Tgt \hat{x}$.
\end{Lemma}

\begin{proof}
Denote by $[X^\mc{A}]$ the homogeneous coordinates of $x \in \mc{Q}^{2m+1}$, and let $\Xi \in \hat{x} \subset \PT$ so that $X^\mc{A} \Xi_\mc{A}^{\bm{\upalpha}} =0$. Then, by Lemma~\ref{lem-Z-orbit-trc} and Proposition~\ref{prop-canonical-distribution}, a vector tangent to $\Drm_\Xi$ can be identif\/ied with a point $Z^{\bm{\upalpha}} = \Xi^{\bm{\upalpha}} + \frac{\ii}{2} ( Z_{(-1)} \cdot \Xi )^{\bm{\upalpha}}$ of a dense open subset of $\mbf{D}_\Xi \subset \PT$. Here, $Z_{(-1)} \in \mfr_{-1} \cong \C^{m+1}$ lies in a complement of the stabiliser $\mfr$ of $\Xi$ as explained in Section~\ref{sec-tw-sp}. The condition that this vector is also tangent to $\hat{x}$ is equivalent to $0 = X^\mc{A} Z \ind*{_{\mc{A}}^{\bm{\upalpha}}} = - 2 X_\mc{A} Z^\mc{A}_{(-1)} \Xi^{\bm{\upalpha}}$, by~\eqref{eq-Clifford_twistor_odd}, i.e., $X_\mc{A} Z^\mc{A}_{(-1)} =0$. This gives a single additional algebraic condition on $Z_{(-1)}^\mc{A}$, and thus the intersection of $\mathrm{D}_\Xi$ and $\Tgt_\Xi \hat{x}$ is $m$-dimensional (for a description in af\/f\/ine coordinates, see the end of Appendix~\ref{sec-affine-odd}). This def\/ines a rank-$m$ distribution $\Tgt^{-1} \hat{x} := \mathrm{D} |_{\hat{x}} \cap \Tgt \hat{x}$ on $\hat{x}$. Since~$\mathrm{D}$ is maximally non-integrable, so must be $\Tgt^{-1} \hat{x}$. That the subbundle $\Nrm^{-1} \hat{x} := ( \mathrm{D} |_{\hat{x}} + \Tgt \hat{x})/\Tgt \hat{x}$ of~$\mathrm{N} \hat{x}$ is of rank~$1$ follows from the isomorphism $ \mathrm{D}|_{\hat{x}} /( \mathrm{D} |_{\hat{x}} \cap \Tgt \hat{x} ) \cong ( \mathrm{D} |_{\hat{x}} + \Tgt \hat{x} )/\Tgt \hat{x}$.
\end{proof}

That $\hat{x}$ is endowed with a canonical rank-$m$ distribution comes as no surprise since each $\hat{x}$ is isomorphic to the generalised f\/lag manifold $P/Q \cong \PT_{(2m-1)}$.

As explained in \cite{Kodaira1962}, the tangent space at a point $x$ of $\mc{Q}^{2m+1}$ injects into $H^0(\hat{x}, \mc{O}(\Nrm \hat{x}))$, the space of global holomorphic sections of $\Nrm \hat{x}$. If $V^a$ is a vector in $\Tgt_x \mc{Q}^{2m+1}$ and $y$ the point inf\/initesimally separated from $x$ by $V^a$, then the corresponding section of $H^0(\hat{x}, \mc{O}(\Nrm \hat{x}))$ can be identif\/ied with $\hat{y}$. Let us f\/ix $x$ to be the origin in $\C\E^{2m+1} \subset \mc{Q}^{2m+1}$. Then $V^a$ can be identif\/ied with $y^a$. We view $\pi^{\mbf{A}}$ as coordinates on $\hat{x}$ given by the locus $\omega^{\mbf{A}}=0$. The inf\/initesimal displacement of $\hat{x}$ along $V^a$ at the origin is $V^\mbf{A} := V^a \nabla_a \omega^{\mbf{A}}$, i.e., $V^\mbf{A} = \frac{1}{\sqrt{2}} V^a \pi_a^\mbf{A}$. This represents a global holomorphic section $\widehat{V}_{\hat{x}}$ of $\Nrm \hat{x}$, and can be identif\/ied with the complex submanifold $\hat{y}$ given by $\omega^{\mbf{A}} = \frac{1}{\sqrt{2}} y^a \pi_a^{\mbf{A}}$.

Before describing such sections, we shall need the following two lemmata.

\begin{Lemma}\label{lem-eigenspinors}
Let $V^a$ be a non-zero vector in $\C \E^{2m+1}$, and let $V_{\mbf{A}}^{\mbf{B}} := V^a \gamma \ind{_a_{\mbf{A}}^{\mbf{B}}}$ be the corresponding spin endomorphism. Then $V^a$ is null if and only if $V_{\mbf{A}}^{\mbf{B}}$ has a zero eigenvalue. Further,
\begin{itemize}\itemsep=0pt
\item if $V^a$ is null, $V_{\mbf{A}}^{\mbf{B}}$ has a single zero eigenvalue of algebraic multiplicity $2^m$, and its eigenspace is isomorphic to the $2^{m-1}$-dimensional spinor space of $\C\E^{2m-1}$,
\item if $V^a$ is non-null, $V_{\mbf{A}}^{\mbf{B}}$ has a pair of eigenvalues $\pm \ii \sqrt{V^a V_a}$, each of algebraic multipli\-ci\-ty~$2^{m-1}$, and their respective eigenspaces are isomorphic to the $2^{m-1}$-dimensional chiral spinor spaces of~$\C\E^{2m}$.
\end{itemize}
\end{Lemma}

\begin{proof}By the Clif\/ford property, we have $V_{\mbf{A}}^{\mbf{C}} V_{\mbf{C}}^{\mbf{B}} = - V^a V_a \delta_{\mbf{A}}^{\mbf{B}}$, and it follows that any eigenvalue of $V_{\mbf{A}}^{\mbf{B}}$ must be equal to $\pm \ii \sqrt{V^a V_a}$. Hence $V^a$ is null if and only if it has a zero eigenvalue. This zero eigenvalue must be of algebraic multiplicity $2^m$ since in this case $V_{\mbf{A}}^{\mbf{B}}$ is nilpotent. One can check that the kernel of $V_{\mbf{A}}^{\mbf{B}}$ can be identif\/ied with the $2^{m-1}$-dimensional spinor space of $\C\E^{2m-1}$ as the orthogonal complement of $V^a$ in $\C\E^{2m+1}$ quotiented by $\langle V^a \rangle$.

If $V^a$ is non-null, the square of $V_{\mbf{A}}^{\mbf{B}}$ is proportional to the identity, and thus, each of the eigenvalues $\pm \ii \sqrt{V^a V_a}$ must have algebraic multiplicity $2^{m-1}$. Each of the eigenspaces can be identif\/ied with each of the chiral spinor spaces of $\C\E^{2m}$ as the orthogonal complement of $\langle V^a \rangle$ in $\C\E^{2m+1}$ -- see, e.g.,~\cite{Penrose1986}.
\end{proof}

\begin{Lemma}\label{lem-non-null-twistor}Let $x$ and $y$ be two points in $\mc{Q}^{2m+1}$ infinitesimally separated by a non-null vector~$V^a$. Then, for every $Z \in \hat{x} \subset \PT$ such that $V^a$ is tangent to the co-$\gamma$-plane $\check{Z}^\perp \subset \Tgt_x \mc{Q}^{2m+1}$, $\mathrm{D}_Z$ intersects $\hat{y}$ in a unique point~$W$, say, such that the corresponding $\gamma$-planes~$\check{Z}$ and~$\check{W}$ intersect maximally.
\end{Lemma}

\begin{proof}With no loss of generality, we may assume that $x$ is the origin in $\C\E^{2m+1} \subset \mc{Q}^{2m+1}$. We then have $V^a=y^a$. Since $V^a$ is non-null, it must lie on some co-$\gamma$-plane of some twistor~$Z$. Following the discussion of Section~\ref{sec-co-gam}, it can be represented by a $1$-parameter family of $\gamma$-planes. In particular, $y$ must lie on one such $\gamma$-plane. If $Z$ is a point on $\hat{x}$, then $[Z^{\bm{\upalpha}}] = [0,\pi^{\mbf{A}}]$ for some $\pi^{\mbf{A}}$. The condition that $y$ lies on the co-$\gamma$-plane $Z^\perp$ is that $\pi^{\mbf{A}}$ is an eigenspinor of $V^a$ with eigenvalue $t$ or $-t$ where $t:=\ii \sqrt{V^a V_a}$. For def\/initeness, let us assume that the eigenvalue is~$t$. With reference to~\eqref{eq-MT-PT-curve}, the point $y^a$ lies in the $\gamma$-plane $W$ given by $[W^{\bm{\upalpha}}] = [t \pi^{\mbf{A}},\pi^{\mbf{A}}]$. Re-expressing this twistor as $W^{\bm{\upalpha}} = \Xi^{\bm{\upalpha}} - \frac{t}{2} \mr{Y}^\mc{A} \Gamma \ind{_{\mc{A}}_{\bm{\upbeta}}^{\bm{\upalpha}}} \Xi^{\bm{\upalpha}}$, we see, by Lemma~\ref{lem-Z-orbit-trc}, that $W$ lies in the intersection of $\mathrm{D}_Z$ and $\hat{y}$. In fact, one can see that the connecting vector from $Z$ to $W$ is given by $\sqrt{V^a V_a} \bm{Y}$, where $\bm{Y}$ is given by~\eqref{eq-Infinity-Vec-field}. Finally, by Proposition~\ref{prop-XiZ}, $Z$ and~$W$ must intersect maximally.
\end{proof}

\begin{Proposition}\label{prop-zero-normal}Let $x$ be a point in $\mc{Q}^{2m+1}$ with corresponding submanifold $\hat{x}$ in $\PT$. Let~$V$ be a~tangent vector at~$x$, and $\widehat{V}_{\hat{x}}$ its corresponding global holomorphic section of~$\mathrm{N} \hat{x}$.
\begin{itemize}\itemsep=0pt
\item Suppose $V$ is null. When $m=1$, $\widehat{V}_{\hat{x}}$ vanishes at a single point on $\hat{x}$, which corresponds to the unique $\gamma$-plane $($i.e., null line$)$ to which $V$ is tangent. When $m>1$, there is a $\frac{1}{2}m(m-1)$-dimensional algebraic subset of $\hat{x}$ biholomorphic to $\PT_{(2m-3)}$ on which $\widehat{V}_{\hat{x}}$ vanishes. Each point of this subset corresponds to a $\gamma$-plane to which $V$ is tangent.
\item Suppose $V$ is non-null. When $m=1$, there are precisely two points, $Z_\pm$ say, on $\hat{x}$, at which $\widehat{V}_{\hat{x}}(Z_\pm) \in \mathrm{N}^{-1}_{Z_\pm} \hat{x}$. Further, $V$ is tangent to the two co-$\gamma$-planes determined by $Z_\pm$. When $m>1$, there are two disjoint $\frac{1}{2}m(m-1)$-dimensional algebraic subsets of $\hat{x}$, biholomorphic to $\PT_{(2m-2)}$ and $\PT'_{(2m-2)}$, over which $\widehat{V}_{\hat{x}}$ is a section of $\mathrm{N}^{-1} \hat{x}$. Each point of these subsets corresponds to a co-$\gamma$-plane to which $V$ is tangent.
\end{itemize}
Conversely, if $\widehat{V}_{\hat{x}}$ vanishes at a point, then $V$ must be null, and if $\widehat{V}_{\hat{x}} (Z) \in \Nrm^{-1}_Z \hat{x}$ for some $Z \in \hat{x}$, then $V$ must be non-null.
\end{Proposition}

\begin{proof}Again, let us assume that $x$ is the origin in $\C\E^{2m+1} \subset \mc{Q}^{2m+1}$, and set $V_{\mbf{A}}^{\mbf{B}} := V^a \gamma \ind{_a_{\mbf{A}}^{\mbf{B}}}$.

If $V^a$ is null, the vanishing of $\widehat{V}_{\hat{x}}$ of a point $\pi^{\mbf{A}}$ of $\hat{x}$ is simply equivalent to $V_{\mbf{A}}^{\mbf{B}} \pi^{\mbf{A}}=0$, i.e., $\pi^{\mbf{A}}$ is a pure eigenspinor of $V_{\mbf{A}}^{\mbf{B}}$. By Lemma~\ref{lem-eigenspinors}, we can immediately conclude that $\widehat{V}_{\hat{x}}$ vanishes at a point when $m=1$, and on a subset of $\hat{x}$ biholomorphic to $\PT_{(2m-3)}$ when $m>1$. Clearly, each point of this subset corresponds to a $\gamma$-plane to which~$V^a$ is tangent.

If $V^a$ is non-null, we know by Lemma~\ref{lem-eigenspinors} that $V_{\mbf{A}}^{\mbf{B}}$ has eigenvalues $\pm \ii \sqrt{V^a V_a}$. In particular, the pure eigenspinors up to scale determine two distinct points on $\hat{x}$ when $m=1$, and two disjoint subsets of $\hat{x}$ biholomorphic to the twistor spaces $\PT_{(2m-2)}$ and $\PT'_{(2m-2)}$ when $m>1$. A point $Z$ on any of these sets corresponds to a co-$\gamma$-plane $\check{Z}^\perp$ to which $V^a$ is tangent. By Lemma \ref{lem-non-null-twistor}, the corresponding submanifold $\hat{y}$ intersects $\mathrm{D}_Z$ at a point $W$. The connecting vector from $Z$ to $W$ clearly lies in $\mathrm{D}_Z$, but is not tangent to $\hat{x}$. In particular, it descends to an element of $\mathrm{N}_Z^{-1} \hat{x}$. Thus, the restriction of $\widehat{V}_{\hat{x}}$ to these subsets is a section of $\mathrm{N}^{-1} \hat{x}$.

Finally, if $\widehat{V}_{\hat{x}}$ vanishes at a point $Z$ say, then $V$ is tangent to the $\gamma$-plane $\check{Z}$, and so must be null. The non-null case is similar.
\end{proof}

\subsubsection{Mini-twistor space}
For any point $x$ of $\mc{Q}^n$, the normal bundle $\Nrm \underline{\hat{x}}$ of $\underline{\hat{x}}$ in $\M\T$ is given by \mbox{$0 {\rightarrow} \Tgt \underline{\hat{x}} {\rightarrow} \Tgt \left. \M\T \right|_{\underline{\hat{x}}} {\rightarrow} \Nrm \underline{\hat{x}} {\rightarrow} 0$}. In this case, $\Nrm \underline{\hat{x}}$ can be identif\/ied with $\Tgt^{-1} \hat{x}$, i.e., mini-twistor space itself, as follows from the description of Section~\ref{sec-co-gam}: taking $x$ in $\C \E^n$ to be the origin, then the complex submanifold $\underline{\hat{x}}$ in~$\M\T$ is def\/ined by $\underline{\omega}_{a_1 \ldots a_{m-1}}=0$, $\pi^{\mbf{A}}$ will be coordinates on~$\underline{\hat{x}}$, and we shall view $\underline{\omega}_{a_1 \ldots a_{m-1}}$ as coordinates of\/f $\underline{\hat{x}}$.

Again, for any $x \in \C \E^n$, $\Tgt_x \C \E^n$ injects into $H^0(\underline{\hat{x}}, \mc{O}(\Nrm \underline{\hat{x}}))$. If $x$ is the origin and $V \in \Tgt_x \C \E^n$ be the vector connecting~$x$ to a point~$y$, we can identify the global holomorphic section $\widehat{V}_{\underline{\hat{x}}}$ of~$\Nrm \underline{\hat{x}}$ as in the previous section. If $\pi^{\mbf{A}}$ are coordinates on $\hat{x}$ given by the locus $\omega^{\mbf{A}}=0$, $\widehat{V}_{\underline{\hat{x}}}$ can be identif\/ied with the complex submanifold~$\underline{\hat{y}}$ given by~$\underline{\omega}_{a_1 \ldots a_{m-1}} = \frac{1}{\sqrt{2}} y^a \gamma \ind*{^{(m)}_{a a_1 \ldots a_{m-1}}_{\mbf{A} \mbf{B}}} \pi^{\mbf{A}} \pi^{\mbf{B}}$, where $y^a = V^a$.

\begin{Proposition}\label{prop-zero-normal-MT}
Let $x$ be a point in $\C \E^{2m+1}$ with corresponding submanifold $\underline{\hat{x}}$ in $\M\T$. Let~$V$ be a tangent vector at~$x$, and $\widehat{V}_{\underline{\hat{x}}}$ its corresponding global holomorphic section of~$\mathrm{N} \underline{\hat{x}}$.
\begin{itemize}\itemsep=0pt
\item Suppose $V$ is null. When $m=1$, $\widehat{V}_{\underline{\hat{x}}}$ has a double zero, which corresponds to the $\gamma$-plane to which $V$ is tangent. When $m>1$, $\widehat{V}_{\underline{\hat{x}}}$ vanishes on a $\frac{1}{2}m(m-1)$-dimensional algebraic subset of $\underline{\hat{x}}$ biholomorphic to $\PT_{(2m-1)}$ of multiplicity $2^m$. Each point of this subset corresponds to a $\gamma$-plane to which $V$ is tangent.
\item Suppose $V$ is non-null. When $m=1$, $\widehat{V}_{\underline{\hat{x}}}$ has two simple zeros, each of which determines a~co-$\gamma$-plane to which $V$ is tangent. When $m>1$, $\widehat{V}_{\underline{\hat{x}}}$ vanishes on two disjoint $\frac{1}{2}m(m-1)$-dimensional algebraic subsets of $\underline{\hat{x}}$ bihomolomorphic to $\PT_{(2m-2)}$ and $\PT'_{(2m-2)}$, each of multiplicity $2^{m-1}$. Each point of these subsets corresponds to a co-$\gamma$-plane to which $V$ is tangent.
\end{itemize}
\end{Proposition}

\begin{proof}With no loss, we assume that $x$ is the origin in $\C\E^{2m+1} \subset \mc{Q}^{2m+1}$, and set $V_{\mbf{A}}^{\mbf{B}} := V^a \gamma \ind{_a_{\mbf{A}}^{\mbf{B}}}$. To determine the zero set of $\underline{\hat{V}}_{\underline{\hat{x}}}$, we simply remark that $V^a \gamma \ind*{^{(m)}_{a a_1 \ldots a_{m-1}}_{\mbf{A} \mbf{B}}} \pi^{\mbf{A}} \pi^{\mbf{B}} =0$ is equivalent to the eigenspinor equation $\pi^{\mbf{C}} V\ind*{_{\mbf{C}}^{[\mbf{A}}} \pi^{\mbf{B}]} =0$. We can then proceed as in the proof of Proposition~\ref{prop-zero-normal} according to whether $V^a$ is null or non-null, and obtain the required zero sets of the section $\widehat{V}_{\underline{\hat{x}}}$ in each case, the multiplicities being given by the algebraic multiplicities of the eigenvalues of $V \ind*{_{\mbf{C}}^{\mbf{A}}}$. In particular, when $m=1$, the solution set is def\/ined by the vanishing of a single homogeneous polynomial of degree $2$, which has two distinct roots generically, but a~single root of multiplicity two when~$V^a$ is null~-- see, e.g.,~\cite{Jones1985}.
\end{proof}

\subsubsection{Even dimensions}
The analysis when $n=2m$ is very similar to the odd-dimensional case without the added complication of the canonical distribution. Again, for any $x$ of $\mc{Q}^{2m}$, $\Tgt_x \C \E^{2m}$ injects into $H^0 (\hat{x}, \mc{O}(\Nrm \hat{x}))$. A null vector in $V^a$ is $\Tgt_x \C \E^{2m}$ def\/ines a global section $\widehat{V}_{\hat{x}}$ of $\Nrm \hat{x}$, which vanishes at a single point when $m=2$, and on a $\frac{1}{2}(m-1)(m-2)$-dimensional algebraic subset of $\hat{x}$, isomorphic to $\PT_{(2m-2)}$, when $m>2$. Each point of this subset corresponds to an $\alpha$-plane to which $V^a$ is tangent.

\subsubsection{Kodaira's theorem and completeness}
Let us now turn to the question of whether $\Tgt_x \mc{Q}^n$ maps to $H^0(\hat{x}, \mc{O}(\Nrm \hat{x}))$ bijectively, and not merely injectively, for any $x \in \mc{Q}^n$. By Kodaira's theorem \cite{Kodaira1962}, $\Tgt_x \mc{Q}^n \cong H^0(\hat{x}, \mc{O}(\Nrm \hat{x})) \cong \C^n$ if and only if the family $\{ \hat{x} \}$ in $\PT$ is \emph{complete}, i.e., \emph{any} inf\/initesimal deformation of $\hat{x}$ arises from an element of $\Tgt_x \mc{Q}^n$. As we have seen in Section~\ref{sec-even-odd}, the twistor space $\PT$ of $\mc{Q}^{2m+1}$ and the twistor space $\widetilde{\PT}$ of $\mc{Q}^{2m+2}$ are both $\frac{1}{2}(m+1)(m+2)$-dimensional complex projective varieties in~$\CP^{2^{m+1}-1}$, and it is the embedding $\mc{Q}^{2m+1} \subset \mc{Q}^{2m+2}$ that induces the canonical distribution~$\Drm$ on~$\PT$. The issue here is that Kodaira's theorem is only concerned with the holomorphic structure of the underlying manifolds, and does not depend on the additional distribution on~$\PT$.

Now, by the twistor correspondences, any point $x$ in $\mc{Q}^{2m+1}$ and $\mc{Q}^{2m+2}$ gives rise to a~$\frac{1}{2}m(m+1)$-dimensional complex submanifold $\hat{x}$ of $\PT$ and $\widetilde{\PT}$ respectively. This means that the analytic family $\{ \hat{x} \}$ parametrised by the points $\{x\}$ of $\mc{Q}^{2m+1}$ can be completed to a~larger fa\-mi\-ly parametrised by the points $\{x\}$ of $\mc{Q}^{2m+2}$ via the embedding $\mc{Q}^{2m+1} \subset \mc{Q}^{2m+2}$. Further, a~complex submanifold $\hat{x}$ corresponds to a point $x$ in $\mc{Q}^{2m+1}$ if and only if $\hat{x}$ is tangent to an $m$-dimensional subspace of $\mathrm{D}_Z$ at every point $Z \in \hat{x}$.

We also need to check whether the family of $\hat{x}$ is complete when $x \in \mc{Q}^{2m+2}$. If it were not, one would be able to f\/ind a group of biholomorphic automorphisms of $\PT$ larger than $\Spin(2m+4,\C)$ and a parabolic subgroup such that the quotient models $\PT$. But the work of \cite{Doubrov2010, Onivsvcik1960} tells us that there is no such group. The same applies to each $\hat{x}$, and since these are biholomorphic to f\/lag varieties, the normal bundle $\Nrm \hat{x}$ can be identif\/ied with a rank-$(m+1)$ holomorphic homogeneous vector bundle over~$\hat{x}$. In the notation of \cite{Baston1989}, we f\/ind that for a~point~$x$ in~$\mc{Q}^{2m+1}$ or~$\mc{Q}^{2m+2}$, the normal bundle $\Nrm \hat{x}$ in $\PT \cong \widetilde{\PT}$ is given by
\begin{gather*}
\renewcommand{\arraystretch}{1.5}
\begin{array}{cc}
m=1 \qquad & m>1 \\
\xy
@={(-2.5,0),(2.5,0)} @@{s0*=0{\times}, s1*=0{\times}};
 {s0+(0,2)}*=0{\sk{1}}; {s1+(0,2)}*=0{\sk{1}};
\endxy \qquad
&
\underbrace{\xy
@={(3.5,-3.5),(3.5,3.5),(0,0),(-5,0),(-10,0)} @@{s0*=0{\bullet}, s1*=0{\bullet}, s2*=0{\bullet}, s3*=0{\times}, s4*=0{\bullet}};
s0; s1 **\dir{-}; s1; s2 **\dir{.}; s2; s3 **\dir{-}; s2; s4 **\dir{-};
 {s0+(0,2)}*=0{\sk{1}}; {s1+(0,2)}*=0{\sk{0}}; {s2+(0,2)}*=0{\sk{0}}; {s3+(0,2)}*=0{\sk{0}}; {s4+(0,2)}*=0{\sk{0}};
 \endxy}_{\mbox{\tiny{$m+1$ nodes}}}
\end{array}
\end{gather*}
Here, the mutilated Dynkin diagram corresponds to the parabolic subalgebra underlying the f\/lag variety $\hat{x}$, and the coef\/f\/icients over the nodes to the irreducible representation that determines the vector bundle. When $m=1$, i.e., for $\mc{Q}^3$ and $\mc{Q}^4$, we recover the well-known result $\Nrm _{\hat{x}} \cong \mc{O}_{\hat{x}}(1) \oplus \mc{O}_{\hat{x}}(1)$, where $\mc{O}_{\hat{x}}(1)$ is the hyperplane bundle over $\hat{x} \cong \CP^1$. We can compute the cohomology using the Bott--Borel--Weil theorem, and verify that indeed $H^0 (\hat{x},\mc{O}(\Nrm \hat{x})) \cong \C^{2m+2}$ and $H^1 (\hat{x},\mc{O}(\Nrm \hat{x})) = 0$ -- this latter condition tells us that there is no obstruction for the existence of our family.

We can play the same game with the family of compact complex submanifolds $\{ \underline{\hat{x}} \}$ in $\M\T$ parametrised by the points $x$ of $\C \E^{2m+1}$. But in this case, for any $x$ of $\C \E^{2m+1}$, the normal bundle $\Nrm \underline{\hat{x}}$ is essentially the total space of $\Tgt^{-1} \hat{x} \rightarrow \hat{x}$, and is described, in the notation of \cite{Baston1989}, as the rank-$m$ holomorphic homogeneous vector bundle
\begin{align*}
\renewcommand{\arraystretch}{1.5}
\begin{array}{cc}
m=1 \qquad & m>1 \\
\xy
@={(0,0)} @@{s0*=0{\times}};
 {s0+(0,2)}*=0{\sk{2}};
\endxy \qquad
&
\underbrace{\xy
@={(5,0),(0,0),(-5,0),(-10,0)} @@{s0*=0{\bullet}, s1*=0{\bullet}, s2*=0{\bullet}, s3*=0{\times}};
s0; s1 **\dir{-}; s1; s2 **\dir{.}; s2; s3 **\dir2{-};
 {s2+(2.5,0)}*=0{>};
{s0+(0,2)}*=0{\sk{1}}; {s1+(0,2)}*=0{\sk{0}}; {s2+(0,2)}*=0{\sk{0}}; {s3+(0,2)}*=0{\sk{0}};
\endxy}_{\mbox{\tiny{$m$ nodes}}}
\end{array}
\end{align*}
When $m=1$, i.e., $\mc{Q}^3$, $\underline{\hat{x}} \cong \CP^1$, and we recover the well-known result $\mc{O}(\Nrm {\underline{\hat{x}}}) \cong \mc{O}_{\underline{\hat{x}}} (2):=\otimes^2 \mc{O}_{\underline{\hat{x}}} (1)$. Again, the Bott--Borel--Weil theorem conf\/irms that $H^0 (\underline{\hat{x}},\mc{O}(\Nrm \underline{\hat{x}})) \cong \C^{2m+1}$ and $H^1 (\underline{\hat{x}},\mc{O}(\underline{\Nrm \hat{x}})) = 0$.

\begin{Remark}
When $n=3$, this analysis was already exploited in~\cite{LeBrun1982} in the curved setting, where the twistor space of a three-dimensional holomorphic conformal structure is identif\/ied with the space of null geodesics. See also~\cite{Hitchin1982a}.
\end{Remark}

\section{Null foliations}\label{sec-null-foliations}
As before, we work in the holomorphic category throughout, i.e., vector f\/ields and distributions will be assumed to be holomorphic.
\begin{Definition}
An \emph{almost null structure} is a holomorphic totally null $m$-plane distribution on~$\mc{Q}^n$, where $n=2m$ or $2m+1$.
\end{Definition}
In other words, an almost null structure is a $\gamma$-plane, $\alpha$-plane or $\beta$-plane distribution. From the discussion of Section~\ref{sec-correspondence-space}, an almost null structure, self-dual when $n=2m$, can be viewed as a holomorphic section of $\F \rightarrow \mc{Q}^n$, or equivalently as a projective pure spinor f\/ield on $\mc{Q}^n$, that is a spinor f\/ield def\/ined up to scale, and which is pure at every point. The geometric properties of an almost null structure on a general spin complex Riemannian manifold can be expressed in terms of the dif\/ferential properties of its corresponding projective pure spinor f\/ield as described in \cite{Taghavi-Chabert2016,Taghavi-Chabert2013}.

The question we now wish to address is the following one: given an almost null structure, how can we encode its geometric properties in twistor space $\PT$?

\subsection{Odd dimensions}
When $n=2m+1$, an almost null structure is more adequately expressed as an inclusion of distributions $N \subset N^\perp$ where $N$ is a holomorphic totally null $m$-plane distribution and $N^\perp$ is its orthogonal complement. One can then investigate the geometric properties of $N$ and $N^\perp$ independently. In the following, $\Gamma (\mc{U}, \mc{O}(N))$ denotes the space of holomorphic sections of $N$ over an open subset $\mc{U}$ of $\mc{Q}^n$, and similarly for $N^\perp$.

\begin{Definition}\label{defn-int-co-geo}
Let $N \subset N^\perp$ be an almost null structure on some open subset $\mc{U}$ of $\mc{Q}^n$. We say that $N$ is
\begin{itemize}\itemsep=0pt
\item \emph{integrable} if $[\bm{X},\bm{Y}] \in \Gamma (\mc{U}, \mc{O}(N))$ for all $\bm{X}, \bm{Y} \in \Gamma (\mc{U}, \mc{O}(N))$,
\item \emph{totally geodetic} if $\nabla_{\bm{Y}} \bm{X} \in \Gamma (\mc{U}, \mc{O}(N))$ for all $\bm{X}, \bm{Y} \in \Gamma (\mc{U}, \mc{O}(N))$,
\item \emph{co-integrable} if $[\bm{X},\bm{Y}] \in \Gamma (\mc{U}, \mc{O}(N^\perp))$ for all $\bm{X}, \bm{Y} \in \Gamma (\mc{U}, \mc{O}(N^\perp))$,
\item \emph{totally co-geodetic} if $\nabla_{\bm{Y}} \bm{X} \in \Gamma (\mc{U}, \mc{O}(N^\perp))$ for all $\bm{X}, \bm{Y} \in \Gamma (\mc{U}, \mc{O}(N^\perp))$.
\end{itemize}
An integrable almost null structure will be referred to as a \emph{null structure}.
\end{Definition}

There is however some dependency regarding the geometric properties of $N$ and $N^\perp$.
\begin{Lemma}[\cite{Taghavi-Chabert2013}]\label{lem-co-geod-int-rel}
Let $N$ be an almost null structure. Then
\begin{itemize}\itemsep=0pt
\item if $N$ is totally co-geodetic, it is also integrable and co-integrable,
\item if $N$ is integrable and co-integrable, it is also geodetic,
\item if $N$ is totally geodetic, it is also integrable.
\end{itemize}
\end{Lemma}

Another important point is the conformal invariance of the above properties. All with the exception of the totally co-geodetic property are conformal invariant -- see \cite{Taghavi-Chabert2013}.

\subsubsection{Local description}\label{sec-null-strc-desc}
The next theorems will be local in nature. That means that we shall work on $\C \E^n$ viewed as a dense open subset of $\mc{Q}^n$. For their proofs, we shall make use of the local coordinates on $\C \E^n$, $\F_{\C\E^n}$ and $\PT_{\setminus \widehat{\infty}}$ given in Appendix \ref{sec-affine-odd}. Let $N$ be an almost null structure on some open subset $\mc{U}$ of $\C \E^n = \{ z^A,z_A,u\}$, and view $N$ as a local holomorphic section of $\F \rightarrow \C \E^n$, i.e., a~holomorphic projective pure spinor f\/ield $[\xi^\mbf{A}]$. We may assume that locally, $[\xi^\mbf{A}]$ def\/ines a~complex submanifold of $\mc{U} \times \mc{U}_0$, where $(\mc{U}_0 , ( \pi^A , \pi^{AB} ))$ is a coordinate chart on the f\/ibers of~$\F_\mc{U}$, given by the graph
\begin{gather}\label{eq-graph-xi}
\Gamma_\xi := \big\{ (x,\pi) \in \mc{U} \times \mc{U}_0 \colon \pi^{AB} = \xi^{AB} ( x ) ,\, \pi^{A} = \xi^{A} ( x ) \big\} ,
\end{gather}
for some $\frac{1}{2}m(m-1)$ and $m$ holomorphic functions $\xi^{AB}=\xi^{[AB]}$ and $\xi^{A}$ respectively on $\mc{U}$. In this case, the distribution $N$ is spanned by the $m$ holomorphic vector f\/ields
\begin{gather}\label{eq-N-span-odd}
\bm{Z}^{A} = \partial ^{A} + \left( \xi^{AD} - \tfrac{1}{2} \xi^{A} \xi^{D} \right) \partial _{D} + \xi^{A} \partial ,
\end{gather}
while its orthogonal complement $N^\perp$ by the $m+1$ holomorphic vector f\/ields
\begin{gather}\label{eq-Nperp-span-odd}
\bm{Z}^{A} = \partial ^{A} + \left( \xi^{AD} - \tfrac{1}{2} \xi^{A} \xi^{D} \right) \partial _{D} + \xi^{A} \partial , \qquad \bm{U} = \partial - \xi^{D} \partial _{D} ,
\end{gather}
where $\partial^A := \parderv{}{z_A}$, $\partial_A := \parderv{}{z^A}$ and $\partial := \parderv{}{u}$. Here, we shall make a slight abuse of notation by denoting the vector f\/ields spanning $N$ and $N^\perp$, and their lifts to $\F_\mc{U}$, both by~\eqref{eq-N-span-odd} and~\eqref{eq-Nperp-span-odd}.

\begin{Remark}It will be understood that when $m=1$ there are no coordinates $\pi^{AB}$. This does not af\/fect the veracity of the following results in this case~-- see however Remark~\ref{rem-fol-1}.
\end{Remark}

\subsubsection{Totally geodetic null structures}
Let $\mc{W}$ be an $(m+1)$-dimensional complex submanifold of $\PT$ and let $\mc{U}$ be an open subset of $\mc{Q}^{2m+1}$. Suppose that for every point $x$ of $\mc{U}$, $\hat{x} \in \widehat{\mc{U}}$ intersects $\mc{W}$ transversely in a point. Then each point of $\mc{W} \cap \hat{x}$ determines a point in the f\/iber $\F_x$, and thus a $\gamma$-plane through $x$. Smooth variations of the point $x$ in $\mc{U}$ thus def\/ine a holomorphic section of $\F_\mc{U} \rightarrow \mc{U}$ and an $(m+1)$-dimensional analytic family of $\gamma$-planes, each of which being the totally geodetic leaf of an integrable almost null structure. Conversely, consider a local foliation by totally null and totally geodetic $m$-dimensional leaves. Then, each leaf must be some af\/f\/ine subset of a $\gamma$-plane. The $(m+1)$-dimensional leaf space of the foliation constitutes an $(m+1)$-dimensional analytic family of $\gamma$-planes, and thus def\/ines an $(m+1)$-dimensional complex submanifold of~$\PT$.

\begin{Theorem}\label{thm-odd-Kerr-theorem-I}
A totally geodetic null structure on some open subset $\mc{U}$ of $\mc{Q}^{2m+1}$ gives rise to an $(m+1)$-dimensional complex submanifold of $\widehat{\mc{U}} \subset \PT$ intersecting $\hat{x} \subset \widehat{\mc{U}}$ transversely for each $x \in \mc{U}$. Conversely, any totally geodetic null structure locally arises in this way.
\end{Theorem}

\begin{proof}Let $N$ be an almost null structure as described in Section~\ref{sec-null-strc-desc}. The condition that $N$ be totally geodetic is $\bm{g} ( \nabla_{\bm{Z}^{A}} \bm{Z}^{B} , \bm{Z}^{C} ) = \bm{g} ( \nabla_{\bm{Z}^{A}} \bm{Z}^{B} , \bm{U} ) = 0$, i.e.,
\begin{gather}\label{eq-int-spinor2twistor-odd}\begin{split}
& \left( \partial ^{A} + \left( \xi^{AD} - \tfrac{1}{2} \xi^{A} \xi^{D} \right) \partial _{D} + \xi^{A} \partial \right) \xi^{BC} = 0 , \\
& \left( \partial ^{A} + \left( \xi^{AD} - \tfrac{1}{2} \xi^{A} \xi^{D} \right) \partial _{D} + \xi^{A} \partial \right) \xi^{B} = 0 .
\end{split}
\end{gather}
We re-express the system \eqref{eq-int-spinor2twistor-odd} of holomorphic partial dif\/ferential equations as
\begin{gather}\label{eq-jetspace-mfd-odd}
\begin{split}
& \rho^{ABC} + \left( \pi^{AD} - \tfrac{1}{2} \pi^{A} \pi^{D} \right) \rho_{D}^{BC} + \pi^{A} \rho^{BC} = 0 , \\
& \sigma^{AB} + \left( \pi^{AD} - \tfrac{1}{2} \pi^{A} \pi^{D} \right) \sigma_{D}^{B} + \pi^{A} \sigma^{B} = 0 ,
\end{split}
\end{gather}
where $\rho^{ABC}:=\partial^{A} \pi^{BC} $, $\rho_{A}^{BC} :=\partial_{A} \pi^{BC}$, $\rho^{AB}:=\partial \pi^{AB}$, $\sigma^{AB}:=\partial^{A} \pi^{B}$, $\sigma_{A}^{B} :=\partial_{A} \pi^{B}$, $\sigma^{A}:=\partial \pi^{A}$. In the language of jets, the locus~\eqref{eq-jetspace-mfd-odd} def\/ines a complex submanifold of the f\/irst jet space $\mc{J}^1 (\C \E^n,\mc{U}_0)$, of which the prolongation of the section~$\Gamma_\xi$ is a submanifold. Now, the distribution $\Tgt^{-1}_E \F = \langle \bm{Z}^A \rangle$ tangent to the f\/ibers of $\F \rightarrow \PT$ is annihilated by the $1$-forms $\dd \pi^A$, $\dd \pi^{AB}$, $\bm{\theta}^A$ and~$\bm{\theta}^0$ as def\/ined in Appendix~\ref{sec-affine-odd}, which can be pulled back to $\mc{J}^1 (\C \E^n,\mc{U}_0)$. The $1$-forms def\/ined~by
\begin{gather}
\begin{split}
& \bm{\phi}^A := \dd \pi^{A} - \sigma_{C}^{A} \bm{\theta}^C - \left( \sigma^{A} - \sigma_{C}^{A} \pi^{C} \right) \bm{\theta}^0 , \\
& \bm{\phi}^{AB} := \dd \pi^{AB} - \rho_{C}^{AB} \bm{\theta}^C - \left( \rho^{AB} - \rho_{C}^{AB} \pi^C \right) \bm{\theta}^0 ,
\end{split}\label{eq-contact-odd}
\end{gather}
vanish on the locus \eqref{eq-jetspace-mfd-odd}, and this implies in particular that, for generic $\rho_{C}^{AB}$, $\rho^{BC}$, $\rho_{C}^{A}$, $\rho^{C}$, the section $\Gamma_\xi$ must be constant along the f\/ibers of $\Tgt^{-1}_E \F$, i.e., the functions $(\xi^A,\xi^{AB})$ depend only on the coordinates $(\omega^0, \omega^A,\pi^A,\pi^{AB})$ of the chart $\mc{V}_0$ of $\PT$. Thus, quotienting $\Gamma_\xi$ along the f\/ibers of $\F \rightarrow \PT$ yields an $(m+1)$-dimensional complex submanifold of~$\PT$ intersecting each~$\hat{x}$ transversely in a point.

The converse is also true: we start with an $(m+1)$-dimensional complex submanifold $\mc{W}$, say, of~$\PT$, which can be locally represented by the vanishing of $\frac{1}{2}m(m+1)$ holomorphic functions $(F^{AB} , F^{A})$ on the chart $(\mc{V}_0, (\omega^0, \omega^A,\pi^A,\pi^{AB}))$. Then $(\dd F^{AB} , \dd F^{A})$ are a set of $1$-forms vanishing on $\mc{W}$. We shall assume that for each $x \in \mc{U}$, the submanifold $\hat{x} \subset \widehat{\mc{U}}$ intersects $\mc{W}$ transversely in a point. This singles out a local holomorphic section $[\xi^{\mbf{A}}]$ of $\mc{U} \times \mc{U}_0 \subset \F \rightarrow \mc{U}$. By the implicit function theorem, we may assume with no loss of generality that this is the graph~$\Gamma_\xi$ given by~\eqref{eq-graph-xi}. The pullbacks of $(\dd F^{AB} , \dd F^{A})$ to $\F$ vanish on $\Gamma_\xi$ and give the restriction
\begin{gather}\label{eq-pullbackF}
\begin{pmatrix}
Q_C^A & Q_{CD}^A \vspace{1mm}\\
Q_C^{AB} & Q_{CD}^{AB}
\end{pmatrix}
\begin{pmatrix}
\dd \pi^{C} \\
\dd \pi^{CD}
\end{pmatrix}
+
\begin{pmatrix}
\bm{Y} F^{A} & \bm{X}_C F^{A} \\
\bm{Y} F^{AB} & \bm{X}_C F^{AB}
\end{pmatrix}
\begin{pmatrix}
\bm{\theta}^0 \\
\bm{\theta}^C
\end{pmatrix} =
\begin{pmatrix}
0 \\
0
\end{pmatrix} ,
\end{gather}
where
\begin{gather}\label{eq-mat}
\begin{pmatrix}
Q_C^A & Q_{CD}^A \vspace{1mm}\\
Q_C^{AB} & Q_{CD}^{AB}
\end{pmatrix} :=
\begin{pmatrix}
\displaystyle \left( \parderv{}{\pi^{C}} + \frac{1}{2} u \parderv{}{\omega^{C}} - z_{C} \parderv{}{\omega^0} \right) F^{A}\!\! & \displaystyle\left( \parderv{}{\pi^{CD}} + z_{[C} \parderv{}{\omega^{D]}} \right) F^{A} \vspace{1mm}\\
\displaystyle\left( \parderv{}{\pi^{C}} + \frac{1}{2} u \parderv{}{\omega^{C}} - z_{C} \parderv{}{\omega^0} \right) F^{AB}\!\! & \displaystyle\left( \parderv{}{\pi^{CD}} + z_{[C} \parderv{}{\omega^{D]}} \right) F^{AB}
\end{pmatrix}\! .\!\!
\end{gather}
At generic points, the matrix \eqref{eq-mat} is invertible, and equations \eqref{eq-pullbackF} can immediately be seen to be equivalent to the vanishing of the forms \eqref{eq-contact-odd}. In particular, $\pi^{AB} = \xi^{AB}(x)$ and \mbox{$\pi^{A} = \xi^A(x)$} satisfy \eqref{eq-int-spinor2twistor-odd}, i.e., the distribution associated to the graph~$\Gamma_{\xi}$ is integrable and totally geo\-detic.
\end{proof}

\subsubsection{Co-integrable null structures}
Let us now suppose that our almost null structure $N$ is integrable and co-integrable on $\mc{U}$. We then have two foliations of $\mc{U}$, one for $N$ and the other for $N^\perp$. By Lemma \ref{lem-co-geod-int-rel}, we know that each leaf of $N$ is totally geodetic and therefore a $\gamma$-plane. Since $N \subset N^\perp$, each $(m+1)$-dimensional leaf of $N^\perp$ contains a one-parameter holomorphic family $\{ \check{Z}_t \}$ of $\gamma$-planes, i.e., of leaves of $N$. Thus each leaf of $N^\perp$ descends to a holomorphic curve on the leaf space of $N$. In particular, by Theorem~\ref{thm-odd-Kerr-theorem-I}, we can identify the leaf space of $N$ with an $(m+1)$-dimensional complex submanifold $\mc{W}$ of $\PT$ foliated by curves, each of which being a one-parameter of twistors $\{ Z_t \}$ and, as we shall show, tangent to the canonical distribution $\mathrm{D}$ of $\PT$.

We start by the remark that at any point $Z$ of $\mc{W}$, any submanifold $\hat{x}$ intersects $\mc{W}$ transversely, i.e., $\Tgt_Z \PT = \Tgt_Z \hat{x} \oplus \Tgt_Z \mc{W}$. Hence, by Lemma~\ref{lem-D-Nx} the intersection of $\Drm_Z$ with $\Tgt_Z \mc{W}$ can only be at most one-dimensional. Now, let~$Z_0$ and~$Z_t$ be two points on $\mc{W}$ corresponding to two inf\/initesimally separated $\gamma$-planes, $\check{Z}_0$ and $\check{Z}_t$ in $\{ \check{Z}_t \}$, contained in the co-$\gamma$-plane $\check{Z}^\perp_0$. Let~$x$ and~$y$ be points on~$\check{Z}_0$ and~$\check{Z}_t$ respectively, so that their corresponding complex subma\-ni\-folds~$\hat{x}$ and~$\hat{y}$ of $\widehat{\mc{U}}$ intersect $\mc{W}$ in $Z_0$ and~$Z_t$ respectively. The vector $V^a$ in $\Tgt_x \mc{U}$ tangent to~$\check{Z}^\perp_0$ connecting~$x$ to~$y$ is non-null, and we know by Lemma~\ref{lem-non-null-twistor} that the vector connecting~$Z_0$ to~$Z_t$ must lie in~$\mathrm{D}_{Z_0}$. This is clearly independent of the choice of points~$x$ and~$y$ on~$\check{Z}_0$ and~$\check{Z}_t$. Assigning a vector tangent to $\Drm_{Z_t}$ at every point of $\{ Z_t \}$ yields a curve corresponding to a~leaf of~$N^\perp$. Proceeding in this way for each leaf of $N^\perp$ gives rise to a foliation by holomorphic curves tangent to~$\Drm$ on~$\mc{W}$. Conversely, any such foliation by curves on a given $(m+1)$-submanifold of~$\PT$ gives rise to an integrable and co-integrable almost null structure.

\begin{Theorem}\label{thm-odd-Kerr-theorem-II}
An integrable and co-integrable almost null structure on some open subset $\mc{U}$ of $\mc{Q}^{2m+1}$ gives rise to an $(m+1)$-dimensional complex submanifold of $\widehat{\mc{U}} \subset \PT$ foliated by holomorphic curves tangent to $\mathrm{D}$ and intersecting $\hat{x} \subset \widehat{\mc{U}}$ transversely for each $x \in \mc{U}$. Conversely, any integrable and co-integrable almost null structure locally arises in this way.
\end{Theorem}

\begin{proof}We recycle the setting and notation of the proof of Theorem \ref{thm-odd-Kerr-theorem-I}. In particular, we take $N$ and $N^\perp$ to be spanned by the vector f\/ields \eqref{eq-N-span-odd} and \eqref{eq-Nperp-span-odd}. The assumption that $N$ be integrable and co-integrable, i.e., $\bm{g} ( \nabla_{\bm{Z}^{A}} \bm{Z}^{B} , \bm{Z}^{C} ) = \bm{g} ( \nabla_{\bm{Z}^{A}} \bm{Z}^{B} , \bm{U} ) = \bm{g} ( \nabla_{\bm{U}} \bm{Z}^{B} , \bm{Z}^{C} ) =0$, gives~\eqref{eq-int-spinor2twistor-odd} and in addition,
\begin{gather}\label{eq-int-spinor2twistor+1}
\big( \partial - \xi^{D} \partial _{D} \big) \xi^{BC} + \big( \big( \partial - \xi^{D} \partial _{D} \big) \xi^{[B} \big) \xi^{C]} = 0 .
\end{gather}
Thus, the system \{\eqref{eq-int-spinor2twistor-odd}, \eqref{eq-int-spinor2twistor+1}\} can be encoded as the complex submanifold of~$\mc{J}^1(\C \E^n, \mc{U}_0)$ arising from the intersection of the locus~\eqref{eq-jetspace-mfd-odd} and the locus
\begin{gather}\label{eq-jetspace-mfd-odd2}
\rho^{BC} - \pi^{D} \rho_{D}^{BC} + \sigma^{[B} \pi^{C]} - \pi^{D} \sigma_{D}^{[B} \pi^{C]} = 0 ,
\end{gather}
and the prolongation of $\Gamma_\xi$ must lie in this intersection. Now, let us def\/ine $\bm{\psi}^{AB} := \bm{\phi}^{AB} - \pi^{[A} \bm{\phi}^{B]}$, where $\bm{\phi}^A$ and $\bm{\phi}^{AB}$ are the $1$-forms~\eqref{eq-contact-odd}. From the proof of Theorem~\ref{thm-odd-Kerr-theorem-I}, the $1$-forms $\bm{\psi}^{AB}$ and~$\bm{\phi}^A$ vanish on the locus~\eqref{eq-jetspace-mfd-odd}. On the other hand, on restriction to the locus \eqref{eq-jetspace-mfd-odd2}, we have $\bm{\psi}^{AB} = \bm{\alpha}^{AB} - \left( \rho_{C}^{AB} - \pi^{[A} \sigma_{C}^{B]} \right) \bm{\theta}^C$, where $\langle \bm{\alpha}^{AB}, \bm{\theta}^A \rangle$ annihilate the rank-$(2m+1)$ distribution $\Tgt^{-2}_E \F = \left\langle \bm{U} , \bm{W}_A , \bm{Z}^A \right\rangle$. One can further check that $\langle \bm{\psi}^{AB} , \bm{\phi}^A \rangle$ annihilate the $m+1$ vector f\/ields $\bm{U} + \left( \sigma^A - \sigma_B^A \pi^B \right) \bm{W}_A$ and $\bm{Z}^A$. These span a rank-$(m+1)$ subdistribution $\mathrm{L}$, say, of $\Tgt^{-2}_E \F$ tangent to $\Gamma_\xi$. By Theorem \ref{thm-odd-Kerr-theorem-I}, $\Gamma_\xi$ descends to an $(m+1)$-dimensional complex submanifold~$\mc{W}$ of~$\PT$. The quotient $\mathrm{L}/\Tgt^{-1}_E \F$ is a~rank-$1$ subbundle of $\Tgt^{-2}_E \F/\Tgt^{-1}_E \F$, which also descends to a~rank-$1$ subdistribution of $\Drm = \Tgt^{-1} \PT$ tangent to~$\mc{W}$. This proves the f\/irst part of the theorem.

Conversely, consider a complex submanifold $\mc{W}$ of $\PT$, transverse to every $\hat{x}$ in $\widehat{\mc{U}}$, given by the vanishing of holomorphic functions $(F^{AB}, F^A)$ on the chart $(\mc{V}_0,(\omega^0, \omega^A, \pi^A, \pi^{AB}))$. By Theorem~\ref{thm-odd-Kerr-theorem-I}, we can associate to $\mc{W}$ a local section $[\xi^\mbf{A}]$ of $\mc{U} \times \mc{U}_0 \subset \F$ with graph $\Gamma_\xi$, so that equations~\eqref{eq-jetspace-mfd-odd} hold. Assume further that the intersection of $\Tgt \mc{W}$ and $\left. \mathrm{D} \right|_\mc{W}$ is one-dimensional at every point. Then the pullbacks of $( \dd F^{AB} , \dd F^{A} )$ to $\mc{U} \times \mc{U}_0 \subset \F$ must vanish on $\Gamma_\xi$ and annihilate both $\Tgt^{-1}_E \F$ and a rank-$(m+1)$ subbundle of $\Tgt^{-2}_E \F \supset \Tgt^{-1}_E \F$. Thus, there exists a vector f\/ield $\bm{V} = \bm{U} + V^A \bm{W}_A $, for some holomorphic functions $V^A$ on $\Gamma_\xi$, annihilating the $1$-forms \eqref{eq-contact-odd}. It is then straightforward to check that this gives us precisely the additional restrictions \eqref{eq-jetspace-mfd-odd2}. In particular, $\pi^{AB} = \xi^{AB}(x)$ and $\pi^{A} = \xi^A(x)$ satisfy \eqref{eq-int-spinor2twistor-odd} and \eqref{eq-int-spinor2twistor+1}, i.e., the distribution associated to the graph $\Gamma_{\xi}$ is integrable and co-integrable.
\end{proof}

\begin{Remark}\label{rem-fol-1}
When $n=3$, Theorems \ref{thm-odd-Kerr-theorem-I} and \ref{thm-odd-Kerr-theorem-II} are equivalent: since $\PT$ is $3$-dimensional and $\mathrm{D}$ has rank $2$, any $2$-dimensional complex submanifold of $\PT$ satisfying the transversality property of the theorems must have non-trivial intersection with $\mathrm{D}$.
\end{Remark}

\subsubsection{Totally co-geodetic null structures}
Finally, we consider a totally co-geodetic null structure $N$. The key point here is that this stronger requirement is not conformally invariant, and for this reason, the appropriate arena is the mini-twistor space $\M\T$ of $\C \E^{2m+1}$. In this case, each leaf of the foliation of $N^\perp$ is totally geodetic, and must therefore be a co-$\gamma$-plane. The $m$-dimensional leaf space can then be identif\/ied as an $m$-dimensional complex submanifold~$\underline{\mc{W}}$ of~$\M\T$.

\looseness=-1 Alternatively, we can recycle the setting of Theorems \ref{thm-odd-Kerr-theorem-I} and \ref{thm-odd-Kerr-theorem-II}: since $N$ is in particular integrable and co-integrable, its leaf space is an $(m+1)$-dimensional complex submanifold $\mc{W}$ of $\PT_{\setminus \widehat{\infty}}$ foliated by curves. However, these curves are very particular since they correspond to totally geodetic leaves of $N^\perp$. Breaking of the conformal invariance can be translated into these curves being the integral curves of the vector f\/ield $\bm{Y}$ induced by the point $\infty$ on $\mc{Q}^n$. Quotienting the submanifold $\mc{W}$ by the f\/low of $\bm{Y}$ thus yields an $m$-dimensional complex submanifold~$\underline{\mc{W}}$ of~$\M\T$.

\begin{Theorem}\label{thm-odd-Kerr-theorem-NC}
A totally co-geodetic null structure on some open subset $\mc{U}$ of $\C \E^{2m+1}$ gives rise to an $m$-dimensional complex submanifold of $\underline{\widehat{\mc{U}}} \subset \M\T$ intersecting each $\hat{x} \subset \underline{\widehat{\mc{U}}}$ transversely for each~$x \in \mc{U}$. Conversely, any totally co-geodetic null structure locally arises in this way.
\end{Theorem}

\begin{proof}Suppose $N$ and $N^\perp$ are both integrable as in the previous section. As already pointed out the integral manifolds of $N$ are totally geodetic. We now impose the further assumption that the integral manifolds of $N^\perp$ are also totally geodetic on $\mc{U}$, i.e., $\bm{g} ( \nabla_{\bm{Z}^{A}} \bm{Z}^{B} , \bm{Z}^{C} ) = \bm{g} ( \nabla_{\bm{Z}^{A}} \bm{Z}^{B} , \bm{U} ) = \bm{g} ( \nabla_{\bm{U}} \bm{Z}^{B} , \bm{Z}^{C} ) = \bm{g} ( \nabla_{\bm{U}} \bm{Z}^{A} , \bm{U} ) = 0$. Then, in addition to \eqref{eq-int-spinor2twistor-odd}, we have
\begin{gather}\label{eq-int-spinor2twistor+2}
\big( \partial - \xi^{D} \partial _{D} \big) \xi^{AB} = 0 , \qquad \big( \partial - \xi^{D} \partial _{D} \big) \xi^{A} = 0 ,
\end{gather}
which can be seen to imply \eqref{eq-int-spinor2twistor+1}. As before, using the same notation as in the proof of Theorem~\ref{thm-odd-Kerr-theorem-I}, we express the system \eqref{eq-int-spinor2twistor-odd}, \eqref{eq-int-spinor2twistor+2} as a complex submanifold of~$\mc{J}^1(\C\E^n,\mc{U}_0)$ def\/ined by~\eqref{eq-jetspace-mfd-odd} and
\begin{gather}\label{eq-jetspace-mfd-odd-mt}
\rho^{AB} - \pi^{D} \rho_D^{AB} = 0 , \qquad \sigma^{A} - \pi^{D} \sigma _{D}^{A} = 0 .
\end{gather}
In particular, the $1$-forms $\dd \pi^{AB} - \rho_{C}^{AB} \bm{\theta}^C$ and $\dd \pi^{A} - \sigma_{C}^{A} \bm{\theta}^C$ vanish on the locus \eqref{eq-jetspace-mfd-odd} and~\eqref{eq-jetspace-mfd-odd-mt}, and this implies in particular that, for generic $\rho_{C}^{AB}$, $\rho^{BC}$, $\rho_{C}^{A}$, $\rho^{C}$, the section~$\Gamma_\xi$ must be constant along the f\/ibers of $\F \rightarrow \M\T$, i.e., the functions $(\xi^A,\xi^{AB})$ depend only on the coordinates $(\underline{\omega}^A,\pi^A,\pi^{AB})$ on the chart $\underline{\mc{V}}_0$ of $\M\T$. Thus, quotienting $\Gamma_\xi$ along the f\/ibers of $\F \rightarrow \M\T$ yields an $m$-dimensional complex submanifold of~$\M\T$.

For the converse, we simply run the argument backwards as in the proof of Theorem \ref{thm-odd-Kerr-theorem-I}.
\end{proof}

\subsection{Even dimensions}
The even-dimensional case is somewhat more tractable than the odd-dimensional case. For one, the orthogonal complement of an $\alpha$-plane or $\beta$-plane distribution $N$ \emph{is} $N$ itself, i.e., $N^\perp = N$. Def\/inition \ref{defn-int-co-geo} still applies albeit with much redundancy. In particular, $N$ is integrable if and only if it is co-integrable. The question now reduces to whether $N$ is integrable or not, and if so, whether it is totally geodetic. But it turns out that these two questions are equivalent.

\begin{Lemma}\label{lem-int2geod}
An almost null structure is integrable if and only if it is totally geodetic.
\end{Lemma}

For a proof, see for instance \cite{Taghavi-Chabert2012,Taghavi-Chabert2016}. The argument leading up to Theorem \ref{thm-odd-Kerr-theorem-I} equally applies to the even-dimensional case -- simply substitute $\gamma$-plane for $\alpha$-plane. For the sake of completeness, we restate the theorem, which was f\/irst used in four dimensions in \cite{Kerr2009}, reformulated in twistor language in \cite{Penrose1967}, and generalised to higher even dimensions in \cite{Hughston1988}. The proof of Theorem \ref{thm-odd-Kerr-theorem-I} can be recycled entirely by `switching of\/f' the coordinates $u$, $\omega^0$, $\pi^{A}$, and so on.

\begin{Theorem}[\cite{Hughston1988}]\label{thm-even-Kerr-theorem}
A self-dual null structure on some open subset $\mc{U}$ of $\mc{Q}^{2m}$ gives rise to an $m$-dimensional complex submanifold of $\widehat{\mc{U}} \subset \PT$ intersecting $\hat{x}$ in $\widehat{\mc{U}}$ transversely for each $x \in \mc{U}$. Conversely, any self-dual null structure locally arises in this way.
\end{Theorem}

\section{Examples}\label{sec-exa}
We now give two examples of co-integrable null structures that will illustrate the mechanism of Theorems \ref{thm-odd-Kerr-theorem-II} and \ref{thm-even-Kerr-theorem}. These arise in connections with conformal Killing spinors and conformal Killing--Yano $2$-forms, and are more transparently constructed in the language of tractor bundles reviewed in Section~\ref{sec-tractor}. As before, we work in the holomorphic category.

\subsection{Conformal Killing spinors}\label{sec-CKspinor}
For def\/initeness, let us stick to odd dimensions, i.e., $n=2m+1$. The even-dimensional case is similar. A (holomorphic) \emph{conformal Killing spinor} on $\mc{Q}^n$ is a section $\xi^{\mbf{A}}$ of $\mc{O}^{\mbf{A}}$ that satisf\/ies
\begin{gather}\label{eq-conformal-Killing-spinor_odd}
\nabla_a \xi^{\mbf{A}} + \tfrac{1}{\sqrt{2}} \bm{\upgamma} \ind{_a_{\mbf{B}}^{\mbf{A}}} \zeta^{\mbf{B}} = 0 ,
\end{gather}
where $\zeta^{\mbf{A}} = \frac{\sqrt{2}}{n} \bm{\upgamma} \ind{^a_{\mbf{B}}^{\mbf{A}}} \nabla_a \xi^{\mbf{B}}$ is a section of $\mc{O}^{\mbf{A}}[-1]$.

The prolongation of equation \eqref{eq-conformal-Killing-spinor_odd} is given by (see for instance \cite{Baum2010} and references therein)
\begin{gather}\label{eq-prolongation-confKillspin_odd}
\nabla_a \xi^{\mbf{A}} + \tfrac{1}{\sqrt{2}} \bm{\upgamma} \ind{_a_{\mbf{B}}^{\mbf{A}}} \zeta^{\mbf{B}} = 0 ,\qquad
\nabla_a \zeta^{\mbf{A}} + \tfrac{1}{\sqrt{2}} \Rho_{ab} \bm{\upgamma} \ind{^b_{\mbf{B}}^{\mbf{A}}} \xi^{\mbf{B}} = 0 .
\end{gather}
These equations are equivalent to the tractor spinor $\Xi^{\bm{\upalpha}} = ( \xi^{\mbf{A}} , \zeta^{\mbf{A}} )$ being parallel with respect to the tractor spinor connection, i.e., $\nabla_a \Xi^{\bm{\upalpha}} = 0$. In a conformal scale for which the metric is f\/lat, integration of~\eqref{eq-prolongation-confKillspin_odd} yields
\begin{gather}\label{eq-exp-twisor-field}
\xi^{\mbf{A}} = \mr{\xi}^{\mbf{A}} - \tfrac{1}{\sqrt{2}} x^a \gamma \ind{_a_{\mbf{B}}^{\mbf{A}}} \mr{\zeta}^{\mbf{B}} ,\qquad
\zeta^{\mbf{A}} = \mr{\zeta}^{\mbf{A}} ,
\end{gather}
where $\mr{\xi}^{\mbf{A}}$ and $\mr{\zeta}^{\mbf{A}}$ denote the constants of integrations at the origin.

A \emph{pure} conformal Killing spinor $\xi^{\mbf{A}}$ def\/ines an almost null structure. The following proposition combines results from \cite{Taghavi-Chabert2016,Taghavi-Chabert2013} recast in the language of tractors using Lemmata~\ref{lem-pure-ompi} and~\ref{lem-pure-ompi-even}. It is valid on any conformal manifold of any dimension.

\begin{Proposition}[\cite{Taghavi-Chabert2016,Taghavi-Chabert2013}]\label{prop-foliating-twistor-spinor}
The almost null structure of a pure conformal Killing spinor is locally integrable and co-integrable if and only if its associated tractor spinor is pure.
\end{Proposition}

By Theorems \ref{thm-odd-Kerr-theorem-II} and \ref{thm-even-Kerr-theorem} one can associate to any such conformal Killing spinor on~$\mc{Q}^n$ a~complex submanifold in $\PT$. These are described in the next two propositions.

\subsubsection{Odd dimensions}

\begin{Proposition}\label{prop-Robinson-congruence-odd}
 Let $\Xi^{\bm{\upalpha}}=(\xi^{\mbf{A}} , \zeta^{\mbf{A}})$ be a constant pure tractor spinor on $\mc{Q}^{2m+1}$, $\Xi$ its associated twistor in $\PT$, $\check{\Xi}$ its corresponding $\gamma$-plane in $\mc{Q}^{2m+1}$, and $\mc{U}:= \mc{Q}^{2m+1} \setminus \check{\Xi}$. Then $\xi^{\mbf{A}}$ is a pure conformal Killing spinor on $\mc{Q}^{2m+1}$ with zero set $\check{\Xi}$, and its associated integrable and co-integrable almost null structure $N_\xi$ on $\mc{U}$ arises from the submanifold $\mathbf{D}_\Xi \setminus \{ \Xi \}$ in $\widehat{\mc{U}} \subset \PT$, where $\mathbf{D}_\Xi$ is given by~\eqref{eq-lin-subsp}. In particular, each leaf of $N_\xi$ consists of a $\gamma$-plane intersecting $\check{\Xi}$ in an $(m-1)$-plane. Each leaf of $N_\xi^\perp$ consists of a $1$-parameter family of $\gamma$-planes intersecting in an $(m-1)$-plane. Any two $\gamma$-planes contained in two distinct leaves of $N_\xi^\perp$ intersect in an $(m-2)$-plane.
\end{Proposition}

\begin{proof}
The line spanned by $\Xi^{\bm{\upalpha}}$ descends to a point $\Xi$ (i.e., $[\Xi^{\bm{\upalpha}}]$) in $\PT$, and thus singles out a $\gamma$-plane $\check{\Xi}$ in $\mc{Q}^n$, which by \eqref{eq-exp-twisor-field} can be immediately identif\/ied with the zero set of $\xi^\mbf{A}$. Of\/f that set, Proposition \ref{prop-foliating-twistor-spinor} tells us that $N_\xi$ is integrable and co-integrable. Correspondingly, the conformal Killing spinor $\xi^{\mbf{A}}$ gives rise to a section $[\xi^{\mbf{A}}]$ of $\F$, which we can re-express as
\begin{gather*}
 \Gamma_{\xi} = \big\{ ([X^\mc{A}] , [Z^{\bm{\upalpha}}]) \in \mc{U} \times \PT \colon Z \ind{^{\bm{\upalpha}}} = X^\mc{A} \Xi _\mc{A} ^{\bm{\upalpha}} \big\} \subset \F .
\end{gather*}
Clearly, a point on $\Gamma_{\xi}$ descends to a twistor $Z$ on $\mathbf{D}_\xi \setminus \{ \Xi \}$ with $\gamma$-plane $\check{Z}$ tangent to~$N_\xi$. Thus, for each~$Z$ on~$\mathbf{D}_\Xi \setminus \{ \Xi \}$ in $\widehat{\mc{U}} \subset \PT$, $\check{Z}$ is precisely a leaf of~$N_\xi$. The point $\Xi$ itself must be excluded from~$\mathbf{D}_\Xi$ since the foliation becomes singular there in the sense the leaves intersect in~$\check{\Xi}$. The geometric interpretation of the leaves of~$N_\xi$ and $N_\xi^\perp$ follows directly from Theorem~\ref{thm-gam-inters} and Corollary~\ref{cor-D-geom}. In particular, each distinguished curve on $\mathbf{D}_\Xi$ can be identif\/ied with a leaf of~$N_\xi^\perp$.
\end{proof}

{\bf Local form.} Let us re-express the $(m+1)$-plane $\mathbf{D}_\Xi$ as \eqref{eq-lin-subsp-v2}. We work in a conformal scale for which~$g_{ab}$ is the f\/lat metric. Since $\Xi^{\bm{\upalpha}}$ is constant, we can substitute the f\/ields for their constants of integration at the origin. Using~\eqref{eq-Zompi} and $\Xi^{\bm{\upalpha}} = I^{\bm{\upalpha}}_{\mbf{A}} \mr{\xi}^{\mbf{A}} + O^{\bm{\upalpha}}_{\mbf{A}} \mr{\zeta}^{\mbf{A}}$, we obtain, in the obvious notation,
\begin{gather}\label{eq-alt-twistor-variety}\begin{split}
&\omega \ind{^a^{\mbf{A}}} \mr{\xi} \ind*{_a^{\mbf{B}}} + 2 \mr{\xi}^{\mbf{A}} \omega^{\mbf{B}} - \omega^{\mbf{A}} \mr{\xi}^{\mbf{B}} = 0 , \\
&\pi \ind{^a^{\mbf{A}}} \mr{\zeta} \ind{_a^{\mbf{B}}} + 2 \mr{\zeta}^{\mbf{A}} \pi^{\mbf{B}} - \pi^A \mr{\zeta}^{\mbf{B}} = 0 , \\
&\omega \ind{^a^{\mbf{A}}} \mr{\zeta} \ind*{_a^{\mbf{B}}} + \omega^{\mbf{A}} \mr{\zeta}^{\mbf{B}} + 4 \pi^{[{\mbf{A}}} \mr{\xi}^{{\mbf{B}}]} = 0 , \\
&\pi \ind{^a^{\mbf{A}}} \mr{\xi} \ind*{_a^{\mbf{B}}} + \pi^{\mbf{A}} \mr{\xi}^{\mbf{B}} + 4 \omega^{[{\mbf{A}}} \mr{\zeta}^{{\mbf{B}}]} = 0 .
\end{split}
\end{gather}
Evaluating at $\omega^{\mbf{A}} = \frac{1}{\sqrt{2}} x^a \gamma \ind{_a_{\mbf{B}}^{\mbf{A}}} \pi^{\mbf{B}}$, using the second and third of \eqref{eq-alt-twistor-variety} together with the purity of~$\Xi^{\bm{\upalpha}}$, we f\/ind that $\pi^{\mbf{A}}$ must be proportional to $\xi^{\mbf{A}} = \mr{\xi}^{\mbf{A}} - \frac{1}{\sqrt{2}} x^a \gamma \ind{_a_{\mbf{B}}^{\mbf{A}}} \mr{\zeta}^{\mbf{B}}$ as expected. This solution then satisf\/ies the f\/irst and fourth equations.

Let us now work in the coordinate chart $(\mc{V}_0 , ( \omega^0, \omega^A , \pi^A , \pi^{AB} )) $ as def\/ined in Section~\ref{sec-affine-odd}, and write
\begin{gather}\label{eq-hom2aff-Xi_odd}\begin{split}
&\mr{\xi}^{\mbf{A}} = \mr{\xi}^0 o^{\mbf{A}} + \ii \frac{1}{2} \mr{\xi}^A \delta_A^{\mbf{A}} - \frac{1}{4} \mr{\xi}^{AB} \delta_{AB}^{\mbf{A}} + \cdots , \\
&\mr{\zeta}^{\mbf{A}} = \frac{1}{\sqrt{2}} \left( \ii \mr{\zeta}^0 o^{\mbf{A}} + \mr{\zeta}^A \delta_A^{\mbf{A}} - \frac{\ii}{4 \mr{\xi}^0} \left( \mr{\xi}^{AB} \mr{\zeta}^0 - 2 \mr{\xi}^A \mr{\zeta}^B \right) \delta_{AB}^{\mbf{A}} + \cdots \right) ,
\end{split}
\end{gather}
where the remaining components of $\mr{\zeta}^{\mbf{A}}$ and $\mr{\xi}^{\mbf{A}}$ depend only on $\mr{\zeta}^0$, $\mr{\zeta}^A$, $\mr{\xi}^A$ and $\mr{\xi}^{AB}$ by the purity of $\Xi^{{\bm{\upalpha}}}$, and where we have assumed $\mr{\xi}^0 \neq 0$.
Substituting \eqref{eq-hom2aff_odd} and \eqref{eq-hom2aff-Xi_odd} into the last of equations \eqref{eq-alt-twistor-variety} yields
\begin{gather*}
\mr{\xi}^0 \pi^{A} - \mr{\xi}^{A} + \mr{\zeta}^0 \omega^{A} - \omega^0 \mr{\zeta}^{A} = 0 , \qquad
\mr{\xi}^0 \pi^{AB} - \mr{\xi}^{AB} + 2 \omega^{[A} \mr{\zeta}^{B]} = 0 ,
\end{gather*}
while the remaining equations do not yield any new information. Now, at every point $Z$ of $\mbf{D}_\Xi$, the $1$-forms
\begin{gather*}
\bm{\beta}^{A} := \mr{\xi}^0 \dd \pi^A + \mr{\zeta}^0 \dd \omega^A - \mr{\zeta}^A \dd \omega^0 , \qquad
\bm{\beta}^{AB} := \mr{\xi}^0 \dd \pi^{AB} + 2 \dd \omega^{[A} \mr{\zeta}^{B]} ,
\end{gather*}
annihilate the vectors tangent to $\mbf{D}_\Xi$ at $Z$ and the line in $\Drm_Z$ spanned by
\begin{gather}\label{eq-tgt2dist}
\bm{V} := V^0 \bm{Y} + V^A \bm{Y}_A ,
\end{gather}
where $V^0 := \mr{\xi}^0 + \frac{1}{2} \mr{\zeta}^0 \omega^0$ and $V^A := \mr{\zeta}^A + \frac{1}{2} \mr{\zeta}^0 \pi^A$. This corroborates the claims of Theorem~\ref{thm-odd-Kerr-theorem-II} and Proposition~\ref{prop-Robinson-congruence-odd}. Note that the vector f\/ield $\bm{V}$ vanishes at the point $[\Xi^{\mbf{\upalpha}}]$ of $\mbf{D}_\Xi$. With no loss, we can set $\mr{\zeta}^0 = -2$. The integral curve, with complex parameter~$t$, of~\eqref{eq-tgt2dist} passing through the point
\begin{gather*}
\big( \omega^0 , \omega^A , \pi^A , \pi^{AB} \big)\\
\qquad{} = \left( \mr{\xi}^0 + \alpha ,
- \frac{1}{2} \big( \mr{\xi}^A + \alpha \mr{\zeta}^A - \mr{\xi}^0 \alpha^A \big) ,
\mr{\zeta}^A + \alpha^A , \frac{1}{\mr{\xi}^0} \big( \mr{\xi}^{AB} + \mr{\xi}^{[A} \mr{\zeta}^{B]} \big) - \alpha^{[A} \mr{\zeta}^{B]} \right) ,
\end{gather*}
for some $\alpha$, $\alpha^A$, is given by
\begin{gather*}
\big( \omega^0 (t), \omega^A (t), \pi^A (t), \pi^{AB} (t)\big) = \left( \mr{\xi}^0 ,- \frac{1}{2} \mr{\xi}^A ,
\mr{\zeta}^A , \frac{1}{\mr{\xi}^0} \big( \mr{\xi}^{AB} + \mr{\xi}^{[A} \mr{\zeta}^{B]} \big) \right) \\
\hphantom{\big( \omega^0 (t), \omega^A (t), \pi^A (t), \pi^{AB} (t)\big) =}{}
+ \left( \alpha ,- \frac{1}{2} \big( \alpha \mr{\zeta}^A - \mr{\xi}^0 \alpha^A \big) ,
 \alpha^A , - \alpha^{[A} \mr{\zeta}^{B]} \right) \ee^{-t} ,
\end{gather*}
Writing $A^\mc{A} = a Y^\mc{A} + A^a Z_a^\mc{A} + b X^\mc{A}$ and $A^a = A^A \delta_A^a + A_A \delta^{aA} + A^0 u^a$ with
\begin{gather*}
\alpha = -a - \frac{1}{2}A^0 \mr{\xi}^0 + \frac{1}{2} A_C \mr{\xi}^C , \\
\alpha^A = \frac{1}{2 \mr{\xi}^0} \big( A_C \mr{\xi}^C \mr{\zeta}^A - A^0 \mr{\xi}^0 \mr{\zeta}^A - 2 \mr{\xi}^{AB} A_B - A^0 \mr{\xi}^A - 2 \mr{\xi}^0 A^A \big) , \\
b = \frac{1}{\mr{\xi}^0} \big(A^0 - \mr{\zeta}^C A_C \big) ,
\end{gather*}
one can recast this integral curve tractorially as $Z^{\bm{\upalpha}} (t) = \frac{\ii}{\sqrt{2}} \big( \mr{\Xi}^{\bm{\upalpha}} + \frac{\ii}{2} \ee^{-t} \mr{A}^{\mc{A}} \mr{\Xi}^{\bm{\upalpha}}_{\mc{A}} \big)$, which is one of the distinguished curves of Lemma~\ref{lem-fol-D} as expected.

\subsubsection{Even dimensions}
In even dimensions, the story is entirely analogous except for the choice of chirality of the tractor spinor. We leave the details to the reader.

\begin{Proposition}\label{prop-Robinson-congruence-even}
Let $\Xi^{\bm{\upalpha'}}=(\xi^{\mbf{A'}} , \zeta^{\mbf{A}})$ be a constant pure tractor spinor on $\mc{Q}^{2m}$, and let $\mc{U}:= \mc{Q}^{2m} \setminus \check{\Xi}$ where $\check{\Xi}$ is the $\beta$-plane defined by $\Xi^{\bm{\alpha'}}$. Then $\xi^{\mbf{A'}}$ is a pure conformal Killing spinor on $\mc{Q}^{2m}$, and its associated null structure $N_\xi$ on $\mc{U}$ arises from the submanifold in $\widehat{\mc{U}} \subset \PT$ defined by
\begin{gather}\label{eq-Kerr_in_Q_even}
\Gamma \ind*{^{(k)}_{\bm{\upalpha}}_{{\bm{\upbeta'}}}} Z \ind{^{\bm{\upalpha}}} \Xi \ind{^{{\bm{\upbeta'}}}} = 0 , \qquad \text{for} \quad k<m, \quad k \equiv m \pmod{2}.
\end{gather}
Each leaf of $N_\xi$ consists of an $\alpha$-plane intersecting $\check{\Xi}$ in an $(m-1)$-plane.
\end{Proposition}

\begin{Remark}In four dimensions, tractor-spinors are always pure, and so almost null structures associated to conformal Killing spinors are always integrable. In this case, the submani\-fold~\eqref{eq-Kerr_in_Q_even} is a complex projective hyperplane in $\PT \cong \CP^3$ given by $\Xi_{\bm{\upalpha}} Z^{\bm{\upalpha}}=0$ where we have used the canonical isomorphism $\PT^* \cong \PT'$. This example was highly instrumental in the genesis of twistor theory~\cite{Penrose1967}. The null structure arising from the intersection of this submanifold with \emph{real} twistor space generates a shearfree congruence of null geodesics in Minkowski space known as the \emph{Robinson congruence}.
\end{Remark}

\subsection{Conformal Killing--Yano 2-forms}
A (holomorphic) \emph{conformal Killing--Yano $($CKY$)$ $2$-form} on $\mc{Q}^n$ is a section $\sigma_{ab}$ of $\mc{O}_{[ab]} [3]$ that satisf\/ies
\begin{gather}\label{eq-CKY2}
 \nabla \ind{_a} \sigma \ind{_{bc}} - \mu \ind{_{abc}} - 2 \mbf{g} \ind{_{a [b}} \varphi \ind{_{c]}} = 0 ,
\end{gather}
where $\mu_{abc} = \nabla_{[a} \sigma_{bc]}$ and $\varphi_a = \frac{1}{n-2} \nabla^b \sigma_{ba}$.
The CKY $2$-form equation \eqref{eq-CKY2} is prolonged to the following system
\begin{gather}\label{eq-prolongation-CKY2}
\begin{split}
& \nabla \ind{_a} \sigma \ind{_{bc}} - \mu \ind{_{abc}} - 2 \mbf{g} \ind{_{a [b}} \varphi \ind{_{c]}} = 0 , \\
& \nabla \ind{_a} \mu \ind{_{bcd}} + 3 \mbf{g} \ind{_{a [b}} \rho \ind{_{cd]}} + 3 \Rho_{a[b} \sigma_{cd]} = 0 , \\
& \nabla \ind{_a} \varphi_b - \rho_{ab} + \Rho \ind{_a^c} \sigma_{cb} = 0 , \\
& \nabla \ind{_a} \rho_{bc} - \Rho \ind{_a^d} \mu_{dbc} + 2 \Rho_{a[b} \varphi_{c]} = 0 .
 \end{split}
\end{gather}
This system can be seen to be equivalent to the existence of a parallel tractor $3$-form, i.e.,
\begin{gather}
\nabla_a \Sigma_{\mc{A} \mc{B} \mc{C}} = 0 , \label{eq-parallel_tractor3form}
\end{gather}
where $\Sigma_{\mc{A}\mc{B}\mc{C}} := ( \sigma_{ab} , \mu_{abc} , \varphi_a , \rho_{ab} ) \in \mc{O}_{[\mc{A}\mc{B}\mc{C}]} \cong \mc{O}_{[ab]}[3] + ( \mc{O}_{[abc]}[3] \oplus \mc{O}_{a}[1] ) + \mc{O}_{[ab]}[1]$. For an arbitrary conformal manifold, equation~\eqref{eq-parallel_tractor3form} no longer holds in general, and necessitates the addition of a `deformation' term as explained in~\cite{Gover2008}.

In f\/lat space, i.e., with $\Rho_{ab}=0$, we can integrate equations \eqref{eq-prolongation-CKY2} to obtain
\begin{gather}\label{eq-CKY-integrated}
\begin{split}
& \sigma \ind{_{a b}} = \mr{\sigma} \ind{_{a b}} + 2 x_{[a} \mr{\varphi}_{b]} + \mr{\mu}_{abc} x^c - 2 \big( x_{[a} \mr{\rho}_{b]c} x^c + \tfrac{1}{4} \big(x^c x_c\big) \mr{\rho}_{ab} \big) , \\
&\mu \ind{_{abc}} = \mr{\mu} \ind{_{abc}} - 3 x \ind{_{[a}} \mr{\rho} \ind{_{b c]}} , \\
& \varphi \ind{_a} = \mr{\varphi} \ind{_a} - \mr{\rho} \ind{_{a b}} x \ind{^b} , \\
& \rho \ind{_{a b}} = \mr{\rho} \ind{_{a b}} ,
 \end{split}
\end{gather}
for some constants $\mr{\sigma} \ind{_{a b}}$, $\mr{\mu} \ind{_{abc}}$, $\mr{\varphi} \ind{_a}$ and $\mr{\rho} \ind{_{a b}}$.

\begin{Remark}In three dimensions, conformal Killing--Yano $2$-forms are Hodge dual to conformal Killing vector f\/ields. These latter are in one-to-one correspondence with parallel sections of tractor $2$-forms.

In four dimensions, a $2$-form $\sigma_{ab}$ is a CKY $2$-form if and only if its self-dual part $\sigma^+_{ab}$ and its anti-self-dual part $\sigma^-_{ab}$ are CKY $2$-forms, with, in the obvious notation, $\mu \ind*{^\pm_{abc}} = (*\varphi^\pm) \ind{_{abc}}$. Self-duality obviously carries over to tractor $3$-forms.
\end{Remark}

\subsubsection{Eigenspinors of a 2-form}
Let us f\/irst assume $n=2m+1$. We recall that an \emph{eigenspinor} $\xi^{\mbf{A}}$ of a $2$-form $\sigma_{ab}$ is a spinor satisfying
\begin{gather}\label{eq-eigenspinor-eq}
\sigma_{ab} \gamma \ind{^{ab}_{{\mbf{C}}}^{[{\mbf{A}}}} \xi^{{\mbf{B}}]} \xi^{{\mbf{C}}} = 0 ,
\end{gather}
i.e., $\sigma_{ab} \gamma \ind{^{ab}_{\mbf{C}}^{\mbf{A}}} \xi^{\mbf{C}} = \lambda \xi^{\mbf{A}}$ for some function $\lambda$. Here, $\gamma \ind{^{ab}_{\mbf{C}}^{\mbf{A}}} := \gamma \ind{^{[a}_{\mbf{C}}^{\mbf{B}}} \gamma \ind{^{b]}_{\mbf{B}}^{\mbf{A}}}$. When $\xi^{\mbf{A}}$ is pure, another convenient way to express the eigenspinor equation~\eqref{eq-eigenspinor-eq} is given by
\begin{gather*}
\sigma^{ab} \gamma \ind*{^{(m+1)}_{abc_3 \ldots c_{m+1}}_{{\mbf{AB}}}} \xi^{\mbf{A}} \xi^{\mbf{B}} = 0 .
\end{gather*}
Therefore, to any $2$-form $\sigma_{ab}$, we can associate a complex submanifold of $\F$ given by the graph
\begin{gather}\label{eq-CKY2-F}
\Gamma_{\sigma} := \big\{ \big(x^a , \big[\pi^{\mbf{A}}\big]\big) \in \C \E^n \times \PT_{(2m-1)}\colon \sigma^{ab} \gamma \ind*{^{(m+1)}_{abc_3 \ldots c_{m+1}}_{{\mbf{AB}}}} \pi^{\mbf{A}} \pi^{\mbf{B}} = 0 \big\} .
\end{gather}
For $\sigma_{ab}$ generic, this submanifold will have many connected components, each of which corresponding to a local section of $\F \rightarrow \mc{Q}^{2m+1}$, i.e., a projective pure spinor f\/ield that is an eigenspinor of~$\sigma_{ab}$. To be precise, in $2m+1$ dimensions, a generic $2$-form $\sigma_{ab}$ viewed as an endomor\-phism~$\sigma \ind{_a^b}$ of the tangent bundle, always has~$m$ distinct pairs of non-zero eigen\-values opposite to each other, i.e., $(\lambda,-\lambda)$, and a zero eigen\-value. In this case, a generic $2$-form viewed as an element of the Clif\/ford algebra has~$2^m$ distinct eigenvalues, and thus $2^m$ distinct eigenspaces, all of whose elements are pure~\cite{Mason2010}.

When $n=2m$, the analysis is very similar: the pure eigenspinor equation is now
\begin{gather*}
\sigma^{ab} \gamma \ind*{^{(m)}_{abc_3 \ldots c_m}_{\mbf{A'}\mbf{B'}}} \xi^{\mbf{A'}} \xi^{\mbf{B'}} = 0 ,
\end{gather*}
and similarity for spinors of the opposite chirality. Such a $2$-form generically has $m$ distinct pairs of non-zero eigenvalues opposite to each other, and as an element of the Clif\/ford algebra, has~$2^m$ eigenspaces that split into two sets of $2^{m-1}$ eigenspaces according to the chirality of the eigenspinors. The eigenspinor equation lifts to a submanifold $\Gamma_{\sigma} := \{ (x^a , [\pi^{{\mbf{A'}}}]) \in \C \E^n \times \PT_{(2m-2)} \colon \sigma^{ab} \gamma \ind*{^{(m)}_{abc_3 \ldots c_m}_{\mbf{A'}\mbf{B'}}} \pi^{{\mbf{A'}}} \pi^{{\mbf{B'}}} = 0 \}$ of $\F$, whose connected components correspond to the distinct primed spinor eigenspaces of $\sigma_{ab}$.

\subsubsection{The null structures of a conformal Killing--Yano 2-forms}
The next question to address is when the almost null structure of an eigenspinor of a $2$-form is integrable and co-integrable.
\begin{Proposition}[\cite{Mason2010}]\label{prop-foliating-CKY2}
Let $\sigma_{ab}$ be a generic conformal Killing $2$-form on~$\mc{Q}^n$ $($or any complex Riemannian manifold$)$. Let $\mu_{abc} := \nabla_{[a} \sigma_{bc]}$. Let $N$ be the almost null structure of some eigenspinor of~$\sigma_{ab}$, and suppose that $\mu_{abc} X^a Y^b Z^c = 0$ for any sections $X^a$, $Y^a$, $Z^a$ of $N^\perp$. Then $N$ is integrable and, when~$n$ is odd, co-integrable too.
\end{Proposition}

In the light of Theorems~\ref{thm-odd-Kerr-theorem-II} and~\ref{thm-even-Kerr-theorem}, the foliations arising from the eigenspinors of a CKY $2$-form $\sigma_{ab}$ can be encoded as complex submanifolds of the twistor space~$\PT$ of~$\mc{Q}^n$. As we shall see in a moment, these submanifolds can be constructed from the corresponding tractor~$\Sigma_{\mc{A}\mc{B}\mc{C}}$.

The additional condition on $\mu_{abc}$ in Proposition \ref{prop-foliating-CKY2} can also be understood in terms of the graph of a connected component of $\Gamma_\sigma$ def\/ined by \eqref{eq-CKY2-F}. For such a graph to descend to a complex submanifold of $\PT$, its def\/ining equations should be annihilated by the vectors tangent to $\F \rightarrow \PT$. Such a condition, in odd dimensions, can be expressed as $0 = \pi^{[\mbf{C}} \pi \ind{^c^{\mbf{D}]}} \nabla_c ( \sigma_{ab} \pi^{a\mbf{A}} \pi^{b\mbf{B}} )$, and using \eqref{eq-CKY2} gives $\mu_{abc} \pi^{a\mbf{A}} \pi^{b\mbf{B}} \pi^{b\mbf{C}} =0$. Thus, we shall be interested in the local sections of $\F \rightarrow \mc{Q}^n$ def\/ined by
\begin{gather}
\Gamma_{\sigma,\mu} := \big\{ \big(x^a , \big[\pi^{\mbf{A}}\big]\big) \in \C \E^n \times \PT_{(2m-1)} \colon\nonumber \\
\hphantom{\Gamma_{\sigma,\mu} := \big\{}{}
\sigma^{ab} \gamma \ind*{^{(m+1)}_{abc_3 \ldots c_{m+1}}_{{\mbf{AB}}}} \pi^{\mbf{A}} \pi^{\mbf{B}} = 0 ,\, \mu^{abc} \gamma \ind*{^{(m+1)}_{abcd_4 \ldots d_{m+1}}_{\mbf{AB}}} \pi^{\mbf{A}} \pi^{\mbf{B}} = 0 \big\} . \label{eq-sectionF-CKY2}
\end{gather}
In even dimensions, this is entirely analogous except that \eqref{eq-sectionF-CKY2} is now
\begin{gather*}
\Gamma_{\sigma,\mu} := \big\{ \big(x^a , \big[\pi^{{\mbf{A'}}}\big]\big) \in \C \E^n \times \PT_{(2m-2)} \colon \\
\hphantom{\Gamma_{\sigma,\mu} := \big\{}{}
\sigma^{ab} \gamma \ind*{^{(m)}_{abc_3 \ldots c_m}_{{\mbf{A'}}{\mbf{B'}}}} \pi^{{\mbf{A'}}} \pi^{{\mbf{B'}}} = 0 , \, \mu \ind{^{abc}} \gamma \ind*{^{(m)}_{abcd_4 \ldots d_{m}}_{\mbf{A'}\mbf{B'}}} \pi^{\mbf{A'}} \pi^{\mbf{B'}} = 0 \big\} .
\end{gather*}

\begin{Proposition}\label{prop-Kerr-variety-O}
Set $n=2m+ \epsilon$, where $\epsilon \in \{0,1\}$. Let $\sigma_{ab}$ be a generic conformal Killing--Yano $2$-form on some open subset $\mc{U}$ of $\mc{Q}^n$, with associated tractor $3$-form $\Sigma_{\mc{A} \mc{B} \mc{C}}$. Then if the almost null structure associated to some eigenspinor of $\sigma_{ab}$ is integrable and co-integrable, it must arise from the submanifold in $\widehat{\mc{U}} \subset \PT$ defined by
\begin{gather}\label{eq-Kerr-variety-O}
 \Sigma^{\mc{A} \mc{B} \mc{C}} \Gamma \ind*{^{(m+1+\epsilon)}_{\mc{A} \mc{B} \mc{C} \mc{D}_4 \ldots \mc{D}_{m+1+\epsilon}} _{{\bm{\upalpha}} {\bm{\upbeta}}}} Z^{\bm{\upalpha}} Z^{\bm{\upbeta}} = 0 .
\end{gather}
\end{Proposition}

\begin{proof}We focus on the odd-dimensional case only, and leave the even-dimensional case to the reader. Let us write
\begin{gather*}
\Sigma _{\mc{A} \mc{B} \mc{C}} = 3 Y _{[\mc{A}} Z _\mc{B}^b Z _{\mc{C}]}^c \sigma_{bc} + \big( Z _\mc{A}^a Z _\mc{B}^b Z _\mc{C}^c \mu_{abc} + 6 X _{[\mc{A}} Y _\mc{B} Z _{\mc{C}]}^c \varphi_c \big) + 3 X _{[\mc{A}} Z _\mc{B}^b Z _{\mc{C}]}^c \rho_{bc} .
\end{gather*}
Since $\Sigma _{\mc{A} \mc{B} \mc{C}}$ is constant, we can substitute the f\/ields for their constants of integration at the origin, $\mr{\sigma} \ind{_{a b}}$, $\mr{\mu} \ind{_{abc}}$, $\mr{\varphi} \ind{_a}$ and $\mr{\rho} \ind{_{a b}}$, so that using~\eqref{eq-Zompi} we can re-express~\eqref{eq-Kerr-variety-O} as
\begin{gather*}
0 = - 3 \sqrt{2} \mr{\sigma}^{ab} \gamma \ind*{^{(m+1)}_{abd_4 \ldots d_{m+2}}_{\mbf{AB}}} \pi^\mbf{A} \pi^\mbf{B} + 2 \mr{\mu}^{abc} \gamma \ind*{^{(m+2)}_{abcd_4 \ldots d_{m+2}\mbf{AB}}} \omega^\mbf{A} \pi^\mbf{B} - 12 \mr{\varphi}^a \gamma \ind*{^{(m)}_{ad_4 \ldots d_{m+2}\mbf{AB}}} \omega^\mbf{A} \pi^\mbf{B} \\
 \hphantom{0=}{}+ 3 \sqrt{2} \mr{\rho}^{ab} \gamma \ind*{^{(m+1)}_{abd_4 \ldots d_{m+2}\mbf{AB}}} \omega^\mbf{A} \omega^\mbf{B} , \\
0 = \sqrt{2} \mr{\mu}^{abc} \gamma \ind*{^{(m+1)}_{abcd_4 \ldots d_{m+1}}_{\mbf{AB}}} \pi^\mbf{A} \pi^\mbf{B} - 6 \mr{\rho}^{ab} \gamma \ind*{^{(m)}_{abd_4 \ldots d_{m+1} \mbf{AB}}} \omega^\mbf{A} \pi^\mbf{B} , \\
0 = - \sqrt{2} \mr{\mu}^{abc} \gamma \ind*{^{(m+1)}_{abcd_4 \ldots d_{m+1}}_{\mbf{AB}}} \omega^\mbf{A} \omega^\mbf{B} + 6 \mr{\sigma}^{ab} \gamma \ind*{^{(m)}_{abd_4 \ldots d_{m+1} \mbf{AB}}} \omega^\mbf{A} \pi^\mbf{B} , \\
0 = 2 \mr{\mu}^{abc} \gamma \ind*{^{(m)}_{abcd_4 \ldots d_m}_{\mbf{AB}}} \omega^\mbf{A} \pi^\mbf{B}.
\end{gather*}
Evaluating this system of equations on the intersection of \eqref{eq-Kerr-variety-O} and $\widehat{\mc{U}}$ amounts to setting $\omega^\mbf{A} = \frac{1}{\sqrt{2}} x^a \gamma \ind{_a_{\mbf{B}}^{\mbf{A}}} \pi^\mbf{B}$, and we f\/ind, after some algebraic manipulations,
\begin{gather*}
0 = - 3 \sqrt{2} \big( \sigma^{ab} \gamma \ind*{^{(m+1)}_{abd_4 \ldots d_{m+2}}_{\mbf{AB}}} \pi^\mbf{A} \pi^\mbf{B} \big)
+ \sqrt{2} (m-1) \big( x_{[d_4|} \mu^{abc} \gamma \ind*{^{(m+1)}_{abc|d_5 \ldots d_{m+2}]\mbf{AB}}} \pi^\mbf{A} \pi^\mbf{B} \big) , \\
0 = \sqrt{2} \mu^{abc} \gamma \ind*{^{(m+1)}_{abcd_4 \ldots d_{m+1}}_{\mbf{AB}}} \pi^\mbf{A} \pi^\mbf{B} , \\
0 = - \frac{(x^e x_e)}{\sqrt{2}} \mu^{abc} \gamma \ind*{^{(m+1)}_{abcd_4 \ldots d_{m+1}}_{\mbf{AB}}} \pi^\mbf{A} \pi^\mbf{B} + 3 \sqrt{2} \sigma^{ab} x^c \gamma \ind*{^{(m+1)}_{abcd_4 \ldots d_{m+1} \mbf{AB}}} \pi^\mbf{B} \pi^\mbf{B} \\
\hphantom{0=}{} + \sqrt{2} (m-2) x_{[d_4|} \mu^{abc} x^f \gamma \ind*{^{(m+1)}_{abcf|d_5 \ldots d_{m+1}]}_{\mbf{AB}}} \pi^\mbf{B} \pi^\mbf{B} , \\
0 = \sqrt{2} \mu^{abc} x^d \gamma \ind*{^{(m+1)}_{abcde_5 \ldots e_{m+1}}_{\mbf{AB}}} \pi^\mbf{A} \pi^\mbf{B} ,
\end{gather*}
where we have made use of \eqref{eq-CKY-integrated} and the identity
\begin{gather*}
\tfrac{1}{4} \big( x^c \gamma \ind{_c_{\mbf{C}}^{\mbf{A}}} \big) \big(\mr{\rho}_{ab} \gamma \ind{^{ab}_{\mbf{A}}^{\mbf{B}}} \big) \big( x^d \gamma \ind{_d_{\mbf{B}}^{\mbf{D}}} \big) = \big( x_{a} \mr{\rho}_{bc} x^c + \tfrac{1}{4} (x^c x_c) \mr{\rho}_{ab} \big) \gamma \ind{^{ab}_{\mbf{C}}^{\mbf{D}}} .
\end{gather*}
In particular, we immediately recover, that on the intersection of the twistor submanifold \eqref{eq-Kerr-variety-O} with $\widehat{\mc{U}}$,
\begin{gather*}
\sigma^{ab} \gamma \ind*{^{(m+1)}_{abc_3 \ldots c_{m+1}}_{\mbf{AB}}} \pi^{\mbf{A}} \pi^{\mbf{B}} = 0 , \qquad \mu \ind{^{abc}} \gamma \ind*{^{(m+1)}_{abcd_4 \ldots d_{m+1}}_{\mbf{AB}}} \pi^{\mbf{A}} \pi^{\mbf{B}} = 0 .
\end{gather*}
But these are precisely the zero set \eqref{eq-sectionF-CKY2} corresponding to the eigenspinors of $\sigma_{ab}$.
\end{proof}

\begin{Remark}In three dimensions, the twistor submanifold is simply a smooth quadric in $\PT \cong \CP^3$.

In four dimensions, the submanifold~\eqref{eq-Kerr-variety-O} restricts to an anti-self-dual tractor $3$-form $\Sigma^- _{\mc{A} \mc{B} \mc{C}}$ corresponding to a self-dual CKY $2$-form $\sigma_{ab}$. Setting $\Sigma^-_{{\bm{\upalpha}} {\bm{\upbeta}}} := \Sigma^- _{\mc{A} \mc{B} \mc{C}} \Gamma \ind{^{\mc{A} \mc{B} \mc{C}}_{{\bm{\upalpha}} {\bm{\upbeta}}}}$, we recover the quadratic polynomial $\Sigma^-_{{\bm{\upalpha}} {\bm{\upbeta}}} Z \ind{^{\bm{\upalpha}}} Z \ind{^{\bm{\upbeta}}} = 0$ given in \cite{Penrose1986}. Under appropriate reality conditions, this submanifold produces a~shearfree congruence of null geodesics in Minkowski space known as the \emph{Kerr congruence}. A~suitable perturbation of Minkowski space by the generator of such a~congruence leads to the solution of Einstein's equations known as the \emph{Kerr metric}~\cite{Kerr1963,Kerr2009}. A~Euclidean analogue is also given in~\cite{Salamon2009}.

In six dimensions, we have a splitting of $\mu_{abc} = \mu^+_{abc} + \mu^-_{abc}$ into a self-dual part and an anti-self-dual part. Since $\xi^{a\mbf{A}} \xi^{b\mbf{B}} \xi^{c\mbf{C}} \mr{\mu}^+_{abc} = 0$ for any $\xi^{\mbf{A}'}$, the obstruction to the integrability of a positive eigenspinor of a generic CKY $2$-form $\sigma_{ab}$ is the anti-self-dual part $\mu^-_{abc}$ of $\mu_{abc}$.
\end{Remark}

\section{Curved spaces}\label{sec-curved}
Let $\mc{M}$ be a complex manifold equipped with a holomorphic non-degenerate symmetric bi\-li\-near form $g_{ab}$. The pair $(\mc{M},g_{ab})$ will be referred to as a \emph{complex Riemannian manifold}. We assume that $\mc{M}$ is equipped with a holomorphic complex orientation and a holomorphic spin structure. We may also assume that one merely has a~holomorphic conformal structure rather than a~metric one. For def\/initeness, we set $n=2m+1$ as the dimension of~$\mc{M}$. The analogue of the correspondence space~$\F$ is the projective pure spinor bundle $\nu\colon \mc{F} \rightarrow \mc{M}$: for any $x \in \mc{M}$, a~point~$p$ in a f\/iber $\nu^{-1}(x)$ is a~totally null $m$-plane in $\Tgt_x \mc{M}$, and sections of $\mc{F}$ are almost null structures on $\mc{M}$. To def\/ine the twistor space of $(\mc{M},g_{ab})$, one must replace the notion of $\gamma$-plane by that of \emph{$\gamma$-surface}, i.e., an $m$-dimensional complex submanifold of~$\mc{M}$ such that at any point of such a surface, its tangent space is totally null with respect to the metric and totally geodetic with respect to the metric connection. The integrability condition for the existence of a~$\gamma$-surface~$\mc{N}$ through a~point~$x$ is~\cite{Taghavi-Chabert2013}
\begin{gather}\label{eq-int-cond-NS}
C_{abcd} X^a Y^b Z^c W^d = 0 , \qquad \text{for all} \quad X^a, Y^a, Z^c \in \Tgt_x \mc{N}, \quad W^a \in \Tgt_x \mc{N}.
\end{gather}
If we def\/ine the twistor space of $(\mc{M},g_{ab})$ to be the $\frac{1}{2}(m+1)(m+2)$-dimensional complex manifold parametrising the $\gamma$-surfaces of $(\mc{M},g_{ab})$, we must have a $\frac{1}{2}m(m+1)$-parameter family of $\gamma$-surfaces through each point of $\mc{M}$. From the integrability condition \eqref{eq-int-cond-NS}, we must conclude that for the twistor space of $(\mc{M},g_{ab})$ to exist, $(\mc{M},g_{ab})$ must be conformally f\/lat in odd dimensions greater than three. In even dimensions the story is similar: one replaces the notion of $\alpha$-plane by that of an $\alpha$-surface in the obvious way. We then f\/ind that for $(\mc{M},g_{ab})$ to admit a twistor space, it must be conformally f\/lat in even dimensions greater than four, and anti-self-dual in dimension four.

Curved twistor theory in dimensions three and four is pretty well-known. In dimension four, we have the \emph{Penrose correspondence}, whereby twistor space is a three-dimensional complex manifold containing a complete analytic family of rational curves with normal bundle $\mc{O}(1)\oplus \mc{O}(1)$ parameterised by the points of an anti-self-dual complex Riemannian manifold \cite{Penrose1976}. In dimension three, the \emph{LeBrun correspondence} can be seen as a special case of the Penrose correspondence: if we endow twistor space with a holomorphic `twisted' contact structure, then a three-dimensional conformal manifold arises as the umbilic conformal inf\/inity of an Einstein anti-self-dual four-dimensional complex Riemannian manifold \cite{LeBrun1982}. Finally, the mini-twistor space in the \emph{Hitchin correspondence} is a two-dimensional complex manifold containing a complete analytic family of rational curves with normal bundle $\mc{O}(2)$ parameterised by the points of an Einstein--Weyl space~\mbox{\cite{Hitchin1982a,Jones1985}}.

Theorems \ref{thm-odd-Kerr-theorem-I} (or \ref{thm-odd-Kerr-theorem-II}), \ref{thm-odd-Kerr-theorem-NC} and \ref{thm-even-Kerr-theorem} can be adapted to the curved setting by interpreting the leaf space of a totally geodetic null foliation as a complex submanifold of twistor space. See~\cite{Calderbank2000} for an application of a `curved' Theorem \ref{thm-odd-Kerr-theorem-NC} in the investigation of three-dimensional Einstein--Weyl spaces.

\appendix
\section{Coordinate charts on twistor space and correspondence space}\label{app-cover}
In this appendix, we construct atlases of coordinates charts covering $\PT$ and $\F$. We refer to the setup of Section~\ref{sec-geo-bkgd} throughout. In particular, we work with the splittings \eqref{eq-split-V}, \eqref{eq-S->S1/2-odd} and \eqref{eq-S->S1/2-even}.

\subsection{Odd dimensions}\label{sec-affine-odd}
Let us introduce a splitting of $\V_0$ as
\begin{gather}\label{eq-split-WW*U}
\V_0 \cong \mbb{W} \oplus \mbb{W}^* \oplus \mbb{U} ,
\end{gather}
where $\mbb{W} \cong \C^m$ is a totally null $m$-plane of $(\V_0,g_{ab})$, and $\mbb{U} \cong \C$ is the one-dimensional complement of $\mbb{W}\oplus \mbb{W}^*$ in $\V_0$. Elements of $\mbb{W}$ and $\mbb{W}^*$ will carry upstairs and downstairs upper-case Roman indices respectively, i.e., $V^A \in \mbb{W}$, and $W _A \in \mbb{W}^*$. The vector subspace $\mbb{U}$ will be spanned by a unit vector $u^a$. Denote by $\delta^{aA}$ the injector from $\mbb{W}^*$ to $\V_0$, and $\delta^a_A$ the injector from $\mbb{W}$ to $\V_0$ satisfying $\delta_a^A \delta^a_B = \delta_B^A$, where $\delta_B^A$ is the identity on $\mbb{W}$ and $\mbb{W}^*$. We shall think of $\{ \delta^{aA} \}$ as a basis for $\mbb{W}$ with dual basis $\{ \delta^a_A \}$ for $\mbb{W}^*$. The splitting~\eqref{eq-split-WW*U} allows us to identify the two copies $\Ss_{\pm\frac{1}{2}}$ of the spinor space of $(\V_0 ,g_{ab})$ with its Fock representation, i.e.,
\begin{gather*}
\Ss_{\pm\frac{1}{2}} \cong \wedge^m \mbb{W} \oplus \wedge^{m-1} \mbb{W} \oplus \cdots \oplus \mbb{W} \oplus \C .
\end{gather*}
This is essentially the strategy adopted in Section~\ref{sec-tw-sp} for the spinors of $\Spin(2m+3,\C)$. To realise it explicitly, we proceed as follows: let $o^{\mbf{A}}$ be a (pure) spinor annihilating $\mbb{W}$ so that $o^{\mbf{A}}$ is a spanning element of $\wedge^m \mbb{W}$. A (Fock) basis for $\Ss_{\pm\frac{1}{2}}$ can then be produced by acting on $o^{\mbf{A}}$ by basis elements of $\wedge^\bullet \mbb{W}^*$, i.e.,
\begin{gather}\label{eq-spin-basis-odd}
\Ss_{\pm\frac{1}{2}} = \big\langle o^{\mbf{A}} , \delta^{\mbf{A}}_{A_1} , \delta^{\mbf{A}}_{A_1A_2} , \ldots \big\rangle ,
\end{gather}
where
\begin{gather*} \delta^{\mbf{A}}_{A_1 \ldots A_k} := \delta_{[A_1}^{a_1} \cdots \delta_{A_k]}^{a_k} o^{\mbf{A_0}} \gamma \ind{_{a_1}_{\mbf{A_0}}^{\mbf{A_1}}} \cdots \gamma \ind{_{a_k}_{\mbf{A_{k-1}}}^{\mbf{A}}} ,
\end{gather*}
for each $k=1, \ldots , m$. With this notation, the Clif\/ford multiplication of $\V_0 \subset \Cl(\V_0 ,g_{ab})$ on~$\Ss_{-\frac{1}{2}}$ is given explicitly by
\begin{alignat}{3}
& \delta^{aA} \gamma \ind{_a_{\mbf{B}}^{\mbf{C}}} \delta \ind*{^{\mbf{B}}_{B_1 \ldots B_p}} = -2 p \delta \ind*{^{\mbf{C}}_{[B_1 \ldots B_{p-1}}} \delta \ind*{^{A}_{{B_p}]}} , \qquad && \delta^a_{A} \gamma \ind{_a_{\mbf{B}}^{\mbf{C}}} \delta \ind*{^{\mbf{B}}_{B_1 \ldots B_p}} = \delta \ind*{^{\mbf{C}}_{B_1 \ldots B_{p}A}} , &\nonumber \\
& u^a \gamma \ind{_a_{\mbf{B}}^{\mbf{C}}} o^{\mbf{B}} = \ii o^{\mbf{C}} , \qquad &&
u^a \gamma \ind{_a_{\mbf{B}}^{\mbf{C}}} \delta \ind*{^{\mbf{B}}_{B_1 \ldots B_p}} = (-1)^p \ii \delta \ind*{^{\mbf{C}}_{B_1 \ldots B_p}} .& \label{eq-Clifford-multiplication-odd}
\end{alignat}
An arbitrary spinor $\pi^{\mbf{A}}$ in $\Ss_{\frac{1}{2}}$ can then be expressed in the Fock basis~\eqref{eq-spin-basis-odd} as
\begin{gather}
\pi^{\mbf{A}} = \pi^0 o^{\mbf{A}}
+ \sum_{k=1}^{[m/2]}\left(- \frac{1}{4}\right)^k \frac{1}{k!} \pi^{A_1\ldots A_{2k}} \delta_{A_1\ldots A_{2k}}^{\mbf{A}}\nonumber\\
\hphantom{\pi^{\mbf{A}} =}{} + \frac{\ii}{2} \sum_{k=0}^{[m/2]}\left(- \frac{1}{4}\right)^k \frac{1}{k!} \pi^{A_1\ldots A_{2k+1}} \delta_{A_1\ldots A_{2k+1}}^{\mbf{A}} , \qquad m > 1 , \nonumber\\
\pi^{\mbf{A}} = \pi^0 o^{\mbf{A}} + \tfrac{\ii}{2} \pi^{A} \delta_{A}^{\mbf{A}} , \qquad m = 1 ,\label{eq-hom2gen-pi_odd}
\end{gather}
where $\left[\frac{m}{2}\right]$ is $\frac{m}{2}$ when $m$ is even, $\frac{m-1}{2}$ when $m$ is odd, and $\pi^0$ and $\pi^{A_1A_2\ldots A_{k}} = \pi^{[A_1A_2\ldots A_{k}]}$ are the components of $\pi^\mbf{A}$. Let us now assume that $\pi^{\mbf{A}}$ is pure, i.e., satisf\/ies~\eqref{eq-pi-pure}. When $m=1$ and $2$, there are no algebraic constraints, and the space of projective pure spinors is isomorphic to~$\CP^1$ and~$\CP^3$ respectively. When $m>2$, the pure spinor variety is then given by the complete intersections of the quadric hypersurfaces
\begin{alignat}{3}
& \pi^0 \pi^{A_1 A_2 \ldots A_{2k+1}} = \pi^{[A_1} \pi^{A_2 \ldots A_{2k+1}]} , \qquad & & k=1, \ldots , [m/2] ,&\nonumber \\
& \pi^0 \pi^{A_1 A_2 A_3 \ldots A_{2k}} = \pi^{[A_1 A_2} \pi^{A_3 \ldots A_{2k}]} , \qquad & & k=1, \ldots , [m/2] ,&\label{eq-purity-cond-coor-odd}
\end{alignat}
in $\C \Pp^{2^m-1}$. We can therefore cover a f\/ibre of $\F$ with $2^m$ open subsets $\mc{U}_0$, $\mc{U}_{A_1 \ldots A_{k}}$, where $\pi^0 \neq 0$ on $\mc{U}_0$ and $\pi^{A_1 \ldots A_{k}} \neq 0$ on $\mc{U}_{A_1 \ldots A_{k}}$, and thus obtain $2^m$ coordinate charts in the obvious way. This induces an atlas of charts on $\F_{\C\E^n}$ given by the open subsets $\C \E^n \times \mc{U}_0$, $\C \E^n \times \mc{U}_{A_1 \ldots A_{k}}$.

Let us now write the spinor $\omega^\mbf{A}$ in $\Ss_{-\frac{1}{2}}$ in the Fock basis as
\begin{gather}
\omega^{\mbf{A}} = \frac{\ii}{\sqrt{2}} \omega^0 o^{\mbf{A}} + \frac{1}{\sqrt{2}} \omega^A \delta_A^{\mbf{A}} , \qquad m = 1 ,\nonumber \\
\omega^{\mbf{A}} = \frac{\ii}{\sqrt{2}} \omega^0 o^{\mbf{A}}
+ \frac{\ii}{2\sqrt{2}} \sum_{k=1}^{[m/2]}\left(- \frac{1}{4}\right)^{k-1} \frac{1}{(k-1)!} \omega^{A_1\ldots A_{2k}} \delta_{A_1A_2\ldots A_{2k}}^{\mbf{A}}
\nonumber\\
\hphantom{\omega^{\mbf{A}} =}{} + \frac{1}{\sqrt{2}}\sum_{k=0}^{[m/2]}\left(- \frac{1}{4}\right)^k \frac{1}{k!} \omega^{A_1 \ldots A_{2k+1}} \delta_{A_1 \ldots A_{2k+1}}^{\mbf{A}} , \qquad m > 1 ,\label{eq-hom2gen-om_odd}
\end{gather}
where $\omega^0$ and $\omega^{A_1A_2\ldots A_{k}} = \omega^{[A_1A_2\ldots A_{k}]}$ are the components of~$\omega^\mbf{A}$. The condition for $Z^{\bm{\upalpha}}=(\omega^\mbf{A}, \pi^\mbf{A})$ to be pure, so that~\eqref{eq-Z-pure} hold, is that the relations
\begin{gather*}
\pi^0 \omega^{A_1 \ldots A_{2k-1} A_{2k}} = \pi^{[A_1 \ldots A_{2k-1}} \omega^{A_{2k}]} - \tfrac{1}{2k} \pi^{A_1 \ldots A_{2k}} \omega^0 , \\
\pi^0 \omega^{A_1 \ldots A_{2k} A_{2k+1}} = \pi ^{[A_1 \ldots A_{2k}} \omega^{A_{2k+1}]} ,
\end{gather*}
hold for $k\geq1$ when $m>1$, and that \eqref{eq-purity-cond-coor-odd} hold too when $m>2$. Hence, we can cover $\PT_{\setminus \widehat{\infty}}$ with $2^m$ open subsets $\mc{V}_0$, where $\pi^0 \neq 0$, and $\mc{V}_{A_1 \ldots A_{k}}$ where $\pi^{A_1 \ldots A_{k}} \neq 0$ in the obvious way. Coordinates on the complement $\widehat{\infty}$ parametrised by $[\omega^\mbf{A},0]$ satisfy the conditions
\begin{gather*}
\omega^0 \omega^{A_1 \ldots A_{2k} A_{2k+1}}= - 2k \omega^{[A_1 \ldots A_{2k}} \omega^{A_{2k+1}]},\qquad \omega^{[A_1 \ldots A_{2k-1}} \omega^{A_{2k}]} = 0 .
\end{gather*}

Let $( z^A,z_A, u )$ be null coordinates on $\C \E^n$ in the sense that $x^a = z^A \delta^a_A + z_A \delta^{aA} + u u^a$ so that the f\/lat metric on $\C \E^n$ takes the form $\bm{g} = 2 \dd z^{A} \odot \dd z_{A} + \dd u \otimes \dd u$. Then the incidence rela\-tion~\eqref{eq-incidence_relation-spinor_odd} reads
\begin{gather*}
\omega^0 = \pi^0 u - \pi^B z_B , \\
\omega^A = \pi^0 z^A + \pi^{AB} z_B + \tfrac{1}{2} \pi^A u , \\
\omega^{A_1 \ldots A_{2k-1} A_{2k}} = \pi^{[A_1 \ldots A_{2k-1}} z^{A_{2k}]} + \tfrac{4k+2}{4k} \pi^{A_1 \ldots A_{2k-1} A_{2k} A_{2k+1}} z_{A_{2k+1}} - \tfrac{1}{2k} \pi^{A_1 \ldots A_{2k}} u , \\
\omega^{A_1 \ldots A_{2k} A_{2k+1}} = \pi^{[A_1 \ldots A_{2k}} z^{A_{2k+1}]} + \pi^{A_1 \ldots A_{2k} A_{2k+1} A_{2k+2}} z_{A_{2k+2}} + \tfrac{1}{2} \pi^{A_1 \ldots A_{2k+1}} u .
\end{gather*}

We now work in the chart $\mc{U}_0$, and since $\pi^0 \neq 0$ there, we can set with no loss of ge\-ne\-rality \mbox{$\pi^0=1$}. Let $(x,\pi)$ be a point in $\F_{\C\E^n}$ and let $(\mc{U}_0,(\pi^A,\pi^{AB}))$ be a coordinate chart containing $\pi \in \F_x$. Let $(\omega,\pi)$ be the image of $(x,\pi)$ under the projection $\mu\colon \F \rightarrow \PT$ so that $(\mc{V}_0 , (\omega^0,\omega^A,\pi^A,\pi^{AB}))$ is a~coordinate chart containing~$(\omega,\pi)$. Then, in these charts, \eqref{eq-hom2gen-om_odd}~and~\eqref{eq-hom2gen-pi_odd} reduce to
\begin{subequations}\label{eq-hom2aff_odd}
\begin{gather}
\omega^\mbf{A} = \tfrac{1}{\sqrt{2}} \big( \ii \omega^0 o^\mbf{A}+ \omega^{A} \delta_{A}^\mbf{A} - \tfrac{\ii}{4} \big( \pi^{AB} \omega^0 - 2 \pi^{A} \omega^{B} \big) \delta_{AB}^\mbf{A} + \cdots \big) , \label{eq-hom2aff_odd-om}\\
\pi^\mbf{A} = o^\mbf{A} + \tfrac{\ii}{2} \pi^A \delta_A^\mbf{A} - \tfrac{1}{4} \pi^{AB} \delta_{AB}^\mbf{A} + \cdots .\label{eq-hom2aff_odd-pi}
\end{gather}
\end{subequations}
More succinctly, $\pi^\mbf{A} = \exp ( - \frac{1}{4} \pi^{ab} \gamma \ind{_{ab}_{\mbf{B}}^{\mbf{A}}} ) o \ind{^{\mbf{B}}}$, where $\pi^{ab} = \pi^{AB} \delta^a_A \delta^b_B + 2 \pi^A \delta^{[a}_A u^{b]}$ belongs to the complement of the stabiliser of $o^{\mbf{A}}$ in $\so(\V_0,g_{ab})$, i.e., $(\pi^A,\pi^{AB})$ are coordinates on a dense open subset of the homogeneous space $P/Q$. We can also rewrite $\omega^\mbf{A}$ more compactly in the two alternative forms{\samepage
\begin{gather*}
\omega^\mbf{A} = \tfrac{1}{\sqrt{2}} \big( \omega^A \delta^a_A + \tfrac{1}{2} \omega^0 u^a \big) \pi \ind*{_a^{\mbf{A}}} + \tfrac{\ii}{2\sqrt{2}} \omega^0 \pi^\mbf{A} , \\
\omega^\mbf{A} = \tfrac{1}{\sqrt{2}} \omega^a \pi \ind*{_a^{\mbf{A}}} , \qquad \text{where} \quad \omega^a := \big( \omega^A - \tfrac{1}{2} \omega^0 \pi^A \big) \delta \ind*{_A^a} + \omega^0 u^a ,
\end{gather*}
from which it is easy to check that $\pi^\mbf{A}$ and $\omega^\mbf{A}$ indeed satisfy the conditions given in Lem\-ma~\ref{lem-pure-ompi}.}

Finally, in the coordinate chart $(\C \E^n \times \mc{U}_0 , ( z^A,z_A, u ; \pi^A, \pi^{AB}))$, we have
\begin{gather*}
x^a \pi _a^\mbf{A} = \ii \big( u - \pi^B z _B \big) o^\mbf{A} + \big( z^B + \pi^{BC} z_C + \tfrac{1}{2} u \pi^B \big) \delta_B^\mbf{A} + \cdots ,
\end{gather*}
so that the incidence relation \eqref{eq-incidence_relation-spinor_odd} reduces to
\begin{gather}\label{eq-omega-coordinates-odd}
\omega^A = z^A + \pi^{AB} z_B + \tfrac{1}{2} \pi^A u , \qquad \omega^0 = u - \pi^B z_B .
\end{gather}

{\bf Tangent and cotangent spaces.} Let us introduce the short-hand notation
\begin{gather*}
\partial_{A} := \parderv{}{z^{A}} = \delta^a_A \nabla_a , \qquad \partial^{A} := \parderv{}{z_{A}} = \delta^{aA} \nabla_a , \qquad \partial := \parderv{}{u} = u^a \nabla_a ,
\end{gather*}
so that $\Tgt_{(x,\pi)} \mc{Q}^n \cong \mfp_{-1} = \langle \partial_{A} , \partial^{A} , \partial \rangle$, and def\/ine $1$-forms
\begin{gather}\label{eq-form-PT}
\bm{\alpha}^{A} := \dd \omega^{A} + \tfrac{1}{2} \pi^{A} \dd \omega^0 - \tfrac{1}{2} \omega^0 \dd \pi^{A} , \qquad \bm{\alpha}^{AB} := \dd \pi^{AB} - \pi^{[A} \dd \pi^{B]} ,
\end{gather}
and vectors
\begin{alignat}{3}
& \bm{X}_{A} := \parderv{}{\omega^{A}} , \qquad && \bm{X}_{AB} := \parderv{}{\pi^{AB}} ,& \nonumber\\
 & \bm{Y} := \parderv{}{\omega^0} - \frac{1}{2} \pi^{C} \parderv{}{\omega^C} , \qquad && \bm{Y}_{A} := \parderv{}{\pi^{A}} - \pi^{B} \parderv{}{\pi^{AB}} + \frac{1}{2} \omega^0 \parderv{}{\omega^A} .& \label{eq-vec-PT}
\end{alignat}
Then bases for the cotangent and tangent spaces of $\PT$ at $(\omega,\pi)$ are given by
\begin{gather*}
\Tgt^*_{(\omega,\pi)} \PT \cong \mfr^*_1 \oplus \mfr^*_2 = \big\langle \dd \omega^0, \dd \pi^A \big\rangle \oplus \big\langle \bm{\alpha}^{A} , \bm{\alpha}^{AB} \big\rangle , \\
 \Tgt_{(\omega,\pi)} \PT \cong \mfr_{-2} \oplus \mfr_{-1}
 = \big\langle \bm{X}_{A} \bm{X}_{AB} \big\rangle \oplus \big\langle \bm{Y} , \bm{Y}_{A} \big\rangle ,
\end{gather*}
respectively.

\begin{Remark}Using~\eqref{eq-hom2aff_odd}, one can check that the expressions for the set \eqref{eq-form-PT} of $\frac{1}{2}m(m+1)$ $1$-forms are none other than the $1$-forms \eqref{eq-hi-contact-ompi}, and thus \eqref{eq-hi-contact}. These forms annihilate the rank-$(m+1)$ canonical distribution $\mathrm{D}$ on $\PT$ spanned by $\bm{Y}$ and $\bm{Y}_{A}$. Further, the vector $\bm{Y}$ clearly coincides with~\eqref{eq-Infinity-Vec-field} to describe mini-twistor space -- this can be checked by using transformations \eqref{eq-hom2aff_odd}.
\end{Remark}

Now, def\/ine the $1$-forms and vectors
\begin{alignat*}{3}
& \bm{\theta}^A := \dd z^A + \big( \pi^{AD} - \tfrac{1}{2} \pi^{A} \pi^D \big) \dd z_D + \pi^A \dd u ,\qquad && \bm{\theta}^0 := \dd u - \pi^C \dd z_C , &\\
& \bm{Z}^{A} := \partial ^{A} + \big( \pi^{AD} - \tfrac{1}{2} \pi^{A} \pi^D \big) \partial _D + \pi^{A} \partial , \qquad &&
\bm{U} := \partial - \pi^D \partial _D ,& \\
 & \bm{W}_{A} := \parderv{}{\pi^{A}} - \pi^{B} \parderv{}{\pi^{AB}} .&&&
\end{alignat*}
Then bases for the cotangent and tangent spaces of $\F$ at $(x,\pi)$ are given by
\begin{subequations}
\begin{gather*}
\Tgt^*_{(x,\pi)} \F \cong {\mfq^*_1}^E \oplus {\mfq_1^*}^F \oplus {\mfq_2^*}^E \oplus {\mfq_2^*}^F \oplus \mfq_3^*
= \langle \dd z_A \rangle \oplus \big\langle \dd \pi^A \big\rangle \oplus \langle \bm{\theta}^0 \rangle \oplus \big\langle \bm{\alpha}^{AB} \big\rangle \oplus \big\langle \bm{\theta}^A \big\rangle , \\ 
\Tgt_{(x,\pi)} \F \cong \mfq_{-3} \oplus \mfq_{-2}^F \oplus \mfq_{-2}^E \oplus \mfq_{-1}^F \oplus \mfq_{-1}^E
= \langle \partial_{A} \rangle \oplus \langle \bm{X}_{AB} \rangle \oplus \langle \bm{U} \rangle \oplus \langle \bm{W}_{A} \rangle \oplus \big\langle \bm{Z}^{A} \big\rangle , 
\end{gather*}
\end{subequations}
respectively.

We note that the coordinates $(\omega^0,\omega^A,\pi^A,\pi^{AB})$ on $\mc{V}_0$ are indeed annihilated by the vectors $\bm{Z}^A$ tangent to the f\/ibres of $\F \rightarrow \PT$. Further, the pullback of $\bm{\alpha}^A$ to $\F$ is given by $\mu^* (\bm{\alpha}^A ) = \bm{\alpha}^{AB} z_B + \bm{\theta}^A$, i.e., the annihilator of $\Drm = \Tgt^{-1} \PT$ pulls back to the annihilator of~$\Tgt^{-2}_E \F$ corresponding to $\mfq_{-2}^E \oplus \mfq_{-1}^F \oplus \mfq_{-1}^E$.

{\bf Mini-twistor space.} By Lemma~\ref{lem-MT-flow}, the mini-twistor space $\M\T$ of $\C \E^n$ is the leaf space of~the vector f\/ield $\bm{Y}$ def\/ined by~\eqref{eq-Infinity-Vec-field}, given in~\eqref{eq-vec-PT} in the coordinate chart $(\mc{V}_0, (\omega^0, \omega^A, \pi^A$, $\pi^{AB}))$. Accordingly, we have a local coordinate chart $(\underline{\mc{V}}_0, ( \underline{\omega}^{A} , \pi^{AB} , \pi^{A} ))$ on $\M\T$ where
\begin{gather*}
\underline{\omega}^{A} = \omega^{A}+ \tfrac{1}{2} \pi^{A} \omega^0 ,
\end{gather*}
which can be seen to be annihilated by $\bm{Y}$. The incidence relation~\eqref{eq-incidence-co-null} or~\eqref{eq-incidence-co-null-2} can then be expressed as
\begin{gather*}
\underline{\omega}^{A} = z^{A} + \big( \pi^{AB} - \tfrac{1}{2} \pi^{A} \pi^{B} \big) z_{B}+ \pi^{A} u ,
\end{gather*}
which are indeed annihilated by $\bm{Z}^A$ and $\bm{U}$. The tangent space of $\M\T$ at a point $(\underline{\omega},\pi)$ in $\underline{\mc{V}}_0$ is clearly
\begin{gather*}
\Tgt_{(\underline{\omega},\pi)} \M\T = \langle \underline{\bm{X}}_A , \bm{X}_{AB} , \bm{W}_A \rangle , \qquad \text{where} \quad \underline{\bm{X}}_A := \parderv{}{\underline{\omega}^A} .
\end{gather*}

{\bf Normal bundle of $\hat{x}$ in $\PT_{\setminus \widehat{\infty}}$.} Let $x$ be a point in $\C \E^n$. In the chart $(\mc{V}_0, (\omega^0, \omega^A, \pi^A, \pi^{AB}))$, the corresponding $\hat{x}$ is given by~\eqref{eq-omega-coordinates-odd}. In particular, the $1$-forms
\begin{gather*}
\bm{\beta}^A (x) := \dd \omega^A - \dd \pi^{AB} z_B - \tfrac{1}{2} \dd \pi^A u , \qquad \bm{\beta}^0 (x) := \dd \omega^0 + \dd \pi^B z_B ,
\end{gather*}
vanish on $\hat{x}$, and the tangent space of $\hat{x}$ at $(\omega,\pi)$ is spanned by the vectors $\bm{Y}_A - z_A \bm{Y}$ and $\bm{X}_{AB} - z_{[A} \bm{X}_{B]}$. This distinguishes the $m$-dimensional subspace $\langle \bm{Y}_A - z_A \bm{Y}\rangle$ tangent to both $\hat{x}$ and the canonical distribution $\Drm$ at $(\omega,\pi)$.

\subsection{Even dimensions}\label{sec-affine-even}
The local description of $\F$ and $\PT$ in even dimensions can be easily derived from the one above.
We split $\V_0$ as $\V_0 \cong \mbb{W} \oplus \mbb{W}^*$ where $\mbb{W} \cong \C^m$ is a totally null $m$-plane of $(\V_0,g_{ab})$, with adapted basis $\{ \delta^{aA} , \delta^a_A \}$. The Fock representations of the irreducible spinor spaces $\Ss_{-\frac{1}{2}}$ and $\Ss'_{-\frac{1}{2}}$ on $\V_0$ are given by
\begin{gather*}
\Ss_{\frac{1}{2}} \cong \Ss'_{-\frac{1}{2}} \cong \wedge^m \mbb{W} \oplus \wedge^{m-2} \mbb{W} \oplus \cdots , \qquad
\Ss'_{\frac{1}{2}} \cong \Ss_{-\frac{1}{2}} \cong \wedge^{m-1} \mbb{W} \oplus \wedge^{m-3} \mbb{W} \oplus \cdots .
\end{gather*}
Let $o^{\mbf{A}'}$ be a (pure) spinor annihilating $\mbb{W}$. Then bases for $\Ss_{\frac{1}{2}}$ and $\Ss_{-\frac{1}{2}}$ can then be produced by acting on $o^{\mbf{A}'}$ by basis elements of $\wedge^{2k} \mbb{W}^*$ and of $\wedge^{2k-1} \mbb{W}^*$. Explicitly,
\begin{gather*}
\Ss_{\frac{1}{2}} = \big\langle o^{\mbf{A'}} , \delta^{\mbf{A'}}_{A_1A_2} , \ldots \big\rangle , \qquad
\Ss_{-\frac{1}{2}} = \big\langle \delta^\mbf{A}_{A_1} , \delta^\mbf{A}_{A_1 A_2 A_3} , \ldots \big\rangle ,
\end{gather*}
where
\begin{gather*}
\delta^{\mbf{A'}}_{A_1 \ldots A_{2k}} := \delta_{[A_1}^{a_1} \cdots \delta_{A_{2k}]}^{a_{2k}} o^{\mbf{A_0'}} \gamma \ind{_{a_1}_{\mbf{A_0'}}^{\mbf{A_1}}} \cdots \gamma \ind{_{a_{2k}}_{\mbf{A_{2k-1}}}^{\mbf{A'}}} , \\
\delta^{\mbf{A}}_{A_1 \ldots A_{2k-1}} := \delta_{[A_1}^{a_1} \cdots \delta_{A_{2k-1}]}^{a_{2k-1}} o^{\mbf{A_0'}} \gamma \ind{_{a_1}_{\mbf{A_0'}}^{\mbf{A_1}}} \cdots \gamma \ind{_{a_{2k-1}}_{\mbf{A_{2k-2}'}}^{\mbf{A}}} .
\end{gather*}
The Clif\/ford action of $\V_0 \subset \Cl(\V_0,g_{ab})$ on $\Ss_{\pm\frac{1}{2}}$ follows the same lines as \eqref{eq-Clifford-multiplication-odd} with appropriate priming of spinor indices.

Coordinate charts in even dimensions can be obtained from the odd-dimensional case by switching of\/f $\pi^{A_1 \ldots A_{k}}$ for all odd $k$, and $\omega^{A_1 \ldots A_{k}}$ for all even $k$. We therefore have a covering of each f\/ibre of $\F$ by $2^{m-1}$ open subsets $\mc{U}_0$, $\mc{U}_{A_1 \ldots A_{2k}}$, and a covering of $\PT_{\setminus \widehat{\infty}}$ by $2^{m-1}$ open subsets $\mc{V}_0$, $\mc{V}_{A_1 \ldots A_{2k}}$ in the obvious way. In particular, in $(\mc{V}_0, (\omega^A , \pi^{AB} ))$, the homogeneous coordinates $[\omega^\mbf{A},\pi^{\mbf{A}'}]$ are given by
\begin{gather*}
\omega^\mbf{A} = \tfrac{1}{\sqrt{2}} \big( \omega^A \delta_A^\mbf{A} - \tfrac{1}{4} \omega^A \pi^{BC} \delta_{ABC}^\mbf{A} + \cdots \big) , \qquad
\pi^{\mbf{A}'} = o^{\mbf{A'}} - \tfrac{1}{4} \pi^{AB} \delta_{AB}^{\mbf{A'}} + \cdots .
\end{gather*}
where the former can also be rewritten as $\omega^\mbf{A} = \frac{1}{\sqrt{2}} \omega^a \pi_a^\mbf{A}$ with $\omega^a := \omega^A \delta_A^a$. Finally, the even-dimensional version of the incidence relation~\eqref{eq-incidence_relation-spinor_odd} can be rewritten as $\omega^A = z^A + \pi^{AB} z_B$.

As for the tangent spaces of $\mc{Q}^{2m}$, its twistor space and their correspondence space, we f\/ind, in the obvious notation, $\Tgt_{(x,\pi)} \mc{Q}^n \cong \mfp_{-1} = \langle \partial_{A} , \partial^{A} , \partial \rangle$, $\Tgt_{(x,\pi)} \F \cong \mfq_{-2} \oplus \mfq_{-1}^F \oplus \mfq_{-1}^E = \langle \partial_{A} \rangle \oplus \langle \bm{X}_{AB} \rangle \oplus \langle \bm{Z}^{A} \rangle$, and $\Tgt_{(\omega,\pi)} \PT \cong \mfr_{-1} = \langle \bm{X}_{A} , \bm{X}_{AB} \rangle$, where $\bm{Z} ^{A} := \partial ^{A} + \pi^{AB} \partial _B$, $\bm{X}_{AB} := \parderv{}{\pi^{AB}}$, $\bm{X}_{A} := \parderv{}{\omega^{A}}$, and so on.

\subsection*{Acknowledgements}
The author would like to thank Boris Doubrov, Lionel Mason and Jan Slov\'{a}k for helpful discussions and comments, and the anonymous referees for their reports. He is also grateful to Luk\'{a}\v{s} Vok\v{r}\'{i}nek and Andreas \v{C}ap for clarifying some aspects of Section~\ref{sec-normal}. This work was funded by a GA\v{C}R (Czech Science Foundation) post-doctoral grant GP14-27885P.

\pdfbookmark[1]{References}{ref}
\LastPageEnding


\begin{thebibliography}{99}
\footnotesize\itemsep=0pt

\bibitem{Bailey1994}
Bailey T.N., Eastwood M.G., Gover A.R., Thomas's structure bundle for
 conformal, projective and related structures, \href{http://dx.doi.org/10.1216/rmjm/1181072333}{\textit{Rocky Mountain~J.
 Math.}} \textbf{24} (1994), 1191--1217.

\bibitem{Baird2013}
Baird P., Eastwood M., C{R} geometry and conformal foliations, \href{http://dx.doi.org/10.1007/s10455-012-9356-7}{\textit{Ann.
 Global Anal. Geom.}} \textbf{44} (2013), 73--90, \href{http://arxiv.org/abs/1011.4717}{arXiv:1011.4717}.

\bibitem{Baird1988}
Baird P., Wood J.C., Bernstein theorems for harmonic morphisms from~{${\bf
 R}^3$} and {$S^3$}, \href{http://dx.doi.org/10.1007/BF01450078}{\textit{Math. Ann.}} \textbf{280} (1988), 579--603.

\bibitem{Baird2003}
Baird P., Wood J.C., Harmonic morphisms between {R}iemannian manifolds,
 \href{http://dx.doi.org/10.1093/acprof:oso/9780198503620.001.0001}{\textit{London Mathematical Society Monographs. New Series}}, Vol.~29, The
 Clarendon Press, Oxford University Press, Oxford, 2003.

\bibitem{Baston1989}
Baston R.J., Eastwood M.G., The {P}enrose transform. Its interaction with
 representation theory, \textit{Oxford Mathematical Monographs}, The Clarendon Press,
 Oxford University Press, New York, 1989.

\bibitem{Baum2010}
Baum H., Juhl A., Conformal dif\/ferential geometry. $Q$-curvature and conformal
 holonomy, \href{http://dx.doi.org/10.1007/978-3-7643-9909-2}{\textit{Oberwolfach Seminars}}, Vol.~40, Birkh\"auser Verlag, Basel,
 2010.

\bibitem{Budinich1989}
Budinich P., Trautman A., Fock space description of simple spinors,
 \href{http://dx.doi.org/10.1063/1.528214}{\textit{J.~Math. Phys.}} \textbf{30} (1989), 2125--2131.

\bibitem{Calderbank2000}
Calderbank D.M.J., Pedersen H., Selfdual spaces with complex structures,
 {E}instein--{W}eyl geometry and geodesics, \href{http://dx.doi.org/10.5802/aif.1779}{\textit{Ann. Inst. Fourier
 (Grenoble)}} \textbf{50} (2000), 921--963, \href{http://arxiv.org/abs/math.DG/9911117}{math.DG/9911117}.


\bibitem{vCap2009}
\v{C}ap A., Slov\'ak J., Parabolic geometries. {I}.~Background and general
 theory, \href{http://dx.doi.org/10.1090/surv/154}{\textit{Mathematical Surveys and Monographs}}, Vol.~154, Amer. Math.
 Soc., Providence, RI, 2009.

\bibitem{Cartan1981}
Cartan E., The theory of spinors, Dover Publications, Inc., New York, 1981.

\bibitem{Cox1976}
Cox D., Flaherty Jr. E.J., A conventional proof of {K}err's theorem,
 \href{http://dx.doi.org/10.1007/BF01609355}{\textit{Comm. Math. Phys.}} \textbf{47} (1976), 75--79.

\bibitem{Curry2014}
Curry S., Gover A.R., An introduction to conformal geometry and tractor
 calculus, with a view to applications in general relativity,
 \href{http://arxiv.org/abs/1412.7559}{arXiv:1412.7559}.

\bibitem{Doubrov2010}
Doubrov B., Slov\'ak J., Inclusions between parabolic geometries, \href{http://dx.doi.org/10.4310/PAMQ.2010.v6.n3.a7}{\textit{Pure
 Appl. Math.~Q.}} \textbf{6} (2010), 755--780, \href{http://arxiv.org/abs/0807.3360}{arXiv:0807.3360}.

\bibitem{Eells1985}
Eells J., Salamon S., Twistorial construction of harmonic maps of surfaces into
 four-manifolds, \textit{Ann. Scuola Norm. Sup. Pisa Cl. Sci.~(4)} \textbf{12}
 (1985), 589--640.

\bibitem{Gover2008}
Gover A.R., \v{S}ilhan J., The conformal {K}illing equation on forms~--
 prolongations and applications, \href{http://dx.doi.org/10.1016/j.difgeo.2007.11.014}{\textit{Differential Geom. Appl.}} \textbf{26}
 (2008), 244--266, \href{http://arxiv.org/abs/math.DG/0601751}{math.DG/0601751}.

\bibitem{Harnad1992}
Harnad J., Shnider S., Isotropic geometry and twistors in higher dimensions.
 {I}.~{T}he generalized {K}lein correspondence and spinor f\/lags in even
 dimensions, \href{http://dx.doi.org/10.1063/1.529538}{\textit{J.~Math. Phys.}} \textbf{33} (1992), 3197--3208.

\bibitem{Harnad1995}
Harnad J., Shnider S., Isotropic geometry and twistors in higher dimensions.
 {II}.~{O}dd dimensions, reality conditions, and twistor superspaces,
 \href{http://dx.doi.org/10.1063/1.531096}{\textit{J.~Math. Phys.}} \textbf{36} (1995), 1945--1970.

\bibitem{Hitchin1982a}
Hitchin N.J., Complex manifolds and {E}instein's equations, in Twistor Geometry
 and Nonlinear Systems ({P}rimorsko, 1980), \href{http://dx.doi.org/10.1007/BFb0066025}{\textit{Lecture Notes in Math.}},
 Vol.~970, Springer, Berlin~-- New York, 1982, 73--99.

\bibitem{Hughston1983}
Hughston L.P., Hurd T.R., A {${\bf C}{\rm P}^{5}$} calculus for space-time
 f\/ields, \href{http://dx.doi.org/10.1016/0370-1573(83)90003-0}{\textit{Phys. Rep.}} \textbf{100} (1983), 273--326.

\bibitem{Hughston1988}
Hughston L.P., Mason L.J., A generalised {K}err--{R}obinson theorem,
 \href{http://dx.doi.org/10.1088/0264-9381/5/2/007}{\textit{Classical Quantum Gravity}} \textbf{5} (1988), 275--285.

\bibitem{Jones1985}
Jones P.E., Tod K.P., Minitwistor spaces and {E}instein--{W}eyl spaces,
 \href{http://dx.doi.org/10.1088/0264-9381/2/4/021}{\textit{Classical Quantum Gravity}} \textbf{2} (1985), 565--577.

\bibitem{Kerr1963}
Kerr R.P., Gravitational f\/ield of a spinning mass as an example of
 algebraically special metrics, \href{http://dx.doi.org/10.1103/PhysRevLett.11.237}{\textit{Phys. Rev. Lett.}} \textbf{11} (1963),
 237--238.

\bibitem{Kerr2009}
Kerr R.P., Schild A., Republication of: {A} new class of vacuum solutions of
 the {E}instein f\/ield equations, \href{http://dx.doi.org/10.1007/s10714-009-0857-z}{\textit{Gen. Relativity Gravitation}}
 \textbf{41} (2009), 2485--2499.

\bibitem{Kodaira1962}
Kodaira K., A theorem of completeness of characteristic systems for analytic
 families of compact submanifolds of complex manifolds, \href{http://dx.doi.org/10.2307/1970424}{\textit{Ann. of Math.}}
 \textbf{75} (1962), 146--162.

\bibitem{LeBrun1982}
LeBrun C.R., {${\cal H}$}-space with a cosmological constant, \href{http://dx.doi.org/10.1098/rspa.1982.0035}{\textit{Proc.
 Roy. Soc. London Ser.~A}} \textbf{380} (1982), 171--185.

\bibitem{Mason2010}
Mason L., Taghavi-Chabert A., Killing--{Y}ano tensors and multi-{H}ermitian
 structures, \href{http://dx.doi.org/10.1016/j.geomphys.2010.02.008}{\textit{J.~Geom. Phys.}} \textbf{60} (2010), 907--923,
 \href{http://arxiv.org/abs/0805.3756}{arXiv:0805.3756}.

\bibitem{Nurowski2010}
Nurowski P., Construction of conjugate functions, \href{http://dx.doi.org/10.1007/s10455-009-9186-4}{\textit{Ann. Global Anal.
 Geom.}} \textbf{37} (2010), 321--326, \href{http://arxiv.org/abs/math.DG/0605745}{math.DG/0605745}.

\bibitem{Onivsvcik1960}
Onishchik A.L., On compact {L}ie groups transitive on certain manifolds,
 \textit{Soviet Math. Dokl.} \textbf{1} (1960), 1288--1291.

\bibitem{Penrose1967}
Penrose R., Twistor algebra, \href{http://dx.doi.org/10.1063/1.1705200}{\textit{J.~Math. Phys.}} \textbf{8} (1967),
 345--366.

\bibitem{Penrose1976}
Penrose R., Nonlinear gravitons and curved twistor theory, \href{http://dx.doi.org/10.1007/BF00762011}{\textit{Gen.
 Relativity Gravitation}} \textbf{7} (1976), 31--52.

\bibitem{Penrose1986}
Penrose R., Rindler W., Spinors and space-time. {V}ol.~2.~Spinor and twistor
 methods in space-time geometry, \href{http://dx.doi.org/10.1017/CBO9780511524486}{\textit{Cambridge Monographs on Mathematical Physics}},
 Cambridge University Press, Cambridge, 1986.

\bibitem{Salamon2009}
Salamon S., Viaclovsky J., Orthogonal complex structures on domains
 in~{${\mathbb R}^4$}, \href{http://dx.doi.org/10.1007/s00208-008-0293-5}{\textit{Math. Ann.}} \textbf{343} (2009), 853--899,
 \href{http://arxiv.org/abs/0704.3422}{arXiv:0704.3422}.



\bibitem{Taghavi-Chabert2012}
Taghavi-Chabert A., The complex {G}oldberg--{S}achs theorem in higher
 dimensions, \href{http://dx.doi.org/10.1016/j.geomphys.2012.01.012}{\textit{J.~Geom. Phys.}} \textbf{62} (2012), 981--1012,
 \href{http://arxiv.org/abs/1107.2283}{arXiv:1107.2283}.

\bibitem{Taghavi-Chabert2016}
Taghavi-Chabert A., Pure spinors, intrinsic torsion and curvature in even
 dimensions, \href{http://dx.doi.org/10.1016/j.difgeo.2016.02.006}{\textit{Differential Geom. Appl.}} \textbf{46} (2016), 164--203,
 \href{http://arxiv.org/abs/1212.3595}{arXiv:1212.3595}.

\bibitem{Taghavi-Chabert2013}
Taghavi-Chabert A., Pure spinors, intrinsic torsion and curvature in odd
 dimensions, \href{http://arxiv.org/abs/1304.1076}{arXiv:1304.1076}.

\bibitem{Tod1995}
Tod K.P., Harmonic morphisms and mini-twistor space, in Further Advances in
 Twistor Theory. {V}ol.~{II}. Integrable Systems, Conformal Geometry and
 Gravitation, \textit{Pitman Research Notes in Mathematics Series}, Vol.~232,
 Editors L.J.~Mason, L.P.~Hughston, P.Z.~Kobak, Longman Scientif\/ic \&
 Technical, Harlow, 1995, 45--46.

\bibitem{Tod1995a}
Tod K.P., More on harmonic morphisms, in Further Advances in Twistor Theory.
 {V}ol.~{II}. Integrable Systems, Conformal Geometry and Gravitation,
 \textit{Pitman Research Notes in Mathematics Series}, Vol.~232, Editors~L.J.
 Mason, L.P.~Hughston, P.Z.~Kobak, Longman Scientif\/ic \& Technical, Harlow,
 1995, 47--48.

\end{thebibliography}
\end{document}